\documentclass[11pt,a4paper]{article}
\usepackage[utf8]{inputenc}
\usepackage[T1]{fontenc}
\usepackage{amsmath}
\usepackage{amscd}
\usepackage{amsfonts}
\usepackage{amsthm}
\usepackage{amssymb}
\usepackage{cite}
\usepackage{ctex}
\usepackage{color}
\usepackage{extarrows}
\usepackage{geometry}
\usepackage{graphicx}
\usepackage{mathrsfs}
\usepackage{indentfirst}
\usepackage{xcolor}
\usepackage[colorlinks,linkcolor=teal,citecolor=teal]{hyperref}
\newtheorem{cor}{\bf Corollary}[subsection]
\newtheorem{lemma}{\bf Lemma}[subsection]
\newtheorem{prop}{\bf Proposition}[subsection]
\newtheorem{thm}{\bf Theorem}[subsection]
\newtheorem{definition}{Definition}
\newtheorem{remark}{Remark}

\begin{document}
	\title{Basic (Dolbeault) Cohomology of Foliated Manifolds with Boundary}
	\author{Qingchun Ji \ \ Jun Yao}
	\date{}
	\maketitle
\renewcommand{\abstractname}{\bf Abstract}	
\begin{abstract}
	In this paper, we develop $L^2$ theory for Riemannian and Hermitian foliations on manifolds with basic boundary. We establish a decomposition theorem, various vanishing theorems, a twisted duality theorem for basic cohomologies and an extension theorem for basic forms of induced Riemannian foliation on the boundary. We prove the complex analogues for Hermitian foliations. To show the Dolbeault decomposition of basic forms, we extend Morrey's basic estimate to foliation version. We also investigate the global regularity for $\bar{\partial}_B$-equations.
\end{abstract}

\section{Introduction}
The theory of basic cohomology for foliations has attracted attention by many researchers. In general, the basic cohomology groups can be infinite-dimensional, and they may not satisfy Poincaré duality (see \cite{Cy84}). However, if $\left(M,\mathcal{F}\right)$ is a transversally oriented Riemannian foliation over a compact oriented smooth manifold, then there are plenty analytic and geometric properties like de Rham cohomologies on a compact Riemannian manifold. Due to this reason, it has been studied extensively by many researchers. The authors (see \cite{EN93}) proved that basic cohomology is a topological invariant of the foliation $ \mathcal{F} $. F. W. Kamber and P. Tondeur (see \cite{KT83D}) established the twisted duality theorem of basic cohomology. In \cite{EH86} and \cite{ESH85}, the authors proved the Hodge decomposition theorem by Molino's structure theorem for Riemannian foliations, and the finite-dimensionality of basic cohomology was deduced as a corollary. Then Kamber-Tondeur (see \cite{KT87}) explained basic Laplacian $\Delta_B:={\rm d}_B\delta_B+\delta_B{\rm d}_B$ (where $\delta_B$ is the formal adjoint of ${\rm d}_B$) as the restriction of a strongly elliptic operator on $M$, and eventually proved the Hodge decomposition of smooth basic forms under the assumption that the mean curvature form $\kappa\in\Omega_{B}^1\left(M\right)$. Later, J. Álvarez López in \cite{ALja92} showed that the arguments in \cite{KT87} still valid by using the basic component $\kappa_B$ to replace $\kappa$. In addition, D. Dom\'{i}nguez (see \cite{Dd98}) proved that for any Riemannian foliation $\mathcal{F}$ on a closed manifold, there is a bundle-like metric for $\mathcal{F}$ with basic mean curvature form, which also indicates that the assumption $\kappa\in\Omega_B^1\left(M\right)$ in \cite{KT87} is redundant. Moreover, E. Park and K. Richardson (see \cite{PR96}) proved that Hodge decomposition holds for arbitrary mean curvature by an innovative method. According to Hodge theorem, \cite{Hjj86} and \cite{MRT91} studied the relationships of positivity of curvature operator and vanishing of basic cohomology. Apart from decomposition of basic forms, in \cite{AT91} the authors proved leafwise Hodge decomposition. In addition, many researchers defined various new cohomologies for basic complex, they proved the topological invariance of new cohomologies, and discussed the conditions for the vanishing of these cohomologies (see \cite{HR13}, \cite{HR19} and \cite{OS16}).

Furthermore, we also investigate foliations that admit a transversely Hermitian structure. For such a Hermitian foliation, one can define the basic Dolbeault complex and its cohomology which is similar to basic de Rham cohomology. There are many progress on basic Dolbeault cohomology on closed manifold. The author in \cite{EKAa90} proved the basic Dolbeault decomposition theorem, which gives the finite-dimensionality of basic Dolbeault cohomology. In addition, if $\mathcal{F}$ is minimal, then Kodaira-Serre duality holds for basic Dolbeault cohomology. Moreover, if we assume that $\mathcal{F}$ is a transversely K{\"a}hler foliation, then this special structure appears naturally in many situations, and such structures play the same role as K{\"a}hler manifolds in general complex manifolds. For example, in \cite{CDY15} the authors proved hard Lefschetz theorem on compact Sasaki manifolds. El Kacimi-Gmira (see \cite{EG97}) studied the stability of transversely K{\"a}hler foliation. In \cite{JR21}, S. D. Jung and K. Richardson finded a new Weitzenb{\"o}ck formulas on a transversely K{\"a}hler foliation and investigated some curvature conditions which can impose restrictions on basic Dolbeault cohomology. The authors in \cite{CW91} proved the Fr{\"o}licher spectral sequence of basic Dolbeault cohomology collapses at the first level. Then S. D. Jung studied the twisted basic Dolbeault cohomology and transverse hard Lefschetz theorem on a transversely Kähler foliation in \cite{Jsd22}. Also, the authors in \cite{GNT16} established the rigidity and vanishing theorems of basic Dolbeault
cohomology of Sasakian manifolds.

In this paper, we focus on the transversally oriented Riemannian (Hermitian) foliation $\mathcal{F}$ over a compact oriented manifold $\overline{M}$ with basic mean curvature form and smooth basic boundary $\partial M$, i.e., the boundary defining function $\rho$ is basic. Roughly speaking, the definitions of basic forms and Riemannian (Hermitian) foliations are defined to ensure that basic (Dolbeault) cohomology theory behaves like de Rham (Dolbeault) theory on the leaf space, that's the reason why we require that $\rho$ is basic. Without loss of generality, we shall assume that $M$ is imbedded in a slightly larger open manifold $M'$. Here $\mathcal{F}$ over $\overline{M}$ means the restriction of Riemannian (Hermitian) foliation over $M'$.

Our first main result establishes the decomposition theorem:

\begin{thm}
	$(=$ {\rm Theorem} $\ref{Hodge decomposition theorem},\ref{Hodge decomposition theorem for dolbeault})$ For $ 0\leq r\leq q $, let $ \mathcal{F} $ be a transversally oriented Riemannian foliation on a compact oriented manifold $ \overline{M} $ with smooth basic boundary $ \partial M $, i.e., $ \rho $ is a basic function. Assume $ g_M $ is a bundle-like metric with basic mean curvature form. There is an $L^2$-orthogonal decomposition
	\begin{equation*}
		H_{B,0}^r={\rm Range}\left(\Delta_{F_B}\right)\oplus\mathcal{H}_B^r={\rm d}_B\delta_B{\rm Dom}\left(F_B\right)\oplus\delta_B{\rm d}_B{\rm Dom}\left(F_B\right)\oplus\mathcal{H}_B^r,
	\end{equation*}
	where the kernel $ \mathcal{H}_B^r $ of $ \Delta_{F_B} $ is a finite-dimensional space. Moreover, we have
	\begin{equation*}
		\Omega_B^r\left(\overline{M}\right)={\rm Im}{\rm d}_B\oplus{\rm Im}\delta_B\oplus\mathcal{H}_B^r.
	\end{equation*}
	
	If $\mathcal{F}$ is a Hermitian foliation, and we further assume that basic estimate holds in $\mathcal{D}_B^{\cdot,r}$, then the corresponding decomposition exists.
\end{thm}

We want to remark that the proof of the decomposition theorem for Hermitian foliations depends on the assumption that basic estimate holds in $\mathcal{D}_B^{\cdot,r}$ (see Definition \ref{befd}), while the assumption holds true automatically for Riemannian foliations since the boundary condition of $u\in\mathcal{D}_B^{\cdot,r}$ is more involved. By establishing the Bochner formula (\ref{Bochner formula for dolbeault}) we have the following geometric characterization for this assumption.

\begin{thm}
	$(=$ {\rm Theorem} $\ref{geometric condition for basic estimate for dolbeault})$ For $ 0\leq r\leq q $, basic estimate holds in $ \mathcal{D}_B^{\cdot,r} $ if and only if $ \partial M $ satisfies $ Z_r $-condition.
\end{thm}

Apart from proving above Theorem, we also give some vanishing theorems for basic cohomologies by the Bochner formulas (\ref{Bochner-Kodaira formula}) and (\ref{Bochner formula for dolbeault}) respectively. The most representative vanishing theorems are:

\begin{thm}
	$(=$ {\rm Theorem} $\ref{existence theorem 3},\ref{existence theorem 3 for dolbeault})$ For $ 1\leq r\leq q $, let $ M $ be an oriented manifold with a transversally oriented Riemannian foliation, and $ g_M $ is a bundle-like metric with basic mean curvature form. If there exists a strictly transversal $ r $-convex exhaustion basic function $ \varphi $, then for any $ \omega\in L^2_{loc}\left(M,\Omega_B^r\right) $ with $ {\rm d}_B\omega=0 $, there is a basic form $ u\in L^2_{loc}\left(M,\Omega_B^{r-1}\right) $ such that $ {\rm d}_Bu=\omega $.
	
	The $L^2$-existence of the equation ${\bar{\partial}_B}u=\omega$ is valid for Hermitian foliations under the corresponding convexity.
\end{thm}

In addition, we prove the following global regularity theorem for $\bar{\partial}_B$-equations.

\begin{thm}
	$(=$ {\rm Theorem} $\ref{grt.})$ For $ 1\leq r\leq q $, let $ M $ be an oriented Riemannian manifold with compact basic boundary $ \partial M $ and $\mathcal{F}$ is a transversally oriented Hermitian foliation. assume that there is an exhaustion basic function $ \varphi' $ satisfying $ Z_r $ and $Z_{r-1}$-condition outside a compact $ K\Subset M $. If the equation $ \bar{\partial}_Bu=\omega $ is solvable for $ \omega\in H_{B,0}^{\cdot,r} $, then the solution $ u\in\Omega_B^{\cdot,r-1}\left(\overline{M}\right) $ whenever $ \omega\in\Omega_B^{\cdot,r}\left(\overline{M}\right) $.
\end{thm}

Moreover, we establish the twisted duality theorem for basic cohomology.

\begin{thm}
	$(=$ {\rm Theorem} $\ref{twisted duality},\ref{twisted duality for dolbeault})$ For $0\leq r\leq q$, then we have
	$$H_B^r\left(\overline{M},\mathcal{C}\right)\cong\left(H_{B,\kappa}^{q-r}\left(\overline{M},\mathcal{F}\right)\right)^*$$
	for Riemannian foliations satisfying the conditions in Theorem \ref{Hodge decomposition theorem}.
	
	For the case of Hermitian foliations, the corresponding twisted duality holds provided that $\partial M$ satisfies $Z_{q-r}$-condition.
\end{thm}

An outline of this paper is as follows. In Chapter 2, we will discuss the Riemannian foliations: In the first section, we will review some formulae and facts concerning with basic forms and present some results of functional analysis preliminaries. In section 2, we give the Bochner formula for the purpose to show the vanishing theorems. section 3 is devoted to establishing the regularity theorem of the equation $F_Bu=\omega$ using the well-known general techniques for coercive boundary value problems (see \cite{ADN5964}).
From this result, we could obtain the de Rham-Hodge decomposition theorem for Riemannian foliations. Various vanishing theorems for basic cohomology will be established in section 4. In the last section, we introduce the induced boundary complex and then prove the twisted duality theorem for basic cohomology, which generalizes the duality theorem in \cite{KT83D}. As an application, we obtain the extension theorem of the induced boundary forms. In Chapter 3, we mainly concern the Hermitian foliations: In section 1, we will recall some definitions of Hermitian foliations and prove some results of functional analysis needed in the proof of basic Dolbeault decomposition. In section 2, we will give the decomposition theorem for tansversely Hermitian foliations under the assumption that basic estimate holds in $\mathcal{D}_B^{\cdot,r}$ by Kohn-Nirenberg's $L^2$-method (see \cite{FK72} and \cite{KN65N}). Section 3 is devoted to giving a geometric characteristic of basic estimate by the arguments originally owing to L. H{\"o}rmander (see \cite{Hlv65JL}). In section 4 and section 5, we will establish vanishing theorems and global regularity theorem of basic Dolbeault cohomology by weighted $L^2$-method. In section 6, we also introduce the induced boundary complex on the boundary $\partial M$ and obtain the extension theorem of the induced boundary forms in the sense of Kohn-Rossi in \cite{KR65}. In addition, we establish the twisted duality for basic Dolbeault cohomology.

Here and throughout this paper, the convention is adopted for summation over pairs of repeated indices. We fix the notations for indices as follows: $ 1\leq a,b,c,\cdots\leq p $, $ p+1\leq\alpha,\beta,\gamma,\cdots\leq n $ (with the additional allowance of $ 1\leq\alpha,\beta,\gamma,\cdots\leq q $) and $ 1\leq i,j,k\cdots\leq n $. We denote multi-indices by $ I,J,K\cdots $.
For any functions $ g $ and $ h $, we use the notation $ g\lesssim h $ or $ g=O\left(h\right) $ to signify the existence of a constant $ C>0 $ such that $ |g|\leq C|h| $.

\section{Riemannian Foliations}
\subsection{Preliminaries of Riemannian foliations}
In this section, we will review some facts and give some functional analysis preparations of Riemannian foliations. The ensuing results will be crucial for our subsequent investigations.

Let $ \left(M',g_{M'}\right) $ be a smooth $ n $ dimensional Riemannian manifold with a foliation $ \mathcal{F} $ of dimension $ p $. It's well-known that foliations are in one-to-one correspondence with integrable subbundles $ L\subseteq TM' $ by Frobenius' theorem. Hence the foliation $ \mathcal{F} $ can be expressed by the exact sequence of vector bundles:
\begin{equation}\label{exacts}
	0\longrightarrow L\longrightarrow TM'\stackrel{\pi}{\longrightarrow} Q\longrightarrow0,
\end{equation}
where $ Q $ is a quotient bundle of dimension $ q:=n-p $, $ \pi $ is a quotient map and we denoted by $\varpi$ the isomorphism from $Q$ to $L^\perp$ induced by the metric $g_{M'}$. The metric $g_{M'}$ on $TM'$ is then a direct sum
$$g_{M'}=g_L\oplus g_{L^\perp}.$$

\begin{definition}
	The foliation $\left(M',\mathcal{F}\right)$ is said to be a Riemannian foliation provided that there is a bundle-like metric on $ M' $, i.e., $L_V\varpi^*g_{L^\perp}=0$ for $V\in\Gamma L$, where $ L_V $ denotes Lie derivative along $ V $.
\end{definition}

The basic forms (introduced by R. L. Reinhart, see \cite{Rbl58,Rbl59F,Rbl59H}) of the foliation $ \mathcal{F} $ are given by
\begin{equation*}
	\Omega_B^\cdot\left(M'\right)\equiv\Omega_B^\cdot\left(M',\mathcal{F}\right)=\lbrace\omega\in\Omega^\cdot\left(M\right)\:|\:V\lrcorner\omega=L_V\omega=0,\ \forall V\in\Gamma L\rbrace,
\end{equation*}
where $ \lrcorner $ means the interior product. The exterior differential d preserves basic forms, since $ V\lrcorner {\rm d}\omega=L_V\omega-{\rm d}\left(V\lrcorner\omega\right)=0 $ and $ L_V{\rm d}\omega={\rm d}L_V\omega=0 $ for all $ V\in\Gamma L $. Hence, one can restrict d on general de Rham complex for $ \Omega^\cdot\left(M'\right) $ to obtain a subcomplex 
\begin{equation*}
	0\longrightarrow\Omega_B^1\left(M'\right)\stackrel{{\rm d}_B}{\longrightarrow}\cdots\stackrel{{\rm d}_B}{\longrightarrow}\Omega_B^r\left(M'\right)\stackrel{{\rm d}_B}{\longrightarrow}\Omega_B^{r+1}\left(M'\right)\stackrel{{\rm d}_B}{\longrightarrow}\cdots\stackrel{{\rm d}_B}{\longrightarrow}\Omega_B^q\left(M'\right)\longrightarrow0.
\end{equation*}
The cohomology
\begin{equation*}
	H^\cdot_B\left(M',\mathcal{F}\right)=H\left(\Omega_B^\cdot\left(M'\right),{\rm d}_B\right)
\end{equation*} 
is called the basic cohomology of $ \mathcal{F} $.

In what follows, we assume that $\left(M',g_{M'},\mathcal{F}\right)$ is a Riemannian foliation and we restict ourselves to the compact manifold $\overline{M}\subseteq M'$ with basic boundary $\partial M$. The notation $\Omega_B^\cdot\left(\overline{M}\right)$  is defined by
\begin{align*}
	\Omega_B^\cdot\left(\overline{M}\right)=\lbrace\omega\in\Omega_B^\cdot\left(M\right)\:|\:\omega\ \text{can be extended smoothly to $M'$}\rbrace.
\end{align*}

Let $ \nabla^M $ denote the Levi-Civita connection associated to the bundle-like Riemannian metric $ g_M:=g_{M'}|_{\overline{M}} $ on $ \overline{M} $. Then there is a canonical metric and torsion free connection $ \nabla $ for $ Y\in\Gamma Q $ defined by (see \cite{Mp71}, \cite{Tp88} and \cite{Tp97})
\begin{equation*}
	\nabla_XY=
 	\left\{
 	\begin{aligned}
 		&\pi\left[X,\varpi\left(Y\right)\right]\ &{\rm for}\ X\in\Gamma L,\\
 		&\pi\left(\nabla^M_X\varpi\left(Y\right)\right)\ &{\rm for}\ X\in\Gamma Q.
 	\end{aligned}
 	\right.
\end{equation*}

Let $ \Omega_Q^\cdot\left(\overline{M}\right) $ be sections of $ \Lambda^\cdot Q^* $ that can be extended smoothly to the boundary, and we denoted by $ \Omega_{Q,c}^\cdot\left(M\right) $ the set of elements of $ \Omega_Q^\cdot\left(\overline{M}\right) $ with compact support disjoint from $ \partial M $. The Sobolev norms on $ \Omega_Q^\cdot\left(\overline{M}\right) $ can be defined as
\begin{equation}\label{norms on Q}
	\|u\|_s^2=\sum_{k=0}^{s}\|\nabla^ku\|^2,\quad s\geq0.
\end{equation}
The global scalar product $ \left(\cdot,\cdot\right) $ for forms restricts on $ \Omega_B^\cdot\left(\overline{M}\right) $, gives rise to a natural scalar product denoted by $ \left(\cdot,\cdot\right)_B $. Then it allows us to define the Sobolev norms on $ \Omega_B^\cdot\left(\overline{M}\right) $ similarly as following
\begin{equation}\label{norms on B}
	\|u\|_{B,s}^2=\sum_{k=0}^{s}\|\nabla^ku\|_B^2,\quad s\geq0,
\end{equation}
where $ \|\cdot\|_B $ is associated to the inner product $ \left(\cdot,\cdot\right)_B $. In fact, we can define the Sobolev norms for any $ s\in\mathbb{R} $ on $ \Omega_Q^\cdot\left(\overline{M}\right) $ and $ \Omega_B^\cdot\left(\overline{M}\right) $ as general definitions of Sobolev norms on manifolds (see \cite{Tme11}).

In particular, it's more convenient in applications for us to give equivalent norms on $ \Omega_B^\cdot\left(\overline{M}\right) $ as follows: Let $ \lbrace\left(x^l,U_l\right)\rbrace_{l=1}^L $ be local coordinate charts, $ \lbrace\eta_l\rbrace_{l=1}^L $ be a $ C^\infty $ partition of unity subordinate to $ \lbrace U_l\rbrace_{l=1}^L $, and let $ D_l^I:=\left(\frac{\partial}{\partial x^l}\right)^I $ act on the components of forms in $ U_l $. Then for any $ u\in\Omega_B^\cdot\left(\overline{M}\right) $, we define
\begin{equation}\label{norm}
	'\|u\|_{B,s}^{2}=\sum_{l=1}^{L}\sum_{|I|\leq s}\|D_l^I\eta_lu\|^2,\quad s\geq0.
\end{equation}
It's easy to prove that $ '\|\cdot\|_{B,s}\thicksim\|\cdot\|_{B,s} $. Thus we will make use of both $ \|\cdot\|_{B,s} $ and $ '\|\cdot\|_{B,s} $ in $ H_{B,s}^\cdot $. For the sake of simplicity, we may omit the superscripts of $ x^l $ and subscripts of $ D_l^I $ for all $ 1\leq l\leq L $ from now on.

We therefore obtain Sobolev spaces $ H_{B,s}^\cdot $ and $ H_s^\cdot $ defined as completions of $ \Omega_B^\cdot\left(\overline{M}\right) $ and $ \Omega_Q^\cdot\left(\overline{M}\right) $ on $ \overline{M} $ w.r.t. the norms $ \|\cdot\|_{B,s} $ and $ \|\cdot\|_s $ respectively. It's evident that $ H_{B,s}^\cdot $ are topologically equivalent to the closure of $ \Omega_B^\cdot\left(\overline{M}\right) $ in $ H_s^\cdot $ with respect to $ \|\cdot\|_s $ for all $ s\geq0 $. Kamber-Tondeur (see \cite{KT87}) proved the Rellich and Sobolev properties for the Sobolev chain $ \lbrace H_{B,s}^\cdot\rbrace_{s\geq0} $ as stated in the following proposition.

\begin{prop}\label{Rellich and Sobolev lemma}
	If $ s>t\geq0 $, then the inclusions $ i:H_{B,s}^\cdot\longrightarrow H_{B,t}^\cdot $ are compact and dense. If s,t are positive integers, and $ u\in H_{B,s}^\cdot\cap H_{s+t}^\cdot $, then for $ t>t_0=\left[n/2\right]+1 $, we have $ u\in\Omega_{B,s+t-t_0}^\cdot\left(\overline{M}\right) $, where $ \Omega_{B,s+t-t_0}^\cdot\left(\overline{M}\right)  $ is the subspace consists of elements in $ H_{B,s+t-t_0}^\cdot $ with continuous partial derivartives to order $ s+t-t_0 $.
\end{prop}

Let $ \left(x=\left(x',x''\right),U\right) $ be a local distinguished coordinate (i.e., for any leaf $\mathcal{L}$, the connected components of $\mathcal{L}\cap U$ can be defined by the equations $x_{p+1}=const.,\cdots,x_n=const.$) of $ M $, where $ x'=\left(x_1,\cdots,x_p\right) $ and $ x''=\left(x_{p+1} ,\dots,x_n\right) $. We choose local frames $ X_a=\frac{\partial}{\partial x_a} $, $ X_\alpha=\frac{\partial}{\partial x_\alpha}+\phi_\alpha^a\left(x\right)\frac{\partial}{\partial x_a} $ where the functions $ \phi_\alpha^a\left(x\right) $ are chosen such that
$$ X_\alpha\perp X_a,\quad\forall a,\alpha, $$
with respect to bundle like metric $ g_M $, moreover,
$$ L|_U\ \text{is spanned by the vector fields}\  X_1,\dots,X_p,\ \text{and}\  Q|_U\xrightarrow[\varpi]{\backsimeq}{\rm span}\lbrace X_\alpha\rbrace_{\alpha=p+1}^n, $$
recalling $L$ is the subbundle associated to the foliation $\mathcal{F}$. It's obvious that $ \lbrace\theta^a:={\rm d}x^a-\phi_\alpha^a{\rm d}x^\alpha\rbrace_{a=1}^p $ and $ \lbrace\theta^\alpha:={\rm d}x^\alpha\rbrace_{\alpha=p+1}^n $ are frames over $ U $ for the dual bundles $ L^* $ and $ Q^* $ respectively. In this chart, the bundle-like metric could be written as
\begin{equation*}
	{\rm d}s^2=\sum_{a,b=1}^pg_{ab}\left(x\right)\theta^a\theta^b+\sum_{\alpha,\beta=1}^qg_{\alpha\beta}\left(x''\right){\rm d}x_\alpha{\rm d}x_\beta.
\end{equation*}

Let $ \Gamma_{ij}^k $ be the coefficients of the connection $ \nabla^M $, i.e., $ \nabla^M_{X_i}X_j=\Gamma_{ij}^kX_k $. 
We denote $ g_{ij}=g_M\left(X_i,X_j\right) $ and $ g^{ij}=g_M\left(\theta^i,\theta^j\right) $. It's clear that the transversal metric $ g_Q:=\varpi^*g_{L^\perp} $ completely
determines the connection $ \nabla $, i.e., $ \nabla_{X_\alpha}X_\beta=\Gamma_{\alpha\beta}^{\gamma}X_\gamma $. For later use, we need the following relationships of the connection coefficients of $ \nabla^M $.
\begin{prop}\label{coefficients of conection}
	$ \Gamma_{ab}^i=\Gamma_{ba}^i,\ \Gamma_{a \alpha}^\beta=\Gamma_{\alpha a}^\beta,\ \Gamma_{\alpha\beta}^\gamma=\Gamma_{\beta\alpha}^\gamma,\ \Gamma_{\alpha\beta}^a=-\Gamma_{\beta\alpha}^a $ and
	\begin{align}\label{expression of some Gamma}
		\Gamma_{ab}^\alpha=-\frac{1}{2}g^{\alpha\beta}\left(X_\beta g_{ab}+g_{ac}\frac{\partial\phi_\beta^c}{\partial x_b}+g_{bc}\frac{\partial\phi_\beta^c}{\partial x_a}\right).
	\end{align}
\end{prop}

\renewcommand{\proofname}{\bf $ Proof $}
\begin{proof}
	The first term is clear from
	\begin{equation}\label{coe. 1}
		0=\left[X_a,X_b\right]=\nabla^M_{X_a}X_b-\nabla^M_{X_b}X_a=\left(\Gamma_{ab}^i-\Gamma_{ba}^i\right)X_i.
	\end{equation}
    
    For the second identity, we know by the definition
    \begin{align}\label{coe. 2'}
    	\left[X_a,X_\alpha\right]=\left[\frac{\partial}{\partial x_a},\phi_\alpha^b\frac{\partial}{\partial x_b}\right]=\frac{\partial\phi_\alpha^b}{\partial x_a}X_b.
    \end{align}
    On the other hand,
    \begin{equation*}
    	\left[X_a,X_\alpha\right]=\nabla^M_{X_a}X_\alpha-\nabla^M_{X_\alpha}X_a=\left(\Gamma_{a \alpha}^b-\Gamma_{\alpha a}^b\right)X_b+\left(\Gamma_{a \alpha}^\beta-\Gamma_{\alpha a}^\beta\right)X_\beta.
    \end{equation*}
    Thus, $ \Gamma_{a \alpha}^\beta=\Gamma_{\alpha a}^\beta $ and
    \begin{align}\label{coe. 2}
    	\Gamma_{a \alpha}^b-\Gamma_{\alpha a}^b=\frac{\partial\phi_\alpha^b}{\partial x_a}.
    \end{align}
    
    For the third one, since $ \left[X_\alpha,X_\beta\right]=\left(\Gamma_{\alpha\beta}^a-\Gamma_{\beta\alpha}^a\right)X_a+\left(\Gamma_{\alpha\beta}^\gamma-\Gamma_{\beta\alpha}^\gamma\right)X_\gamma $ and
    \begin{align*}
    	\left[X_\alpha,X_\beta\right]&=\left[\frac{\partial}{\partial x^\alpha},\phi_\beta^b\frac{\partial}{\partial x^b}\right]+\left[\phi_\alpha^a\frac{\partial}{\partial x^a},\frac{\partial}{\partial x^\beta}\right]+\left[\phi_\alpha^a\frac{\partial}{\partial x^a},\phi_\beta^b\frac{\partial}{\partial x^b}\right]\\
    	&=\left(\frac{\partial\phi_\beta^a}{\partial x^\alpha}-\frac{\partial\phi_\alpha^a}{\partial x^\beta}+\phi_\alpha^b\frac{\partial\phi_\beta^a}{\partial x^b}-\phi_\beta^b\frac{\partial\phi_\alpha^a}{\partial x^b}\right)X_a,
    \end{align*}
    it follows that $ \Gamma_{\alpha\beta}^\gamma=\Gamma_{\beta\alpha}^\gamma $, $ \Gamma_{\alpha\beta}^a-\Gamma_{\beta\alpha}^a=\frac{\partial\phi_\beta^a}{\partial x^\alpha}-\frac{\partial\phi_\alpha^a}{\partial x^\beta}+\phi_\alpha^b\frac{\partial\phi_\beta^a}{\partial x^b}-\phi_\beta^b\frac{\partial\phi_\alpha^a}{\partial x^b} $.
    
    For the fourth relationship, by $ (\ref{coe. 2'}) $ we have
    \begin{align*}
    	0&=X_a\left<X_\alpha,X_\beta\right>\\
    	&=\left<\nabla_{X_a}^MX_\alpha,X_\beta\right>+\left<X_\alpha,\nabla_{X_a}^MX_\beta\right>\\
    	&=\left<\nabla_{X_\alpha}^MX_a+\left[X_a,X_\alpha\right],X_\beta\right>+\left<X_\alpha,\nabla_{X_\beta}^MX_a+\left[X_a,X_\beta\right]\right>\\
    	&=\left<\nabla_{X_\alpha}^MX_a,X_\beta\right>+\left<X_\alpha,\nabla_{X_\beta}^MX_a\right>\\
    	&=-\left(\Gamma_{\alpha\beta}^b+\Gamma_{\beta\alpha}^b\right)g_{ab}.
    \end{align*}
    Then we multiply both sides of above equality by $ g^{ac} $ and summing over the index $ a $, it turns out that $ \Gamma_{\alpha\beta}^c+\Gamma_{\beta\alpha}^c=0 $.
    
    Although the proof of the expression $ (\ref{expression of some Gamma}) $ is gory, we will present all the details in the following for clarity.. The compatibility of connection $\nabla^M$ and metrci $g_M$ gives
    \begin{align}
    	X_\alpha g_{ac}&=X_\alpha\left<X_a,X_c\right>=\Gamma_{\alpha a}^bg_{bc}+\Gamma_{\alpha c}^bg_{ab},\label{coe. 3}\\
    	0&=X_a\left<X_c,X_\alpha\right>=\Gamma_{ac}^\beta g_{\beta\alpha}+\Gamma_{a\alpha}^bg_{cb},\label{coe. 4}\\
    	0&=X_c\left<X_\alpha,X_a\right>=\Gamma_{c\alpha}^bg_{ba}+\Gamma_{ca}^\beta g_{\alpha\beta}.\label{coe. 5}
    \end{align}
    Calculating $ (\ref{coe. 3})+(\ref{coe. 4})-(\ref{coe. 5}) $ and making use of $ (\ref{coe. 1}) $ we obtain
    \begin{align*}
    	X_\alpha g_{ac}=\left(\Gamma_{\alpha a}^b+\Gamma_{a\alpha}^b\right)g_{bc}+\left(\Gamma_{\alpha c}^b-\Gamma_{c\alpha}^b\right)g_{ab}.
    \end{align*}
    Then by $ (\ref{coe. 2}) $ we have
    \begin{align}\label{coe. 6}
    	\Gamma_{\alpha a}^d+\Gamma_{a\alpha}^d=g^{cd}X_\alpha g_{ac}+g^{cd}g_{ab}\frac{\partial\phi_\alpha^b}{\partial x_c}.
    \end{align}
    Again by the equalities $ (\ref{coe. 2}) $ and $ (\ref{coe. 6}) $ we can solve the expression of $ \Gamma_{a\alpha}^b $ as follows
    \begin{align}\label{coe. 7}
    	\Gamma_{a\alpha}^b=\frac{1}{2}\left(g^{bc}X_\alpha g_{ac}+g^{bc}g_{ad}\frac{\partial\phi_\alpha^d}{\partial x_c}+\frac{\partial\phi_\alpha^b}{\partial x_a}\right).
    \end{align}
   Substituting the expression $ (\ref{coe. 7}) $ into the equality $ (\ref{coe. 4}) $ implies the conclusion.
\end{proof}

The difference of the metric connection $ \nabla $ and the Levi Civita connection $ \nabla^M $ is giving by the following proposition.
\begin{prop}\label{difference of connection}
		Let $ 0\leq r\leq q $, then $ \nabla^M_{X_\alpha}w=\nabla_{X_\alpha}w-\Gamma_{\alpha a}^{\beta}\theta^a\wedge X_\beta\lrcorner w $ for all $ w\in\Omega_Q^r\left(M\right) $.
\end{prop}

\begin{proof}
	Let $ w=w_I\theta^I\in\Omega_Q^r\left(M\right) $, we have
	\begin{align*}
		\nabla_{X_\alpha}^Mw=&X_{\alpha}\left(w_I\right)\theta^I+w_I\nabla^M_{X_\alpha}\theta^I\\
		=&X_{\alpha}\left(w_I\right)\theta^I-w_I\Gamma_{\alpha i}^\beta\theta^i\wedge X_{\beta}\lrcorner\theta^I\\
		=&X_{\alpha}\left(w_I\right)\theta^I-w_I\Gamma_{\alpha \gamma}^\beta\theta^\gamma\wedge X_{\beta}\lrcorner\theta^I-w_I\Gamma_{\alpha a}^\beta\theta^a\wedge X_{\beta}\lrcorner\theta^I\\
		=&\nabla_{X_\alpha}w-\Gamma_{\alpha a}^{\beta}\theta^a\wedge X_\beta\lrcorner w.
	\end{align*}
    The proof is thus complete.
\end{proof}

Next we express $ {\rm d}_B $ in terms of the connection $\nabla$.
\begin{prop}\label{the local expression of d_B}
	Let $ 0\leq r\leq q $, if $ u\in\Omega_B^r\left(M\right) $, then $ {\rm d}_Bu=\theta^\alpha\wedge\nabla_{X_\alpha}u $.
\end{prop}

\begin{proof}
	Let $ u=u_I\theta^I\in\Omega_B^r\left(M\right) $, then 
	\begin{align*}
		{\rm d}u=&\theta^i\wedge\nabla^M_{X_i}u\\
		=&\theta^a\wedge\nabla_{X_a}^Mu+\theta^\alpha\wedge\nabla_{X_\alpha}^Mu\\
		=&\theta^a\wedge X_a\left(u_I\right)\theta^I+\theta^a\wedge u_I\nabla_{X_a}^M\theta^I+\theta^\alpha\wedge\nabla_{X_\alpha}^Mu\\
		=&-\theta^a\wedge\Gamma_{aj}^\alpha\theta^j\wedge X_\alpha\lrcorner u+\theta^\alpha\wedge\nabla_{X_\alpha}u-\theta^\alpha\wedge\Gamma_{\alpha a}^\beta\theta^a\wedge X_\beta\lrcorner u\\
		=&\theta^\alpha\wedge\nabla_{X_\alpha}u.
	\end{align*}
	The validity of the second to the last equality arises from the property of basic forms and Proposition $ \ref{difference of connection} $, and the last line holds true due to Proposition $ \ref{coefficients of conection} $.
\end{proof}

Now, we calculate the local expression of the formal adjoint $ \delta_B $ of $ {\rm d}_B $ with repect to inner product $ \left(\cdot,\cdot\right)_B $. Without loss of generality, we assume that the boundary defining function $\rho$ satisfies $ \rho<0 $ inside $ M $, $ \rho>0 $ outside $ \overline{M} $ and $ |{\rm d}\rho|=1 $ on $ \partial M $.
\begin{prop}\label{the reprensentation of formal adjoint}
	Let $ 1\leq r\leq q $, for all $ u\in\Omega_B^{r-1}\left(\overline{M}\right) $, $ v\in\Omega_B^r\left(\overline{M}\right) $,
	\begin{equation}\label{divergence formula}
		\left({\rm d}_Bu,v\right)_B=\left(u,\delta_Bv\right)_B+\int_{\partial M}\left<u,{\rm grad}\rho\lrcorner v\right>,
	\end{equation}
    where $ \delta_Bv=-g^{\alpha\beta}X_\alpha\lrcorner\nabla_{X_\beta}v+\tau\lrcorner v $, and $ \tau:=\pi\left(g^{ab}\nabla^M_{X_a}X_b\right) $ is the mean curvature vector field of the foliation $ \mathcal{F} $. In particular,
    \begin{align*}
    	\delta v=\delta_Bv+S\left(v\right),
    \end{align*}
    where $\delta$ is the formal adjoint operator of ${\rm d}$ w.r.t. $\left(\cdot,\cdot\right)$ and $S$ is a zero order operator on $\Omega_B^0\left(\overline{M}\right)$.
\end{prop}

\begin{proof}
	Since $\nabla$ is metric connection, we know by Proposition \ref{the local expression of d_B}
	\begin{align*}
		\left<{\rm d}_Bu,v\right>=&g^{\alpha\beta}\left<\nabla_{X_\alpha}u,X_\beta\lrcorner v\right>\\
		=&g^{\alpha\beta}X_\alpha\left<u,X_\beta\lrcorner v\right>-g^{\alpha\beta}\left<u,\nabla_{X_\alpha}\left(X_\beta\lrcorner v\right)\right>\\
		=&{\rm div}^M\left(g^{\alpha\beta}\left<u,X_\beta\lrcorner v\right>X_\alpha\right)-X_\alpha\left(g^{\alpha\beta}\right)\left<u,X_\beta\lrcorner v\right>-g^{\alpha\beta}\left<u,X_\beta\lrcorner v\right>{\rm div}^MX_\alpha-g^{\alpha\beta}\left<u,\nabla_{X_\alpha}\left(X_\beta\lrcorner v\right)\right>\\	
		=&{\rm div}^M\left(g^{\alpha\beta}\left<u,X_\beta\lrcorner v\right>X_\alpha\right)+\left(\Gamma_{\alpha\gamma}^\alpha g^{\beta\gamma}+\Gamma_{\alpha\gamma}^\beta g^{\alpha\gamma}\right)\left<u,X_\beta\lrcorner v\right>+g^{\alpha\beta}\left<u,\left<g^{ab}\nabla_{X_a}^MX_b,X_\alpha\right>X_\beta\lrcorner v\right>\\
		&-g^{\alpha\beta}\Gamma_{\gamma\alpha}^{\gamma}\left<u,X_\beta\lrcorner v\right>-g^{\alpha\beta}\left<u,\Gamma_{\alpha\beta}^\gamma X_\gamma\lrcorner v\right>-\left<u,g^{\alpha\beta}X_\beta\lrcorner\nabla_{X_\alpha}v\right>\\
		=&{\rm div}^M\left(g^{\alpha\beta}\left<u,X_\beta\lrcorner v\right>X_\alpha\right)-\left<u,g^{\alpha\beta}X_\beta\lrcorner\nabla_{X_\alpha}v\right>+\left<u,\tau\lrcorner v\right>\\
		=&:{\rm div}^M\left(g^{\alpha\beta}\left<u,X_\beta\lrcorner v\right>X_\alpha\right)+\left<u,\delta_Bv\right>.
	\end{align*}
    Integrating both sides over M and using the divergence theorem, it follows that
    \begin{align*}
    	\left({\rm d}_Bu,v\right)_B=\int_{M}\left<{\rm d}_Bu,v\right>=&\int_M{\rm div}^M\left(g^{\alpha\beta}\left<u,X_\beta\lrcorner v\right>X_\alpha\right)+\int_M\left<u,\delta_Bv\right>\\
    	=&\int_{\partial M}g^{\alpha\beta}\left<u,X_\beta\lrcorner v\right>X_\alpha\rho+\int_M\left<u,\delta_Bv\right>\\
    	=&\left(u,\delta_Bv\right)_B+\int_{\partial M}\left<u,{\rm grad}\rho\lrcorner v\right>.
    \end{align*}
    
    For the second conclusion, according to Proposition \ref{difference of connection} we have
    \begin{align*}
    	\delta v:=&-g^{ij}X_i\lrcorner\nabla_{X_j}^Mv\\
    	=&-g^{\alpha\beta}X_\alpha\lrcorner\nabla_{X_\beta}v+g^{\alpha\beta}X_\alpha\lrcorner\left(\nabla_{X_\beta}-\nabla_{X_\beta}^M\right)v-g^{ab}X_a\lrcorner\nabla_{X_b}^Mv\\
    	=&-g^{\alpha\beta}X_\alpha\lrcorner\nabla_{X_\beta}v+\tau\lrcorner v-g^{\alpha\beta}\Gamma_{\alpha a}^{\gamma}\theta^a\wedge X_\beta\lrcorner X_\gamma\lrcorner v\\
    	=&:\delta_Bv+S\left(v\right).
    \end{align*}
    This the proof is complete.
\end{proof}

\begin{remark}
	We require that the dual form $\kappa$ of the mean curvature vector field $\tau$ of $\mathcal{F}$ is a basic form, it means that $\tau$ is projectable $($i.e., for all vector fields $V\in\Gamma L$ the Lie bracket $\left[V,\tau\right]\in\Gamma L)$, thus $\delta_B$ is well defined. 
\end{remark}

As an application, we find the difference of general Laplacian $ \Delta:={\rm d}\delta+\delta {\rm d} $ and basic Laplacian $ \Delta_B:={\rm d}_B\delta_B+\delta_B{\rm d}_B $, which is already known in \cite{KT87}.
\begin{prop}\label{difference of two Laplacian}
	Let $ 0\leq r\leq q $, for $ u\in\Omega_B^r\left(\overline{M}\right) $,
	\begin{equation*}
		\Delta u=\Delta_Bu+S'\left(u\right),
	\end{equation*}
    where $ S'\left(u\right):={\rm d}S\left(u\right)+S\left({\rm d}u\right) $. Thus $ \Delta_B $ is the restriction of a strongly elliptic operator.
\end{prop}

\begin{proof}
	For $ u\in\Omega_B^r\left(\overline{M}\right) $, then owing to Proposition \ref{the reprensentation of formal adjoint},
	\begin{equation*}
		\Delta u=\left({\rm d}\delta+\delta {\rm d}\right)u={\rm d}\left(\delta_B+S\right)u+\left(\delta_B+S\right){\rm d}_Bu=\Delta_Bu+{\rm d}S\left(u\right)+S\left({\rm d}u\right).
	\end{equation*}
    Since $ S':={\rm d}S+S{\rm d} $ is of first order, the above identity implies that $ \Delta_B $ is the restriction of a strongly elliptic operator $ \Delta-\tilde{S}' $ on $ \Omega_B^r\left(\overline{M}\right) $, where $ \tilde{S}' $ is the extension of $ S' $.
\end{proof}

To be used later, we introduce the operators
\begin{align}\label{Q-operator}
	{\rm d}_Q:=\theta^\alpha\wedge\nabla_{X_\alpha},\quad\delta_Q:=-g^{\alpha\beta}X_\alpha\lrcorner\nabla_{X_\beta}+\tau\lrcorner
\end{align}
and the transversal Laplacian $ \Delta_Q:={\rm d}_Q\delta_Q+\delta_Q{\rm d}_Q $ which are all well-defined on $ \Omega_Q^\cdot\left(M\right) $. As the $\Omega_B^\cdot\left(M\right)$ is preserved by $ {\rm d}_Q $ and $ \delta_Q $ we have $ {\rm d}_Q|_{\Omega_B^\cdot\left(M\right)}={\rm d}_B $ and $ \delta_Q|_{\Omega_B^\cdot\left(M\right)}=\delta_B $. The formula ($ \ref{divergence formula} $) for operators $ {\rm d}_Q $ and $ \delta_Q $ also holds on $ \Omega_Q^\cdot\left(\overline{M}\right) $ by the same calculation in the proof of Proposition $ \ref{the reprensentation of formal adjoint} $. 

We need to extend the domain of basic exterior differential $ {\rm d}_B $ as follows.
\begin{prop}
	The closure of $ {\rm d}_B|_{\Omega_B^r\left(\overline{M}\right)} $ with respect to $ H_{B,0}^r $ $($i.e., the operator whose graph
	is the closure of the graph of $ {\rm d}_B|_{\Omega_B^r\left(\overline{M}\right)} $ in $H_{B,0}^r \times H_{B,0}^{r+1})$ is well-defined for any $ 0\leq r\leq q $, which we shall use the same symbol $ {\rm d}_B $ to denote its closure.
\end{prop}

\begin{proof}
	If $ \{u_\nu\}_{\nu=1}^\infty\subseteq\Omega_B^r\left(\overline{M}\right) $ such that $ u_\nu\stackrel{\nu\rightarrow\infty}{\longrightarrow}0 $ and $ {\rm d}_Bu_\nu\stackrel{\nu\rightarrow\infty}{\longrightarrow}v $ in $ H_{B,0}^r $, then for all $ w\in\Omega_{Q,c}^r\left(M\right) $ we have
	\begin{equation*}
		\left(v,w\right)=\lim_{n\rightarrow\infty}\left({\rm d}_Bu_n,w\right)=\lim_{n\rightarrow\infty}\left({\rm d}u_n,w\right)=\lim_{n\rightarrow\infty}\left(u_n,\delta w\right)=0.
	\end{equation*} 
	It means that $ v=0 $, which turns out that the closure of $ {\rm d}_B $ is well-defined.
\end{proof}

For $ 0\leq r\leq q $, we denote by $ {\rm d}_B^* $ the Hilbert space adjoint of $ {\rm d}_B $ on $ H_{B,0}^r $, and define 
$$ \mathcal{D}_B^r={\rm Dom}\left({\rm d}_B^*\right)\cap\Omega_B^r\left(\overline{M}\right). $$
Henceforth, we always require that boundary defining function $ \rho $ is basic. It's apparent that $\partial M$ is a saturated set on $M'$ (i.e., if $p\in\partial M$, then $\partial M$ contains
the leaf passing through $p$), provided that $\rho$ is a basic function. Furthermore, we have the following:

\begin{prop}
	If $M'$ has a complete Riemannian metric which is bundle-like with respect to $\mathcal{F}$ and all leaves of $\mathcal{F}$ are closed, then $\rho$ is basic when $\partial M$ is saturated on $M'$.
\end{prop}

\begin{proof}
	Since $\partial M$ is a saturated set, $\partial M=\cup_\alpha\mathcal{L}_\alpha$ where $\mathcal{L}_\alpha$ is the leaf of $\mathcal{F}$. For any point $p\in M$, we define a function $\rho\left(p\right):=\inf_\alpha {\rm dist}\left(p,\mathcal{L}_\alpha\right)$ where ${\rm dist}\left(p,\mathcal{L}_\alpha\right)$ is the greatest lower bound of the length of all paths joining $p$ to $\mathcal{L}_\alpha$. According to Lemma 4.3 in \cite{Hr60}, we have ${\rm dist}\left(\mathcal{L}_p,\mathcal{L}_\alpha\right)={\rm dist}\left(p,\mathcal{L}_\alpha\right)$ which gives $\rho\left(\mathcal{L}_p\right)=\rho\left(p\right)$. It means that $\rho$ is basic.
\end{proof}

\begin{prop}\label{boundary condition}
	For $ 0\leq r\leq q $, $ \mathcal{D}_B^r=\lbrace u\in\Omega_B^r\left(\overline{M}\right)\ |\;{\rm grad}\rho\lrcorner u=0\ {\rm on}\ \partial M\rbrace $, and $ {\rm d}_B^*=\delta_B $ on $ \mathcal{D}_B^r $.
\end{prop}

\begin{proof}
	Recalling $ {\rm Dom}\left({\rm d}_B^*\right)=\lbrace u\in H_{B,0}^r\ |\ \exists C_u>0\ s.t.\ |\left(u,{\rm d}_Bv\right)_B|\leq C_u\|v\|_B,\ \forall v\in\Omega_B^{r-1}\left(\overline{M}\right)\rbrace $. Clearly, the direction $ "\supseteq" $ is trivial. On the other hand, if $ u\in\Omega_B^r\left(\overline{M}\right) $ and $ v\in\Omega_B^{r-1}\left(\overline{M}\right) $, it follows from Proposition $ \ref{the reprensentation of formal adjoint} $ that 
	\begin{equation*}
		\left(u,{\rm d}_Bv\right)_B=\left(\delta_Bu,v\right)_B+\int_{\partial M}\left<{\rm grad}\rho\lrcorner u,v\right>.
	\end{equation*}
    Thus, $ u\in\mathcal{D}_B^r $ forces that 
    \begin{equation*}
    	\left(\left(\delta_B-{\rm d}_B^*\right)u,\rho v\right)_B+\int_{\partial M}\left<{\rm grad}\rho\lrcorner u,\rho v\right>=\left(\left(\delta_B-{\rm d}_B^*\right)u,\rho v\right)_B=0
    \end{equation*}
    for all $ v\in\Omega_B^r\left(\overline{M}\right) $  since $ \rho $ is a basic function and $ \rho|_{\partial M}\equiv0 $. Then the above integration formula gives $ \delta_Bu={\rm d}_B^*u $ if we choose $ v $ to be the basic form $ \left(\delta_B-{\rm d}_B^*\right)u $ because $ \rho<0 $ in $ M $. It turns out that the boundary term should vanish for all $ v\in\Omega_B^{r-1}\left(\overline{M}\right)$, i.e., $ {\rm grad}\rho\lrcorner u=0 $ on $ \partial M $.
\end{proof}

For the purpose of studying the basic Laplacian operator $ \Delta_B $, we need the following well-known Friedrichs Extension Theorem (see \cite{Fko44}).

\begin{thm}\label{Friedrichs extension theorem}
	Let $ \left(H,\|\cdot\|\right) $ be a Hilbert space, D is a dense subspace of H with Hermitian form G satisfying $ G\left(\phi,\phi\right)\geq\|\phi\|^2 $ for $ \phi\in D $. Suppose $ \left(D,G\right) $ is also a Hilbert space, then there is a unique self-adjoint operator F with $ {\rm Dom}\left(F\right)\subseteq D $, such that $ G\left(\phi,\psi\right)=\left(F\phi,\psi\right) $ for $ \forall\phi\in{\rm Dom}\left(F\right) $ and $ \forall\psi\in D $.
\end{thm} 

In this context, for $ 0\leq r\leq q $, we define $ G_B $ on $ \mathcal{D}_B^r $ by
\begin{equation*}
	G_B\left(u,v\right)=\left({\rm d}_Bu,{\rm d}_Bv\right)_B+\left(\delta_Bu,\delta_Bv\right)_B+\left(u,v\right)_B,
\end{equation*} 
let $ \tilde{\mathcal{D}}_B^r $ be the completion of $ \mathcal{D}_B^r $ with respect to $ G_B $. In order to apply Theorem $ \ref{Friedrichs extension theorem} $, we need to show that $ \tilde{\mathcal{D}}_B^r $ can be treated as a dense subspace of $ H_{B,0}^r $.

\begin{prop}\label{verification of Friedrichs extension theorem}
	For $ 0\leq r\leq q $, the map $ i:\tilde{\mathcal{D}}_B^r\longrightarrow H_{B,0}^r $ induced by $ \mathcal{D}_B^r\longrightarrow H_{B,0}^r $ is injective, and $ \tilde{\mathcal{D}}_B^r $ is dense in $ H_{B,0}^r $.
\end{prop}

\begin{proof}
	For any $ \tilde{u}\in\tilde{\mathcal{D}}_B^r $, there exists a $ G_B $-Cauchy sequence $ \lbrace u_\nu\rbrace_{\nu=1}^\infty $ in $ \mathcal{D}_B^r $ such that $ u_\nu\stackrel{G_B}{\longrightarrow}\tilde{u} $ as $ \nu\rightarrow0 $. Then $ \lbrace u_\nu\rbrace_{\nu=1}^\infty,\ \lbrace {\rm d}_Bu_\nu\rbrace_{\nu=1}^\infty $ and $ \lbrace\delta_Bu_\nu\rbrace_{\nu=1}^\infty $ are Cauchy sequences of $ H_{B,0}^r,\ H_{B,0}^{r+1} $ and $ \ H_{B,0}^{r-1} $ respectively. Let $ i\left(\tilde{u}\right):=u $ where $ u=\lim_{\nu\rightarrow\infty}u_\nu $ in $ H_{B,0}^r $. Furthermore, $ u\in{\rm Dom}\left({\rm d}_B\right)\cap{\rm Dom}\left({\rm d}_B^*\right) $ and $ G_B\left(u,u\right)=\|{\rm d}_Bu\|_B^2+\|{\rm d}_B^*u\|_B^2+\|u\|_B^2 $ because $ {\rm d}_B^* $ and $ {\rm d}_B $ are closed operators, which means that $ u=\tilde{u} $. Thus the map $ i $ is injective. It's obvious that $ \rho*\Omega_B^r\left(\overline{M}\right):=\lbrace\rho u\:|\:u\in\Omega_B^r\left(\overline{M}\right)\rbrace\subseteq\mathcal{D}_B^r $ is dense in $ \Omega_B^r\left(\overline{M}\right) $ then also dense in $ H_{B,0}^r $. In fact, if $ v\in\left(\rho*\Omega_{B}^r\left(\overline{M}\right)\right)^\perp\cap\Omega_{B}^r\left(\overline{M}\right) $, then $ 0=\left(\rho v,v\right)_B=\int_{M}\rho|v|^2 $ which implies that $ v\equiv0 $ since $ \rho<0 $ in $ M $. Therefore, $ \tilde{\mathcal{D}}_B^r $ is a dense subspace of $ H_{B,0}^r $.
\end{proof}

We will use $ F_B $ to denote the Friedrichs operator associated to the form $ G_B $. We are going to describe the boundary condition of smooth elements in Dom($ F_B $) and obtain the expression of $ F_B $. 

\begin{prop}\label{expression of $ F_B $}
	For $ 0\leq r\leq q $, if $ u\in\mathcal{D}_B^r $, then $ u\in{\rm Dom}\left(F_B\right) $ if and only if $ {\rm d}_Bu\in\mathcal{D}_B^{r+1} $, in this case we have $ F_Bu=\left(\Delta_B+{\rm Id}\right)u $.
\end{prop}
 
\begin{proof}
	$ "\Rightarrow" $ If $ u\in\mathcal{D}_B^r\cap{\rm Dom}\left(F_B\right) $, then
	\begin{equation}\label{expression of F_B}
		\left(F_Bu,v\right)_B=G_B\left(u,v\right)=\left({\rm d}_Bu,{\rm d}_Bv\right)_B+\left(\delta_Bu,\delta_Bv\right)_B+\left(u,v\right)_B=\left(\left(\Delta_B+{\rm Id}\right)u,v\right)_B
	\end{equation}
	for all $ v\in\rho*\Omega_B^r\left(\overline{M}\right)\subseteq\mathcal{D}_B^r $. Since $ \rho*\Omega_B^r\left(\overline{M}\right) $ is dense in $ H_{B,0}^r $, we have $ F_Bu=\left(\Delta_B+{\rm Id}\right)u $. In particular, if $ v\in\mathcal{D}_B^r $, the equality ($ \ref{expression of F_B} $) should be valid. For the first term,
	\begin{align*}
		\left({\rm d}_Bu,{\rm d}_Bv\right)_B=\left(\delta_B{\rm d}_Bu,v\right)_B+\int_{\partial M}\left<{\rm grad}\rho\lrcorner {\rm d}_Bu,v\right>.
	\end{align*}
    Thanks to $ \left({\rm grad}\rho\lrcorner\right)^2=0 $, we know that for $ v:={\rm grad}\rho\lrcorner {\rm d}_Bu\in\mathcal{D}_B^r $ the boundary term $ |{\rm grad}\rho\lrcorner {\rm d}_Bu|^2 $ must be vanished, which means $ {\rm d}_Bu\in\mathcal{D}_B^{r+1} $.
    
    The second term can be handled in a similar way,
    \begin{align*}
   	    \left(\delta_Bu,\delta_Bv\right)_B=\left({\rm d}_B\delta_Bu,v\right)_B-\int_{\partial M}\left<\delta_Bu,{\rm grad}\rho\lrcorner v\right>=\left({\rm d}_B\delta_Bu,v\right)_B.
    \end{align*}
    ''$ \Leftarrow $'' Conversely, we know that $ G_B\left(u,v\right)=\left(\left(\Delta_B+{\rm Id}\right)u,v\right) $ for all $ v\in\mathcal{D}_B^r $ if $ {\rm d}_Bu\in\mathcal{D}_B^{r+1} $, hence $ u\in{\rm Dom}\left(F_B\right) $ and $ F_Bu=\left(\Delta_B+{\rm Id}\right)u $.
\end{proof}


Similarly, for $ 0\leq r\leq q $ we can define $ {\rm d}_Q^* $, $ \mathcal{D}_Q^r:={\rm Dom}({\rm d}_Q^*)\cap\Omega_Q^r\left(\overline{M}\right) $, and define $ G_Q $ on $ \mathcal{D}_Q^r $ through replacing $ \left(\cdot,\cdot\right)_B $ by the general inner product $ \left(\cdot,\cdot\right) $ in $ M $. We also have the corresponding Friedrichs operator denoted by $ F_Q $. Then the same statements in Proposition $ \ref{boundary condition} $, Proposition $ \ref{verification of Friedrichs extension theorem} $ and Proposition $ \ref{expression of $ F_B $} $ also hold for these $ Q $-operators while the arguments are easier since $ \Omega_{Q,c}^r\left(M\right) $ is dense in $ \mathcal{D}_Q^r $. In particular, we have $ \mathcal{D}_Q^r\cap\Omega_B^r\left(\overline{M}\right)=\lbrace u\in\Omega_B^r\left(\overline{M}\right)\ |\ {\rm grad}\rho\lrcorner u=0\ {\rm on}\ \partial M\rbrace=\mathcal{D}_B^r $, it turns out that $ {\rm Dom}\left(F_Q\right)\cap\mathcal{D}_Q^r\cap\Omega_B^r\left(\overline{M}\right)={\rm Dom}\left(F_B\right)\cap\mathcal{D}_B^r $ which yields
\begin{align}\label{FQFB}
	F_Q|_{{\rm Dom}\left(F_B\right)}=F_B.
\end{align}

We will make use of above equality several times in the next section.

\subsection{Bochner formula}
For the development of the de Rham-Hodge theory and vanishing theorems for Riemannian foliations Bochner formula is the keystone. Conveniently, we will establish the Bochner formula for $u\in\mathcal{D}_Q^r$, in applications this has
the advantage that one has to estimate the Sobolev norms locally for the solution of the equation $F_Bu=\omega$ in the next section.

\begin{definition}
	Let $ w $ be a $ C^2 $ function on $\overline{M}$. The following quadratic form $ L_w $ is said to be a transversal Levi form of w associated to the Riemannian foliation $ \mathcal{F} $
	\begin{equation*}
		L_w\left(\xi,\xi\right):=w_{\alpha\beta}\xi_\alpha \xi_\beta:=\left(\nabla_{X_\alpha,X_\beta}^2w\right)\xi_\alpha \xi_\beta=\left(X_\beta X_\alpha w-\Gamma_{\beta\alpha}^\gamma X_\gamma w\right)\xi_\alpha \xi_\beta,
	\end{equation*}
	where $ \xi=\xi_\alpha X_\alpha $ is a transversal vector field and $ \Gamma_{\beta\alpha}^\gamma $'s are coefficients of the connection $ \nabla $.
\end{definition}

It's easy to see that the definition is independent of the choice of local frames. 

\begin{thm}
	For $ 0\leq r\leq q $, $ u\in\mathcal{D}_Q^r $ we have
	\begin{align}\label{Bochner-Kodaira formula}
		\|{\rm d}_Qu\|^2+\|\delta_Qu\|^2=&\int_{M}g^{\alpha\beta}\left<\nabla_{X_\alpha} u,\nabla_{X_\beta}u\right>+2\int_{M}\left<g^{\alpha\beta}\wedge X_\alpha\lrcorner\nabla_{A\left(X_\beta,X_\gamma\right)}u,g^{\gamma\sigma}\left(X_\sigma\lrcorner u\right)\right>\nonumber\\
		&+\int_{M}\left<R^\nabla_{X_\beta,X_\gamma}g^{\alpha\beta}\left(X_\alpha\lrcorner u\right),g^{\gamma\sigma}\left(X_\sigma\lrcorner u\right)\right>+\int_{M}\left<Ric^\nabla\left(X_\gamma\right)\lrcorner u,g^{\gamma\sigma}\left(X_\sigma\lrcorner u\right)\right>\nonumber\\
		&+\int_{M}\left<\mathcal{A}_\tau u,u\right>+\int_{\partial M}\left<\rho_{\beta\gamma}g^{\alpha\beta}\left(X_\alpha\lrcorner u\right),g^{\gamma\sigma}\left(X_\sigma\lrcorner u\right)\right>,		
	\end{align}
	where $ A\left(X_\beta,X_\gamma\right):=\pi^\perp\left(\nabla_{X_\beta}^MX_\gamma\right) $ is the second fundamental form of $ Q $ in $ \left(M,g_M\right) $, $ R^\nabla $ is the curvature operator of the connection $ \nabla $, $ Ric^\nabla\left(X_\gamma\right):=g^{\alpha\beta}R_{X_\gamma,X_\beta}^\nabla X_\alpha $, $ \tau $ is the mean curvature vector field of $ \mathcal{F} $, and $ \mathcal{A}_\tau u:=\theta^\alpha\wedge\left(\nabla_{X_\alpha}\tau\right)\lrcorner u $.
\end{thm}
\begin{proof}
	Let $ u=u_J\theta^J\in\mathcal{D}_Q^r $, then	
	\begin{align*}
		&-g^{\alpha\beta}X_\alpha\lrcorner\nabla_{X_\beta}\left(\theta^\gamma\wedge\nabla_{X_\gamma}u\right)\\
		=&-g^{\alpha\beta}X_\alpha\lrcorner\left(-\Gamma_{\beta\sigma}^\gamma\theta^\sigma\wedge\nabla_{X_\gamma}u+\theta^\gamma\wedge\nabla_{X_\beta}\nabla_{X_\gamma}u\right)\\
		=&-g^{\alpha\beta}\left(-\Gamma_{\beta\alpha}^\gamma\nabla_{X_\gamma}u+\Gamma_{\beta\sigma}^\gamma\theta^\sigma\wedge X_{\alpha}\lrcorner\nabla_{X_\gamma}u+\nabla_{X_\beta}\nabla_{X_\alpha}u-\theta^\gamma\wedge X_\alpha\lrcorner\nabla_{X_\beta}\nabla_{X_\gamma}u\right)\\
		=&-g^{\alpha\beta}\nabla_{X_\beta}\nabla_{X_\alpha}u+g^{\alpha\beta}\theta^\gamma\wedge X_\alpha\lrcorner\nabla_{X_\beta}\nabla_{X_\gamma}u+g^{\alpha\beta}\Gamma_{\alpha\beta}^\gamma\nabla_{X_\gamma}u-g^{\alpha\beta}\Gamma_{\beta\sigma}^\gamma\theta^\sigma\wedge X_\alpha\lrcorner\nabla_{X_\gamma}u.
	\end{align*}
	On the other hand,
	\begin{align*}
		&-\theta^\gamma\wedge\nabla_{X_\gamma}\left(g^{\alpha\beta}X_\alpha\lrcorner\nabla_{X_\beta}u\right)\\
		=&-\theta^\gamma\wedge X_\gamma\left(g^{\alpha\beta}\right)X_\alpha\lrcorner\nabla_{X_\beta}u-g^{\alpha\beta}\theta^\gamma\wedge\nabla_{X_\gamma}\left(X_\alpha\lrcorner\nabla_{X_\beta}u\right)\\
		=&-\theta^\gamma\wedge X_\gamma\left(\left<\theta^\alpha,\theta^\beta\right>\right)X_\alpha\lrcorner\nabla_{X_\beta}u-g^{\alpha\beta}\theta^\gamma\wedge\nabla_{X_\gamma}X_\alpha\lrcorner\nabla_{X_\beta}u-g^{\alpha\beta}\theta^\gamma\wedge X_\alpha\lrcorner\nabla_{X_\gamma}\nabla_{X_\beta}u\\
		=&\theta^\gamma\left(\Gamma_{\gamma\sigma}^\alpha g^{\beta\sigma}+\Gamma_{\gamma\sigma}^\beta g^{\alpha\sigma}\right)X_\alpha\lrcorner\nabla_{X_\beta}u-g^{\alpha\beta}\theta^\gamma\wedge\Gamma_{\gamma\alpha}^\sigma X_\sigma\lrcorner\nabla_{X_\beta}u-g^{\alpha\beta}\theta^\gamma\wedge X_\alpha\lrcorner\nabla_{X_\gamma}\nabla_{X_\beta}u\\
		=&g^{\alpha\sigma}\Gamma_{\gamma\sigma}^\beta\theta^\gamma\wedge X_\alpha\lrcorner\nabla_{X_\beta}u-g^{\alpha\beta}\theta^\gamma\wedge X_\alpha\lrcorner\nabla_{X_\gamma}\nabla_{X_\beta}u.
	\end{align*}
	Then from (\ref{Q-operator}) we arrive at:
	\begin{align*}
		\left(\delta_Q{\rm d}_Q+{\rm d}_Q\delta_Q\right)u&=\left(-g^{\alpha\beta}X_\alpha\lrcorner\nabla_{X_\beta}+\tau\lrcorner\right)\theta^\gamma\wedge \nabla_{X_\gamma}u+\theta^\gamma\wedge \nabla_{X_\gamma}\left(-g^{\alpha\beta}X_\alpha\lrcorner\nabla_{X_\beta}+\tau\lrcorner\right)u\\
		&=\left(-g^{\alpha\beta}\nabla_{X_\beta}\nabla_{X_\alpha}+g^{\alpha\beta}\Gamma_{\alpha\beta}^\gamma\nabla_{X_\gamma}\right)u+g^{\alpha\beta}\theta^\gamma\wedge X_\alpha\lrcorner\left[\nabla_{X_\beta},\nabla_{X_\gamma}\right]+L_\tau^Qu\\
		&=-g^{\alpha\beta}\nabla_{X_\alpha,X_\beta}^2u+g^{\alpha\beta}\theta^\gamma\wedge X_\alpha\lrcorner\left[\nabla_{X_\beta},\nabla_{X_\gamma}\right]u+L_\tau^Qu,
	\end{align*}
	where $ L_\tau^Q:=\tau\lrcorner\circ{\rm d}_Q+{\rm d}_Q\circ\tau\lrcorner $. 
	
	We claim that $ \nabla_{\left[X_\beta,X_\gamma\right]}=2\nabla_{A\left(X_\beta,X_\gamma\right)} $, recalling $ A\left(X_\beta,X_\gamma\right):=\pi^\perp\left(\nabla_{X_\beta}^MX_\gamma\right) $. Indeed,   
	\begin{align*}
		\nabla_{\left[X_\beta,X_\gamma\right]}&=\left(\Gamma_{\beta\gamma}^\alpha-\Gamma_{\gamma\beta}^\alpha\right)\nabla_{X_\alpha}u+\left(\Gamma_{\beta\gamma}^a-\Gamma_{\gamma\beta}^a\right)\nabla_{X_a}=2\nabla_{\Gamma_{\beta\gamma}^aX_a}=2\nabla_{A\left(X_\beta,X_\gamma\right)}.
	\end{align*}
	The second equality follows from Proposition $ \ref{coefficients of conection} $. By
	$$ R^\nabla_{X_\beta,X_\gamma}=\left[\nabla_{X_\beta},\nabla_{X_\gamma}\right]-\nabla_{\left[X_\beta,X_\gamma\right]}=\left[\nabla_{X_\beta},\nabla_{X_\gamma}\right]-2\nabla_{A\left(X_\beta,X_\gamma\right)} $$
	we know that
	\begin{equation*}
		\left(\delta_Q{\rm d}_Q+{\rm d}_Q\delta_Q\right)u=-g^{\alpha\beta}\nabla_{X_\alpha,X_\beta}^2u+g^{\alpha\beta}\theta^\gamma\wedge X_\alpha\lrcorner R_{X_\beta,X_\gamma}^\nabla u+2g^{\alpha\beta}\theta^\gamma\wedge X_\alpha\lrcorner\nabla_{A\left(X_\beta,X_\gamma\right)}u+L_\tau^Qu.
	\end{equation*}
	
	It's clear that $ \left[R_{X_\beta,X_\gamma}^\nabla,X_\alpha\lrcorner\right]=\left(R_{X_\beta,X_\gamma}^\nabla X_\alpha\right)\lrcorner $. In fact, for $ u\in\Omega_{Q}^r\left(\overline{M}\right) $
	\begin{align*}
		\left[R_{X_\beta,X_\gamma}^\nabla,X_\alpha\lrcorner\right]u&=\left[\left[\nabla_{X_\beta},\nabla_{X_\gamma}\right],X_\alpha\lrcorner\right]u+\left[\nabla_{\left[X_\beta,X_\gamma\right]},X_\alpha\lrcorner\right]u\\
		&=\left[\nabla_{X_\beta},\nabla_{X_\gamma}\right]\left(X_\alpha\lrcorner u\right)-X_\alpha\lrcorner\left[\nabla_{X_\beta},\nabla_{X_\gamma}\right]u+\nabla_{\left[X_\beta,X_\gamma\right]}\left(X_\alpha\lrcorner u\right)-X_\alpha\lrcorner\nabla_{\left[X_\beta,X_\gamma\right]}u\\
		&=\left(\left[\nabla_{X_\beta},\nabla_{X_\gamma}\right]X_\alpha\right)\lrcorner u+\left(\nabla_{\left[X_\beta,X_\gamma\right]}X_\alpha\right)\lrcorner u\\
		&=\left(R_{X_\beta,X_\gamma}^\nabla X_\alpha\right)\lrcorner u.
	\end{align*}
	The third equality is valid since $ \nabla_{X_i}\left(X_j\lrcorner u\right)=\left(\nabla_{X_i}X_j\right)\lrcorner u+X_j\lrcorner\nabla_{X_i}u $ for all $ X_i,X_j\in\{X_k\}_{k=1}^n $.
	Thus,
	\begin{align}\label{expression of box}
		\left(\delta_Q{\rm d}_Q+{\rm d}_Q\delta_Q\right)u=&-g^{\alpha\beta}\nabla_{X_\alpha,X_\beta}^2u+g^{\alpha\beta}\theta^\gamma\wedge R_{X_\beta,X_\gamma}^\nabla\left(X_\alpha\lrcorner u\right)+\theta^\gamma\wedge Ric^\nabla\left(X_\gamma\right)\lrcorner u\nonumber\\
		&+2g^{\alpha\beta}\theta^\gamma\wedge X_\alpha\lrcorner\nabla_{A\left(X_\beta,X_\gamma\right)}u+L_\tau^Qu.
	\end{align}
	
	Now, we handle the first term in R.H.S. of $ \left(\ref{expression of box}\right) $, first note that
	\begin{align*}
		&{\rm div}^Q\left(g^{\alpha\beta}\left<\nabla_{X_\alpha}u,u\right>X_\beta\right)\nonumber\\
		=&X_\beta\left(g^{\alpha\beta}\left<\nabla_{X_\alpha}u,u\right>\right)+g^{\alpha\beta}\left<\nabla_{X_\alpha}u,u\right>{\rm div}^QX_\beta\nonumber\\
		=&X_\beta\left(\left<\theta^\alpha,\theta^\beta\right>\right)\left<\nabla_{X_\alpha}u,u\right>+g^{\alpha\beta}\left<\nabla_{X_\beta}\nabla_{X_\alpha}u,u\right>+g^{\alpha\beta}\left<\nabla_{X_\alpha}u,\nabla_{X_\beta}u\right>+g^{\alpha\beta}\left<\nabla_{X_\alpha}u,u\right>g^{\gamma\sigma}\left<\Gamma_{\gamma\beta}^\lambda X_\lambda,X_\sigma\right>\nonumber\\
		=&-\left(\Gamma_{\beta\gamma}^\alpha g^{\beta\gamma}+\Gamma_{\beta\gamma}^\beta g^{\alpha\gamma}\right)\left<\nabla_{X_\alpha}u,u\right>+g^{\alpha\beta}\left<\nabla_{X_\beta}\nabla_{X_\alpha}u,u\right>+g^{\alpha\beta}\left<\nabla_{X_\alpha} u,\nabla_{X_\beta}u\right>+\Gamma_{\gamma\beta}^\gamma g^{\alpha\beta}\left<\nabla_{X_\alpha}u,u\right>\nonumber\\
		=&\left<g^{\alpha\beta}\nabla_{X_\beta}\nabla_{X_\alpha}u-g^{\beta\gamma}\Gamma_{\beta\gamma}^\alpha\nabla_{X_\alpha} u,u\right>+g^{\alpha\beta}\left<\nabla_{X_\alpha} u,\nabla_{X_\beta}u\right>\nonumber\\
		=&\left<g^{\alpha\beta}\nabla_{X_\alpha,X_\beta}^2u,u\right>+g^{\alpha\beta}\left<\nabla_{X_\alpha} u,\nabla_{X_\beta}u\right>.
	\end{align*}
	It gives
	\begin{align}\label{handle the first term of box}
		-\left<g^{\alpha\beta}\nabla_{X_\alpha,X_\beta}^2u,u\right>&=g^{\alpha\beta}\left<\nabla_{X_\alpha} u,\nabla_{X_\beta}u\right>-{\rm div}^Q\left(g^{\alpha\beta}\left<\nabla_{X_\alpha}u,u\right>X_\beta\right)\nonumber\\
		&=g^{\alpha\beta}\left<\nabla_{X_\alpha} u,\nabla_{X_\beta}u\right>-{\rm div}^M\left(g^{\alpha\beta}\left<\nabla_{X_\alpha}u,u\right>X_\beta\right)+{\rm div}^L\left(g^{\alpha\beta}\left<\nabla_{X_\alpha}u,u\right>X_\beta\right)\nonumber\\
		&=g^{\alpha\beta}\left<\nabla_{X_\alpha} u,\nabla_{X_\beta}u\right>-{\rm div}^M\left(g^{\alpha\beta}\left<\nabla_{X_\alpha}u,u\right>X_\beta\right)-g^{\alpha\beta}\left<\nabla_{X_\alpha}u,u\right>\left<X_\beta,g^{ab}\nabla_{X_a}^MX_b\right>\nonumber\\
		&=g^{\alpha\beta}\left<\nabla_{X_\alpha} u,\nabla_{X_\beta}u\right>-{\rm div}^M\left(g^{\alpha\beta}\left<\nabla_{X_\alpha}u,u\right>X_\beta\right)-\left<\nabla_\tau u,u\right>.
	\end{align}
	
	Then we deal with the mean curvature term of $ \left(\ref{expression of box}\right) $ as follows
	\begin{align}\label{mean curvature term}
		L_\tau^Qu-\nabla_\tau u&={\rm d}_Q\circ\tau\lrcorner u+\tau\lrcorner\circ {\rm d}_Qu-\nabla_\tau u\nonumber\\
		&=\theta^\alpha\wedge\nabla_{X_\alpha}\left(\tau\lrcorner u\right)+\tau\lrcorner\left(\theta^\alpha\wedge\nabla_{X_\alpha}u\right)-\nabla_\tau u\nonumber\\
		&=\theta^\alpha\wedge\left(\nabla_{X_\alpha}\tau\right)\lrcorner u+\left(\tau\lrcorner\theta^\alpha\right)\nabla_{X_\alpha}u-\nabla_\tau u\nonumber\\
		&=\theta^\alpha\wedge\left(\nabla_{X_\alpha}\tau\right)\lrcorner u+\nabla_{\left(\tau\lrcorner\theta^\alpha\right)X_\alpha}u-\nabla_\tau u\nonumber\\
		&=\theta^\alpha\wedge\left(\nabla_{X_\alpha}\tau\right)\lrcorner u\nonumber\\
		&=:\mathcal{A}_\tau u.
	\end{align}
	
	Combining (\ref{expression of box}), (\ref{handle the first term of box}) and (\ref{mean curvature term}) and by the divergence theorem, we get
	\begin{align}\label{integration of box}
		\left(\left({\rm d}_Q\delta_Q+\delta_Q{\rm d}_Q\right)u,u\right)=&\int_{M}g^{\alpha\beta}\left<\nabla_{X_\alpha} u,\nabla_{X_\beta}u\right>+2\int_{M}\left<g^{\alpha\beta}\wedge X_\alpha\lrcorner\nabla_{A\left(X_\beta,X_\gamma\right)}u,g^{\gamma\sigma}\left(X_\sigma\lrcorner u\right)\right>\nonumber\\
		&+\int_{M}\left<R^\nabla_{X_\beta,X_\gamma}g^{\alpha\beta}\left(X_\alpha\lrcorner u\right),g^{\gamma\sigma}\left(X_\sigma\lrcorner u\right)\right>+\int_{M}\left<Ric^\nabla\left(X_\gamma\right)\lrcorner u,g^{\gamma\sigma}\left(X_\sigma\lrcorner u\right)\right>\nonumber\\
		&+\int_{M}\left<\mathcal{A}_\tau u,u\right>-\int_{\partial M}\left<\nabla_{{\rm grad}\rho}u,u\right>.
	\end{align}
	Since $ u\in\mathcal{D}_Q^r $, it follows from Proposition $ \ref{the reprensentation of formal adjoint} $ and Proposition \ref{boundary condition} that
	\begin{align}\label{interation by parts}
		\|{\rm d}_Qu\|^2+\|\delta_Qu\|^2=\left(\left({\rm d}_Q\delta_Q+\delta_Q{\rm d}_Q\right)u,u\right)+\int_{\partial M}\left<u,{\rm grad}\rho\lrcorner {\rm d}_Qu\right>.
	\end{align}
	
	We start to handle the integration on the boundary $ \partial M $. For $ u=u_J\theta^J\in\mathcal{D}_Q^r $,
	\begin{align*}
		{\rm grad}\rho\lrcorner {\rm d}_Qu&=g^{\alpha\beta}X_\beta\left(\rho\right)X_\alpha\lrcorner\left(\theta^\gamma\wedge\nabla_{X_\gamma}u\right)\\
		&=g^{\alpha\beta}X_\beta\left(\rho\right)\nabla_{X_\alpha}u-g^{\alpha\beta}X_\beta\left(\rho\right)\theta^\gamma\wedge X_\alpha\lrcorner\nabla_{X_\gamma}u\\
		&=\nabla_{{\rm grad}\rho}u-g^{\alpha\beta}X_\beta\left(\rho\right)\theta^\gamma\wedge X_\alpha\lrcorner\left(X_\gamma\left(u_J\right)\theta^J+u_J\nabla_{X_\gamma}\theta^J\right)\\
		&=\nabla_{{\rm grad}\rho}u-g^{\alpha\beta}X_\beta\left(\rho\right)\theta^\gamma\wedge \left(X_\gamma\left(u_{\alpha I}\right)\theta^{\alpha I}-u_JX_\alpha\lrcorner\Gamma_{\gamma\sigma}^\lambda\theta^\sigma\wedge X_\lambda\lrcorner\theta^J\right)\\
		&=\nabla_{{\rm grad}\rho}u-g^{\alpha\beta}X_\beta\left(\rho\right)\theta^\gamma\wedge \left(X_\gamma\left(u_{\alpha I}\right)\theta^{\alpha I}-u_J\Gamma_{\gamma\alpha}^\lambda X_\lambda\lrcorner\theta^J+u_J\Gamma_{\gamma\sigma}^\lambda\theta^\sigma\wedge X_\alpha\lrcorner X_\lambda\lrcorner\theta^J\right)\\
		&=\nabla_{{\rm grad}\rho}u-\theta^\gamma\wedge g^{\alpha\beta}X_\beta\left(\rho\right)X_\gamma\left(u_{\alpha I}\right)\theta^{\alpha I}+\theta^\gamma\wedge g^{\alpha\beta}X_\beta\left(\rho\right)\Gamma_{\gamma\alpha}^\lambda u_{\lambda I}\theta^{\lambda I}.
	\end{align*}
	The last line holds because that all $ \Gamma_{\gamma\sigma}^\lambda $ are symmetric with $ \gamma,\sigma $ by proposition $ \ref{coefficients of conection} $, while $ \theta^\gamma\wedge\theta^\sigma $ are anti-commutative.
	We therefore obtain the following integration on $ \partial M $
	\begin{align*}
		&\int_{\partial M}\left<u,{\rm grad}\rho\lrcorner {\rm d}_Qu\right>-\int_{\partial M}\left<\nabla_{{\rm grad}\rho}u,u\right>\\
		=&\int_{\partial M}\left<\theta^\gamma\wedge g^{\alpha\beta}X_\beta\left(\rho\right)X_\gamma\left(u_{\alpha I}\right)\theta^{\alpha I}+\theta^\gamma\wedge g^{\alpha\beta}X_\beta\left(\rho\right)\Gamma_{\gamma\alpha}^\lambda u_{\lambda I}\theta^{\lambda I},u\right>\\
		=&\int_{\partial M}\left<-g^{\alpha\beta}X_\beta\left(\rho\right)X_\gamma\left(u_{\alpha I}\right)+g^{\alpha\beta}X_\beta\left(\rho\right)\Gamma_{\gamma\alpha}^\lambda u_{\lambda I},g^{\gamma\sigma}u_{\sigma I'}\right>g^{II'},
	\end{align*}
	where $ g^{II'}:=\left<\theta^I,\theta^{I'}\right> $. 
	
	Due to Proposition \ref{boundary condition}, $ u\in\mathcal{D}_Q^r $ means that $ g^{\alpha\beta}X_\beta\left(\rho\right)u_{\alpha I}=0 $ on $ \partial M $ for any mult-index $ I $ with length $ r-1 $. Moreover, it implies that the tangential derivatives of $ g^{\alpha\beta}X_\beta\left(\rho\right)u_{\alpha I} $ also vanish on $\partial M$, i.e.,
	\begin{align*}
		0&\equiv g^{\sigma\gamma}u_{\sigma I'}X_\gamma\left(g^{\alpha\beta}X_\beta\left(\rho\right)u_{\alpha I}\right)\\
		&=g^{\sigma\gamma}u_{\sigma I'}X_\gamma\left(g^{\alpha\beta}u_{\alpha I}\right)X_\beta\left(\rho\right)+g^{\sigma\gamma}g^{\alpha\beta}u_{\sigma I'}u_{\alpha I}X_\gamma X_\beta\left(\rho\right)\\
		&=-g^{\sigma\gamma}u_{\sigma I'}\left(\Gamma_{\gamma\lambda}^\alpha g^{\lambda\beta}+\Gamma_{\gamma\lambda}^\beta g^{\lambda\alpha}\right)u_{\alpha I}X_\beta\left(\rho\right)+g^{\sigma\gamma}g^{\alpha\beta}u_{\sigma I'}X_\gamma\left(u_{\alpha I}\right)X_\beta\left(\rho\right)+g^{\sigma\gamma}g^{\alpha\beta}u_{\sigma I'}u_{\alpha I}X_\gamma X_\beta\left(\rho\right).
	\end{align*}
	Substituting indices yields
	\begin{align*}
		&\left<g^{\alpha\beta}X_\beta\left(\rho\right)\Gamma_{\gamma\alpha}^\lambda u_{\lambda I},g^{\gamma\sigma}u_{\sigma I'}\right>-\left<X_\beta\left(\rho\right)g^{\alpha\beta}X_\gamma\left(u_{\alpha I}\right),g^{\sigma\gamma}u_{\sigma I'}\right>\\
		=&\left<X_\gamma X_\beta\left(\rho\right)g^{\alpha\beta}u_{\alpha I},g^{\sigma\gamma}u_{\sigma I'}\right>-\left<\Gamma_{\gamma\beta}^\lambda X_\lambda\left(\rho\right)g^{\beta\alpha}u_{\alpha I},g^{\sigma\gamma}u_{\sigma I'}\right>\\
		=&\left<\nabla_{X_\beta,X_\gamma}^2\left(\rho\right)g^{\alpha\beta}u_{\alpha I},g^{\sigma\gamma}u_{\sigma I'}\right>
	\end{align*}
	on $\partial M$. Hence,
	\begin{align}\label{boundary integration}
		\int_{\partial M}\left<u,{\rm grad}\rho\lrcorner {\rm d}_Qu\right>-\int_{\partial M}\left<\nabla_{{\rm grad}\rho}u,u\right>=\int_{\partial M}\left<\rho_{\beta\gamma}g^{\alpha\beta}\left(X_\alpha\lrcorner u\right),g^{\gamma\sigma}\left(X_\sigma\lrcorner u\right)\right>.
	\end{align}
	
	At last, combining (\ref{integration of box}), (\ref{interation by parts}) and (\ref{boundary integration}), the proof is thus complete.
\end{proof}

\begin{remark}\label{Bochner for basic form}
	The term $ \int_{M}\left<g^{\alpha\beta}\wedge X_\alpha\lrcorner\nabla_{A\left(X_\beta,X_\gamma\right)}u,g^{\gamma\sigma}\left(X_\sigma\lrcorner u\right)\right> $ in the R.H.S. of $ (\ref{Bochner-Kodaira formula}) $ vanishes provided that $ u\in\mathcal{D}_B^r $. In fact, for $ u\in\mathcal{D}_B^r $ we have
	\begin{align*}
		\nabla_{A\left(X_\beta,X_\gamma\right)}u=\nabla_{\pi^\perp\left(\nabla_{X_\beta}^MX_\gamma\right)}u=\Gamma_{\beta\gamma}^a\nabla_{X_a}u.
	\end{align*}
	Let $ \lbrace X_{i_1},\cdots,X_{i_r}\rbrace $ be the subset of local basis $ \lbrace X_i\rbrace_{i=1}^{n} $ which is introduced in section 2, then
	\begin{align*}
		\nabla_{X_a}u\left(X_{i_1},\cdots,X_{i_r}\right)&=X_a\left(u\left(X_{i_1},\cdots,X_{i_r}\right)\right)-\sum_{m=1}^{r}u\left(X_{i_1},\cdots,\nabla_{X_a}X_{i_m}\cdots,X_{i_r}\right)\\
		&=\frac{\partial}{\partial x_a}u\left(X_{i_1},\cdots,X_{i_r}\right)-\sum_{m=1}^{r}u\left(X_{i_1},\cdots,\pi\left(\left[X_a,X_{i_m}\right]\right),\cdots,X_{i_r}\right)\\
		&=0.
	\end{align*}
	The last line is valid since the coefficients of basic forms are independent of $ x'=\left(x_1,\cdots,x_p\right) $ and $ \left[X_i,X_j\right]\in\Gamma L $ for all $ X_{i},X_{j}\in\{X_i\}_{i=1}^n $ which follows from the proof of Proposition $ \ref{coefficients of conection} $.
\end{remark}

Furthermore, we have an estimate relating the $H_{B,1}$-norm of $u\in\mathcal{D}_B^r$ to the $G_B$-norm as follows.
\begin{cor}\label{H1e}
	For $0\leq r\leq q$, if $u\in\tilde{\mathcal{D}}_B^r$, then there exists a constant $C>0$, depending only on the geometry of $\overline{M}$ such that 
	\begin{align*}
		\|u\|_{B,1}^2\leq CG_B\left(u,u\right).
	\end{align*}
\end{cor}
\begin{proof}
	First we restrict ourselves to smooth basic forms $u\in\mathcal{D}_B^r$. By the Bochner formula $ \left(\ref{Bochner-Kodaira formula}\right) $ for $u\in\mathcal{D}_B^r$ and Remark \ref{Bochner for basic form} we read off
	\begin{align*}
		\|{\rm d}_Bu\|_B^2+\|\delta_Bu\|_B^2=&\int_{M}g^{\alpha\beta}\left<\nabla_{X_\alpha} u,\nabla_{X_\beta}u\right>+\int_{M}\left<\mathcal{A}_\tau u,u\right>+\int_{M}\left<R^\nabla_{X_\beta,X_\gamma}g^{\alpha\beta}\left(X_\alpha\lrcorner u\right),g^{\gamma\sigma}\left(X_\sigma\lrcorner u\right)\right>\\
		&+\int_{M}\left<Ric^\nabla\left(X_\gamma\right)\lrcorner u,g^{\gamma\sigma}\left(X_\sigma\lrcorner u\right)\right>+\int_{\partial M}\left<\rho_{\beta\gamma}g^{\alpha\beta}\left(X_\alpha\lrcorner u\right),g^{\gamma\sigma}\left(X_\sigma\lrcorner u\right)\right>\\
		=&\|u\|_{B,1}^2+O\left(\|u\|_B^2+\int_{\partial M}|u|^2\right).
	\end{align*}
	Thus, we can apply Stokes' theorem and Cauchy-Schwarz inequality to estimate
	\begin{align*}
		\int_{\partial M}|u|^2\leq\varepsilon\|u\|_{B,1}^2+C_\varepsilon\|u\|_B^2,
	\end{align*}
	which implies the desired inequality for smooth basic forms.
	
	Finally dropping the restriction of smoothness, one can conclude the assertion of the corollary for arbitrary $u\in\tilde{\mathcal{D}}_B^r$ by means of an approximating sequence $\{u_\nu\}_{\nu=1}^\infty\subseteq\mathcal{D}_B^r$ and a completeness argument.
\end{proof}

Using the terminology of Lions and Magenes \cite{LM72}, above corollary implies that the bilinear form $G_B$ is $H_{B,1}$-coercive on the space $\tilde{\mathcal{D}}_B^r$. This is a
crucial point to show the regularity for the solution of the equation $F_Bu=\omega$, while the corresponding result is false for Hermitian foliation which we will investigate in detail in the next chapter.

\subsection{The basic de Rham-Hodge decomposition}
In this section, we establish the regularity theorem of the equation $ F_Bu=\omega $ firstly. Since $ F_B=\Delta_B+{\rm Id} $ on smooth forms of ${\rm Dom}\left(F_B\right)$, we then get back to the study of basic Laplacian $ \Delta_B|_{{\rm Dom}\left(F_B\right)} $ from $ F_B $ and prove the Hodge decomposition theorem.

Since we will consider Sobolev norms locally, it's convenient for us to divide the local coordinate charts $ \{\left(x,U_l\right)\}_{l=1}^L $ into two parts $ \lbrace\left(x,U'_l\right)\rbrace_{l=1}^{L_1} $ and $ \lbrace\left(x,U''_l\right)\rbrace_{l=1}^{L_2} $, where $ U'_l\cap\partial M=\emptyset $ and $ U_l''\cap\partial M\neq\emptyset $. Let $ \lbrace\eta_l\rbrace_{l=1}^{L}:=\lbrace\eta_l'\rbrace_{l=1}^{L_1}\cup\lbrace\eta_l''\rbrace_{l=1}^{L_2} $ be a partition of unity subordinate to the covering $ \lbrace U_l\rbrace_{l=1}^{L} $, where $ \eta_l' $ and $ \eta_l'' $ with support in $ U'_l $ and $ U''_l $ respectively. For the boundary chart, we always use normal coordinate chart for convenience, i.e., the coordinate functions set is chosen as $ \lbrace x_1,\cdots,x_{n-1},\rho\rbrace $, and in this local chart we denote $ D_\rho=\frac{\partial}{\partial\rho} $, $ D_t^i=\frac{\partial}{\partial x_i} $ for $ i<n $, $ D_t^I:=\left(\frac{\partial}{\partial x}\right)^{I} $ if the $n^{th}$-component of $ I $ equals to zero.

Now, we are ready to formulate the existence and regularity theorem of the equation $ F_Bu=\omega $.
\begin{thm}\label{main theorem}
	Let $ 0\leq r\leq q $, for given $ \omega\in H_{B,0}^r $, there is a unique $ u\in\tilde{\mathcal{D}}_B^r $ such that $ F_Bu=\omega $. Moreover, if $ \omega\in\Omega_B^r\left(\overline{M}\right) $, then $ u\in\Omega_B^r\left(\overline{M}\right) $, and for each integer $ s\geq0 $, we have
	\begin{align}\label{estimate}
		\|u\|_{B,s+2}^2\lesssim\|\omega\|_{B,s}^2.
	\end{align}
\end{thm}

We will prove the estimate ($ \ref{estimate} $) for smooth basic forms at first, it suffices to prove the desired estimate locally. For this purpose, we need the following series of lemmas.

\begin{lemma}\label{main 1}
	For $ 0\leq r\leq q $, let $ I $ be a multi-index with order $ m $, and let $ D^{I} $ act on the components of forms in $ U_l' $, then 
	$$ \|D^{I}\eta_l'u\|_1^2\lesssim G_Q\left(D^{I}\eta_l'u,D^{I}\eta_l'u\right)+\|u\|_{B,m}^2 $$ for all 
	$ u\in\Omega_{B}^r\left(\overline{M}\right) $. In addition, if we replace $ D^I $ and $ \eta_l' $ by $ D_t^I $ and $ \eta_l'' $, then above estimate still holds for any $u\in\mathcal{D}_B^r$.
\end{lemma}
\begin{proof}
	Since the interior case is simpler, we only prove the estimate on a normal coodinate chart $U''$. Since $ D_t^I\eta''u\in\mathcal{D}_Q^r $, reading the Bochner formula $ \left(\ref{Bochner-Kodaira formula}\right) $ for $D_t^I\eta''u$ yields that for any $ 0<\varepsilon<<1 $, there exists constant $C_\varepsilon>0$ such that
	\begin{align*}
		G_Q\left(D_t^I\eta''u,D_t^I\eta''u\right)&=\|\nabla D_t^I\eta''u\|^2+\int_{\partial M}\left<\rho_{\beta\gamma}g^{\alpha\beta}\left(X_\alpha\lrcorner D_t^I\eta''u\right),g^{\gamma\sigma}\left(X_\sigma\lrcorner D_t^I\eta''u\right)\right>+O\left(\|u\|_{B,m}^2\right)\\
		&=\|D_t^I\eta''u\|_1^2+O\left(\int_{\partial M}|D_t^I\eta''u|^2+\|u\|_{B,m}^2\right)\nonumber\\
		&\geq\left(1-\varepsilon\right)\|D_t^I\eta''u\|_1^2+C_\varepsilon\|u\|_{B,m}^2,
	\end{align*}
	where the last line follows from Stokes' theorem and Cauchy-Schwarz inequality.
\end{proof}

\begin{lemma}\label{main 3}
	For $ 0\leq r\leq q $. Let $ I $ be multi-index with order $ m $, and let $ D^{I} $ act on the components of forms, then for all $ u\in\Omega_B^r\left(\overline{M}\right) $ we have
	\begin{align}\label{integration by prats for G_Q}
		G_Q\left(D^{I}\eta_l' u,D^{I}\eta_l' u\right)=G_Q\left(u,\eta_l' \left(D^{I}\right)^*D^{I}\eta_l' u\right)+O\left(\|u\|_{B,m}\|u\|_{B,m+1}\right),
	\end{align}
    where $ \left(D^{I}\right)^* $ is formal adjoint of $ D^{I} $. Additionally, if we replace $ D^I $ and $ \eta_l' $ by $ D_t^I $ and $ \eta_l'' $, then equality $ (\ref{integration by prats for G_Q}) $ still hold for all $ u\in\Omega_B^r\left(\overline{M}\right) $.
\end{lemma}
\begin{proof}
	For $ u\in\Omega_B^r\left(\overline{M}\right) $, then
	\begin{align*}
		&\left({\rm d}_QD^{I}\eta_l' u,{\rm d}_QD^{I}\eta_l' u\right)\nonumber\\
		=&\left(D^{I}\eta_l' {\rm d}_Qu,{\rm d}_QD^{I}\eta_l' u\right)+O\left(\|u\|_{B,m}\|u\|_{B,m+1}\right)\\
		=&\left({\rm d}_Qu,\eta_l'\left(D^{I}\right)^*{\rm d}_QD^{I}\eta_l' u\right)+O\left(\|u\|_{B,m}\|u\|_{B,m+1}\right)\\
		=&\left({\rm d}_Qu,{\rm d}_Q\eta_l'\left(D^{I}\right)^*D^{I}\eta_l' u\right)+\left({\rm d}_Qu,\left[\eta_l'\left(D^{I}\right)^*,{\rm d}_Q\right]D^{I}\eta_l' u\right)+O\left(\|u\|_{B,m}\|u\|_{B,m+1}\right)\\
		=&\left({\rm d}_Qu,{\rm d}_Q\eta_l'\left(D^{I}\right)^*D^{I}\eta_l' u\right)+\|{\rm d}_Qu\|_{m}\|\left[\eta_l'\left(D^{I}\right)^*,{\rm d}_Q\right]D^{I}\eta_l' u\|_{-m}+O\left(\|u\|_{B,m}\|u\|_{B,m+1}\right)\\
		=&\left({\rm d}_Qu,{\rm d}_Q\eta_l'\left(D^{I}\right)^*D^{I}\eta_l' u\right)+O\left(\|u\|_{B,m}\|u\|_{B,m+1}\right).
	\end{align*}
	Correspondingly, we have $ \left(\delta_QD^{I}\eta_l' u,\delta_QD^{I}\eta_l' u\right)=\left(\delta_Qu,\delta_Q\eta_l'\left(D^{I}\right)^*D^{I}\eta_l' u\right)+O\left(\|u\|_{B,m}\|u\|_{B,m+1}\right) $. Adding them gives the equality $(\ref{integration by prats for G_Q})$.
	
	The key point in above calculation is that all boundary terms in integration by parts vanish, which is also true for $D_t^I\eta_l''u$ since $ D_t^I $ is tangential. Thus, the parallel calculation of the proof given above implies the conclusion if we replace $ D^I $ and $ \eta_l' $ by $ D_t^I $ and $ \eta_l'' $.
\end{proof}

\begin{lemma}\label{main 4}
	For $ 0\leq r\leq q $, then $ \|u\|_{B,1}^2\lesssim\|F_Bu\|_{B}^2 $ for all $ u\in{\rm Dom}\left(F_B\right) $.
\end{lemma}
\begin{proof}
	Taking into account Corollary \ref{H1e} for $ u\in{\rm Dom}\left(F_B\right) $,
	\begin{align*}
		\|u\|_{B,1}^2\lesssim G_B\left(u,u\right)=\left(F_Bu,u\right)_B\leq\|F_Bu\|_B^2+\|u\|_B^2\leq\|F_Bu\|_B^2.
	\end{align*}
    The last inequality is valid since $ F^{-1} $ is a bounded operator.
\end{proof}

On the basis of these preliminary computations we can prove the estimate (\ref{estimate}) for smooth forms.

\renewcommand{\proofname}{\bf $ Proof\ of\ estimate\ (\ref{estimate}) $}
\begin{proof}
	We prove the desired estimate (\ref{estimate}) using induction on $ s $. For $ s=0 $, when $ u\in\Omega_B^r\left(\overline{M}\right)\cap{\rm Dom}(F_B) $ from $ (\ref{norm}) $ we have
	\begin{align*}
		\|u\|_{B,2}^2\thicksim&\sum_{\alpha=p+1}^{n}\sum_{l=1}^{L_1}\|D_\alpha\eta_l'u\|_1^2+\sum_{\alpha=p+1}^{n-1}\sum_{l=1}^{L_2}\|D_t^\alpha\eta_l''u\|_1^2+\sum_{l=1}^{L_2}\|D_\rho\eta_l''u\|_1^2+\|u\|_{B,1}^2\\
		\thicksim&\sum_{\alpha=p+1}^{n}\sum_{l=1}^{L_1}\|D_\alpha\eta_l' u\|_1^2+\sum_{\alpha=p+1}^{n-1}\sum_{l=1}^{L_2}\|D_t^\alpha\eta_l''u\|_1^2+\sum_{l=1}^{L_2}\|D_\rho^2\eta_l''u\|^2+\|u\|_{B,1}^2.
	\end{align*}
    
    We first handle the third term. Since $ F_B $ is elliptic, then
    \begin{align*}
    	F_Bu=AD_\rho^2u+\sum_{\alpha=p+1}^{n-1}B_\alpha D_t^\alpha D_\rho u+\sum_{\alpha,\beta=p+1}^{n-1}C_{\alpha\beta}D_t^\alpha D_t^\beta u+\cdots,
    \end{align*}
    where $ A,B_\alpha,C_{\alpha\beta} $ are invertible and dots denote lower order terms. Above expression gives us
    \begin{align}\label{ellipticity}
    	D_\rho^2\eta_l''u=\eta_l''A^{-1}F_Bu-\sum_{\alpha=p+1}^{n-1}A^{-1}B_\alpha D_t^\alpha D_\rho\eta_l''u-\sum_{\alpha,\beta=p+1}^{n-1}A^{-1}C_{\alpha\beta}D_t^\alpha D_t^\beta\eta_l''u+\cdots,
    \end{align}
    where dots denote lower order terms. Thus,
    \begin{align}\label{the estimate of D_rho^2}
    	\|D_\rho^2\eta_l''u\|^2\lesssim\|F_Bu\|_B^2+\sum_{\alpha=p+1}^{n-1}\|D_t^\alpha\eta_l''u\|_1^2+\|u\|_{B,1}^2.
    \end{align}
    
    Then for any $ 0<\varepsilon<<1 $, there exists constant $ C_\varepsilon>0 $ such that
    \begin{align*}
        \|u\|_{B,2}^2\lesssim&\sum_{\alpha=p+1}^{n}\sum_{l=1}^{L_1}\|D_\alpha\eta_l' u\|_1^2+\sum_{\alpha=p+1}^{n-1}\sum_{l=1}^{L_2}\|D_t^\alpha\eta_l''u\|_1^2+\|F_Bu\|_B^2+\|u\|_{B,1}^2\\
        \lesssim&\sum_{\alpha=p+1}^{n}\sum_{l=1}^{L_1}G_Q\left(D_\alpha\eta_l' u,D_\alpha\eta_l' u\right)+\sum_{\alpha=p+1}^{n-1}\sum_{l=1}^{L_2}G_Q\left(D_t^{\alpha}\eta_l''u,D_t^{\alpha}\eta_l''u\right)+\|F_Bu\|_B^2+\|u\|_{B,1}^2\\
        \lesssim&\sum_{\alpha=p+1}^{n}\sum_{l=1}^{L_1}G_Q\left(u,\eta_l'\left(D_\alpha\right)^*D_\alpha\eta_l' u\right)+\sum_{\alpha=p+1}^{n-1}\sum_{l=1}^{L_2}G_Q\left(u,\eta_l''\left(D_t^{\alpha}\right)^*D_t^{\alpha}\eta_l''u\right)+\|F_Bu\|_B^2+\|u\|_{B,1}^2\\
        &+O\left(\|u\|_{B,1}\|u\|_{B,2}\right)\\
        =&\sum_{\alpha=p+1}^{n}\sum_{l=1}^{L_1}\left(F_Bu,\eta_l'\left(D_\alpha\right)^*D_\alpha\eta_l' u\right)+\sum_{\alpha=p+1}^{n-1}\sum_{l=1}^{L_2}\left(F_Bu,\eta_l''\left(D_t^{\alpha}\right)^*D_t^{\alpha}\eta_l''u\right)+\|F_Bu\|_B^2+\|u\|_{B,1}^2\\
        &+O\left(\|u\|_{B,1}\|u\|_{B,2}\right)\\
        \lesssim&C_\varepsilon\|F_Bu\|_{B}^2+\varepsilon\|u\|_{B,2}^2+C_\varepsilon\|u\|_{B,1}^2\\
        \lesssim&C_\varepsilon\|F_Bu\|_{B}^2+\varepsilon\|u\|_{B,2}^2.
    \end{align*}
    The second line holds by Lemma $ \ref{main 1} $. The third line is an application of  Lemma $ \ref{main 3} $, and the last line follows from Lemma $ \ref{main 4} $. This proves the estimate for s = 0.
    
    Suppose the theorem is valid for $ s-1 $, for $ u\in\Omega_B^r\left(\overline{M}\right)\cap{\rm Dom}(F_B) $ by $ (\ref{norm}) $ we have
    \begin{align*}
    	\|u\|_{B,s+2}^2\thicksim&\sum_{|I|=s+1}\sum_{l=1}^{L_1}\|D^{I}\eta_l' u\|_1^2+\sum_{|I|=s+1}\sum_{l=1}^{L_2}\|D_t^{I}\eta_l''u\|_1^2+\|u\|_{B,s+1}^2+\sum_{|J|+\iota=s+2,\iota\geq2}\sum_{l=1}^{L_2}\|D_t^{J}D_\rho^\iota\eta_l''u\|^2.
    \end{align*}
    
    We shall handle the first three items. By Lemma $ \ref{main 1} $, Lemma $ \ref{main 3} $ and the inductive hypothesis, for any $ 0<\varepsilon<<1 $ there exists constant $ C_\varepsilon>0 $ such that
    \begin{align*}
    	&\sum_{|I|=s+1}\sum_{l=1}^{L_1}\|D^{I}\eta_l' u\|_1^2+\sum_{|I|=s+1}\sum_{l=1}^{L_2}\|D_t^{I}\eta_l''u\|_1^2+\|u\|_{B,s+1}^2\\
    	\lesssim&\sum_{|I|=s+1}\sum_{l=1}^{L_1}G_Q\left(D^I\eta_l' u,D^I\eta_l' u\right)+\sum_{|I|=s+1}\sum_{l=1}^{L_2}G_Q\left(D_t^{I}\eta_l''u,D_t^{I}\eta_l''u\right)+\|u\|_{B,s+1}^2\\
    	\lesssim&\sum_{|I|=s+1}\sum_{l=1}^{L_1}G_Q\left(u,\eta_l'\left(D^I\right)^*D^I\eta_l' u\right)+\sum_{|I|=s+1}\sum_{l=1}^{L_2}G_Q\left(u,\eta_l''\left(D_t^I\right)^*D_t^{I}\eta_l''u\right)+\|u\|_{B,s+1}^2+O\left(\|u\|_{B,s+1}\|u\|_{B,s+2}\right)\\
    	=&\sum_{|I|=s+1}\sum_{l=1}^{L_1}\left(F_Bu,\eta_l'\left(D^I\right)^*D^I\eta_l' u\right)+\sum_{|I|=s+1}\sum_{l=1}^{L_2}\left(F_Bu,\eta_l''\left(D_t^I\right)^*D_t^{I}\eta_l''u\right)+\|u\|_{B,s+1}^2+O\left(\|u\|_{B,s+1}\|u\|_{B,s+2}\right)\\
    	\lesssim&C_\varepsilon\|F_Bu\|_{B,s}^2+\varepsilon\|u\|_{B,s+2}^2+C_\varepsilon\|u\|_{B,s+1}^2\\
    	\lesssim&C_\varepsilon\|F_Bu\|_{B,s}^2+\varepsilon\|u\|_{B,s+2}^2.
    \end{align*}
    
    Now we shall use induction on $ \iota $ to handle the term $ \sum_{|J|+\iota=s+2,\iota\geq2}\sum_{l=1}^{L_2}\|D_t^{J}D_\rho^\iota\eta_l''u\|^2 $. For $ \iota=2 $, applying $ D_t^J(|J|=s) $ to (\ref{ellipticity}) yields
    \begin{align}\label{estimate of D_t^JD_rho^2}
    	\|D_t^JD_\rho^2\eta_l''u\|^2\lesssim\|F_Bu\|_{B,s}^2+\sum_{|I|=s+1}\|D_t^I\eta_l''u\|_1^2+\|u\|_{B,s+1}^2\lesssim\|F_Bu\|_{B,s}^2,
    \end{align}
    where the second inequality is valid, since the second term can be controlled by the same argument as above and for the third term we use the inductive hypothesis. Then by induction, we suppose the estimate is true for all $ k\leq\iota-1 $ s.t. $ |J|+k=s+2 $, applying $ D_t^JD_\rho^{\iota-2}(|J|+\iota=s+2) $ to (\ref{ellipticity}) we obtain
    \begin{align*}
    	\|D_t^{J}D_\rho^\iota\eta_l''u\|^2\lesssim\|F_Bu\|_{B,s}^2+\sum_{|J|+\iota=s+2}\|D_t^JD_\rho^{\iota-1}\eta_l''u\|^2+\sum_{|J|+\iota=s+2}\|D_t^JD_\rho^{\iota-2}\eta_l''u\|^2+\|u\|_{B,s+1}^2\lesssim\|F_Bu\|_{B,s}^2.
    \end{align*}
    Therefore, summing over $ J,\iota,l $, the desired estimate is established.
\end{proof}

\renewcommand{\proofname}{\bf $ Proof $}

Hence, we have proved the priori estimate when $ u $ is a smooth basic form, and now we are prepared to prove the regularity. By means of Proposition \ref{Rellich and Sobolev lemma}, it suffices to prove the following theorem.

\begin{thm}\label{regularity theorem of F_B^epsilon}
	Let $ 0\leq r\leq q $, if $ u\in{\rm Dom}\left(F_B\right) $ and $ F_B u\in H_{B,s}^r $ for some $ s\in\mathbb{Z}_+ $, then $ u\in H_{B,s+2}^r $.
\end{thm}

In order to prove the regularity, we introduce the technique of difference operators $ \nabla_h^k $ developed by L. Nirenberg (see \cite{Nl55}), we define $ \nabla_h^k $ on $ \mathbb{R}^N $ as follows  
\begin{equation*}
	\nabla_h^k w\left(y\right)=\frac{1}{2h}\left[w\left(y_{1},\dots,y_k+h,\cdots,y_N\right)-w\left(y_{1},\dots,y_k-h,\cdots,y_N\right)\right],
\end{equation*}
where $ w\left(y\right) $ is a function on $ \mathbb{R}^N $. The operator $ \nabla_h^k $ acts on forms componentwise. For multi-index $ \tilde{I}=\left(i_1,\cdots,i_N\right) $, and let $ H=\left(h_{1,1},\cdots,h_{1,i_1},\cdots,h_{N,1},\cdots,h_{N,i_N}\right) $, we define $ \nabla_H^{\tilde{I}}=\Pi_{l=1}^N\Pi_{j=1}^{i_l}\nabla_{h_{l,j}}^l $.  In our case, we always work in a distinguished coordinate chart $ \left(x=\left(x',x''\right),U\right) $, on which basic forms are only depend on $ x'' $. Since we only care about derivatives in transversal direction (i.e., $Q$-direction), it's enough to consider the operator $ \nabla_H^I $ for multi-indices $I$ with first $p$ conponents vanishing. In a normal coordinate chart, $ \nabla_H^I $ is only defined on tangential direction. We need the following lemma which is proved by G. B. Folland and J. J. Kohn (see \cite{FK72}):

\begin{lemma}\label{main 5}
	Let $ \Omega $ be a compact set in either $ \mathbb{R}^N $ or $ \mathbb{R}_-^{N+1} $, and let $ w_1,\ w_2 $ be square-integrable functions supported in $ \Omega $. Let $ D $ be a first order differential operator and let $ |I|=m $. Then, 
	
	{\rm (i)} If $ w_1\in H_m $, $ \|\nabla_H^Iw_1\|\lesssim\|w_1\|_m $ uniformly as $ H\rightarrow0 $.
	
	{\rm (ii)} If $ w_1\in H_m $, $ \|\left[D,\nabla_H^I\right]w_1\|\lesssim\|w_1\|_m $ uniformly as $ H\rightarrow0 $.
	
	{\rm (iii)} If $ w_1\in H_{m-1} $, $ \left(\nabla_H^Iw_1,w_2\right)=\left(w_1,\nabla_H^Iw_2\right)+O\left(\|w_1\|_{m-1}\|w_2\|\right) $ uniformly as $ H\rightarrow0 $.
	
	{\rm (iv)} If $ w_1\in H_{m-1} $ and $ w_2\in H_1 $, then $ |\left(w_1,\left[\nabla_H^I,D\right]v\right)|\lesssim\|w_1\|_{m-1}\|w_2\|_1 $ uniformly as $ H\rightarrow0 $.
	
	{\rm (v)} If $ w_1\in H_{m'} $ and $ \|\nabla_H^Iw_1\|_{m'} $ is bounded as $ H\rightarrow0 $, then $ D^{I}w_1\in H_{m'} $.
\end{lemma}

In this context, $ D $ is the first-order differential operator $ {\rm d}_Q $ or $ \delta_Q $. Since the difference operators can act on forms componentwise, Lemma $ \ref{main 5} $ still hold on forms. It's clear that $ D_t^{I} $ and $ \nabla_H^I $ preserve $ \mathcal{D}_B^r $, $ \nabla_H^I $ preserves $ \tilde{\mathcal{D}}_B^r $ and $ \tilde{\mathcal{D}}_Q^r $, multiplying a smooth function preserves $ \mathcal{D}_Q^r $ and $\tilde{\mathcal{D}}_Q^r$.

To prove Theorem \ref{regularity theorem of F_B^epsilon}, we need the following lemmas.

\begin{lemma}\label{main 7}
	For $ 0\leq r\leq q $, let $ I $ be a multi-index with order $ m $, and let $ \nabla_H^I $ act on the components of forms in $ U_l' $, then for $u\in\tilde{\mathcal{D}}_B^r$
	$$ \|\nabla_H^I\eta_l'u\|_1^2\lesssim G_Q\left(\nabla_H^I\eta_l'u,\nabla_H^I\eta_l'u\right)+\|u\|_{B,m}^2 $$ 
	as $ H\rightarrow0 $ provided that the R.H.S. of above estimate is bounded. In addition, the estimate holds accordingly with $\eta_l'$ replaced by $\eta_l''$.
\end{lemma}
\begin{proof}
	The proof is parallel to that of Lemma \ref{main 1}. We only prove the conclusion in a boundary chart since the interior case is easier. Due to $\nabla_H^I\eta_l''u\in\tilde{\mathcal{D}}_Q^r$, reading the Bochner formula (\ref{Bochner-Kodaira formula}) for $\nabla_H^I\eta_l''u$ as
	\begin{align*}
		G_Q\left(\nabla_H^I\eta_l''u,\nabla_H^I\eta_l''u\right)=\|\nabla_H^I\eta_l''u\|_1^2+O\left(\int_{\partial M}|\nabla_H^I\eta_l''u|^2+\|\nabla_H^I\eta_l''u\|^2\right).
	\end{align*}
	The assertion holds in view of of (i) in Lemma \ref{main 5}, Stokes' theorem and Cauchy-Schwarz inequality.
\end{proof}

\begin{lemma}\label{main 9}
	For $ 0\leq r\leq q $, let $ I $ be a multi-index  with order $ m+1 $, if $ u\in{\rm Dom}\left(F_B\right)\cap H_{B,m+1}^r $ and $ F_B u\in H_{B,m}^r $, then $ \|D^I\eta_l'u\|_1^2+\|D_t^I\eta_l''u\|_1^2<+\infty $.
\end{lemma}
\begin{proof}
	By (v) in Lemma \ref{main 5}, it's enough to show that $ \|\nabla_H^I\eta_l'u\|_1^2+\|\nabla_H^I\eta_l''u\|_1^2<+\infty $ as $ H\rightarrow0 $. According to Lemma \ref{main 7}, we are reduced to proving that $ G_Q\left(\nabla_H^I\eta_l' u,\nabla_H^I\eta_l' u\right)+G_Q\left(\nabla^I_H\eta_l''u,\nabla^I_H\eta_l''u\right) $ is bounded as $ H\rightarrow0 $ for $ |I|=m+1 $.
	\begin{align*}
		\left({\rm d}_Q\nabla_H^I\eta_l' u,{\rm d}_Q\nabla_H^I\eta_l' u\right)&\xlongequal[{\rm (i)\ (ii)}]{Lemma\ \ref{main 5}}\left(\nabla_H^I\eta_l' {\rm d}_Qu,{\rm d}_Q\nabla_H^I\eta_l' u\right)+O\left(\|u\|_{B,m+1}\|\nabla_H^I\eta_l' u\|_1\right)\\
		&\xlongequal[{\rm (iii)}]{Lemma\ \ref{main 5}}\left(\eta_l'{\rm d}_Qu,\nabla_H^I{\rm d}_Q\nabla_H^I\eta_l' u\right)+O\left(\|u\|_{B,m+1}\|\nabla_H^I\eta_l' u\|_1\right)\\
		&\xlongequal[{\rm (iv)}]{Lemma\ \ref{main 5}}\left(\eta_l'{\rm d}_Qu,{\rm d}_Q\nabla_H^I\nabla_H^I\eta_l' u\right)+O\left(\|u\|_{B,m+1}\|\nabla_H^I\eta_l' u\|_1\right)\\
		&\xlongequal[{\rm (i)}]{Lemma\ \ref{main 5}}\left({\rm d}_Qu,{\rm d}_Q\eta_l'\nabla_H^I\nabla_H^I\eta_l' u\right)+O\left(\|u\|_{B,m+1}\|\nabla_H^I\eta_l' u\|_1\right).
	\end{align*}
	Correspondingly,
	$$ \left(\delta_Q\nabla_H^I\eta_l' u,\delta_Q\nabla_H^I\eta_l' u\right)=\left(\delta_Qu,\delta_Q\eta_l'\nabla_H^I\nabla_H^I\eta_l' u\right)+O\left(\|u\|_{B,m+1}\|\nabla_H^I\eta_l' u\|_1\right). $$ 
	Then let $ \nabla_H^I=\nabla_{H'}^{{I}'}\nabla_h^\alpha $, for any $ 0<\varepsilon<<1 $ there exists constant $ C_\varepsilon>0 $ such that
	\begin{align*}
		G_Q\left(\nabla_H^I\eta_l' u,\nabla_H^I\eta_l' u\right)
		=&\left(F_Bu,\eta_l'\nabla_H^I\nabla_H^I\eta_l' u\right)+O\left(\|u\|_{B,m+1}\|\nabla_H^I\eta_l' u\|_1\right)\\
		=&\left(\nabla_{H'}^{{I}'}\eta_l' F_Bu,\nabla_h^\alpha\nabla_H^I\eta_l' u\right)+O\left(\|u\|_{B,m+1}\|\nabla_H^I\eta_l' u\|_1+\|F_Bu\|_{m-1}\|\nabla_h^\alpha\nabla_H^I\eta_l' u\|\right)\\
		=&O\left(\|F_Bu\|_{B,m}\|\nabla_H^I\eta_l' u\|_1+\|u\|_{B,m+1}\|\nabla_H^I\eta_l' u\|_1\right)\\
		\lesssim&C_\varepsilon\|F_Bu\|_{B,m}^2+C_\varepsilon\|u\|_{B,m+1}^2+\varepsilon\|\nabla_H^I\eta_l' u\|_1^2.
	\end{align*}
	The first equality is valid by (\ref{FQFB}). The second and third equality follow from (iii) and (i) in Lemma \ref{main 5} respectively.
\end{proof}

Now, we are in a position to complete the proof of Thoerem \ref{regularity theorem of F_B^epsilon}.
\renewcommand{\proofname}{\bf $ Proof\ of\ Theorem\ \ref{regularity theorem of F_B^epsilon} $}
\begin{proof}
	We prove the theorem using induction on $ s $. For $ s=0 $, from $ (\ref{norm}) $ we have
	\begin{align*}
		\|u\|_{B,2}^2\thicksim\sum_{\alpha=p+1}^{n}\sum_{l=1}^{L_1}\|D_\alpha\eta_l' u\|_1^2+\sum_{\alpha=p+1}^{n-1}\sum_{l=1}^{L_2}\|D_t^\alpha\eta_l''u\|_1^2+\sum_{l=1}^{L_2}\|D_\rho^2\eta_l''u\|^2+\|u\|_{B,1}^2.
	\end{align*}
    Owing to estimate (\ref{the estimate of D_rho^2}) and Lemma \ref{main 9}, then
    \begin{align*}
    	\|u\|_{B,2}^2\lesssim\sum_{\alpha=p+1}^{n}\sum_{l=1}^{L_1}\|D_\alpha\eta_l' u\|_1^2+\sum_{\alpha=p+1}^{n-1}\sum_{l=1}^{L_2}\|D_t^\alpha\eta_l''u\|_1^2+\|F_B^\epsilon u\|_B^2+\|u\|_{B,1}^2<+\infty.
    \end{align*}

    Suppose the theorem holds for $ s-1 $, for $ u\in{\rm Dom}(F_B) $ by $ (\ref{norm}) $ we get
    \begin{align*}
    	\|u\|_{B,s+2}^2\thicksim&\sum_{|I|=s+1}\sum_{l=1}^{L_1}\|D^{I}\eta_l' u\|_1^2+\sum_{|I|=s+1}\sum_{l=1}^{L_2}\|D_t^{I}\eta_l''u\|_1^2+\|u\|_{B,s+1}^2+\sum_{|J|+\iota=s+2,\iota\geq2}\sum_{l=1}^{L_2}\|D_t^{J}D_\rho^\iota\eta_l''u\|^2.
    \end{align*}
    Thanks to Lemma \ref{main 9} and inductive hypothesis, we only need to prove the last term above is bounded. One can use precisely the same argument as in that proof of estimate (\ref{estimate}), we find that $ \|D_t^{J}D_\rho^2\eta_l''u\|^2 $ is bounded by the estimate (\ref{estimate of D_t^JD_rho^2}). Proceeding by induction on $ \iota $, applying $ D_t^JD_\rho^{\iota-2}(|J|+\iota=s+2) $ to (\ref{ellipticity}), it follows that $ \|D_t^{J}D_\rho^\iota\eta_l''u\|^2 $ can be controlled by the norms of derivatives of $ F_B u $ of order $ s $ and the norms of derivatives of $ u $ which are already estimated.
\end{proof}

We have established the regularity of the equation $ F_Bu=\omega $, before going to prove the Hodge decomposition theorem, we need more properties of the operators $F_B$ and $\Delta_{F_B}:=\Delta_B|_{{\rm Dom}\left(F_B\right)}$.

\renewcommand{\proofname}{\bf $ Proof $}
\begin{prop}\label{higher estimate}
	Let $ 0\leq r\leq q $, for the equation $ F_Bu=\omega $, if $ \omega\in H_{B,s}^r $, then $ u\in H_{B,s+2}^r $ and $ \|u\|_{B,s+2}\lesssim\|\omega\|_{B,s} $. 
\end{prop}
\begin{proof}
	The conclusion just an application of Theorem $ \ref{main theorem} $ through approximating $ \omega $ by smooth basic forms.
\end{proof}

\begin{prop}\label{smoothness of eigenform}
	For $ 0\leq r\leq q $, there is a complete orthonormal basis of eigenforms for $ \Delta_{F_B} $ in $ H_{B,0}^r $ which are smooth up to the boundary. The spectrum is discrete, non-negative, with no finite limit point, and each eigenvalue occurs with finite multiplicity. In particular, the kernel $ \mathcal{H}_B^r $ of $ \Delta_{F_B} $ is a finite-dimensional subspace of $ \Omega_B^r\left(\overline{M}\right) $. Moreover, for each integer s, we have
	$$ \|u\|_{B,s+2}^2\lesssim\|\Delta_Bu\|_{B,s}^2+\|u\|_B^2 $$
	uniformly for all $ u\in{\rm Dom}\left(F_B\right)\cap\Omega_B^r\left(\overline{M}\right) $.
\end{prop}
\begin{proof}
	Since $ F_B=\Delta_B+{\rm Id} $ on smooth forms, we will analyze the spectrum of $ F_B $ at first. It follows from Proposition $ \ref{higher estimate} $ and Proposition $ \ref{Rellich and Sobolev lemma} $ that $ F_B^{-1}:H_{B,0}^r\longrightarrow H_{B,2}^r\longrightarrow H_{B,0}^r $ is compact, because the first map is bounded and the second is compact. Thus, the spectrum of $ F_B $ is discrete, non-negative, with no finite limit point, and each eigenvalue occurs with finite multiplicity. Furthermore, the eigenforms can form a complete orthonormal basis of $ H_{B,0}^r $, in view of the theory of self-adjoint compact operators. 
	
	Next, we need to show that the eigenforms of $ F_B $ are all smooth up the boundary. In fact, if $ u\in H_{B,0}^r $ is an eigenform with repect to eigenvalue $ \lambda $, then $ \left(F_B-\lambda\right)u=0 $. Set $ \omega=\lambda u $ so that $ F_Bu=\omega $, then since $ \omega\in H_{B,0}^r $, we have $ u\in H_{B,2}^r $ according to Proposition $ \ref{higher estimate} $. Hence, $ \omega\in H_{B,2}^r $ and we use Proposition $ \ref{higher estimate} $ inductively, then $ u\in\Omega_B^r\left(\overline{M}\right) $ by Proposition $ \ref{Rellich and Sobolev lemma} $. It means that $ \left(F_B-\lambda\right)u $ is hypoelliptic.
	
	Finally, we prove the estimate of $ \Delta_B $ by induction on $ s $ by means of Theorem $ \ref{main theorem} $. For $ s=0 $, $ \|u\|_{B,2}^2\lesssim\|F_Bu\|_B^2\lesssim \|\Delta_Bu\|_B^2+\|u\|_B^2 $, assume that the conclusion is valid for all $ k\leq s+1 $, then for $ s+2 $ we have 
	\begin{align*}
		\|u\|_{B,s+2}^2\lesssim\|F_Bu\|_{B,s}^2\lesssim\|\Delta_Bu\|_{B,s}^2+\|u\|_{B,s}^2\lesssim\|\Delta_Bu\|_{B,s}^2+\|\Delta_Bu\|_{B,s-2}^2+\|u\|_B^2\lesssim\|\Delta_Bu\|_{B,s}^2+\|u\|_B^2.
	\end{align*}
	The proof is thus complete.
\end{proof}

On the basis of these existence and regularity results for $F_B$ and $\Delta_{F_B}$ we now can get back to the main business at hand, namely the de Rham-Hodge decomposition for Riemannian foliations.
\begin{thm}\label{Hodge decomposition theorem}
	For $ 0\leq r\leq q $, let $ \mathcal{F} $ be a transversally oriented Riemannian foliation on a compact oriented manifold $ \overline{M} $ with smooth basic boundary $ \partial M $, i.e., $ \rho $ is a basic function. Assume $ g_M $ is a bundle-like metric with projectable mean curvature $ \tau $. There is an orthogonal decomposition
	\begin{equation}\label{Hodge decomposition}
		H_{B,0}^r={\rm Range}\left(\Delta_{F_B}\right)\oplus\mathcal{H}_B^r={\rm d}_B\delta_B{\rm Dom}\left(F_B\right)\oplus\delta_B{\rm d}_B{\rm Dom}\left(F_B\right)\oplus\mathcal{H}_B^r,
	\end{equation}
    where the kernel $ \mathcal{H}_B^r $ of $ \Delta_{F_B} $ is a finite-dimensional space. Moreover, we have
    \begin{equation}\label{smooth Hodge decomposition}
    	\Omega_B^r\left(\overline{M}\right)={\rm Im}{\rm d}_B\oplus{\rm Im}\delta_B\oplus\mathcal{H}_B^r.
    \end{equation}
\end{thm}
\begin{proof}
	Owing to Proposition $ \ref{smoothness of eigenform} $, the spectrum is discrete, then the eigenvalues of $ \Delta_B|_{\left(\mathcal{H}_{B}^r\right)^\perp} $ are positive, and $ \|u\|_B\lesssim\|\Delta_Bu\|_B $ hold uniformly for all $ u\in{\rm Dom}\left(\Delta_{F_B}\right)\cap\left(\mathcal{H}_{B}^r\right)^\perp $. Hence the range of $ \Delta_{F_B} $ is closed, which gives the first equality of ($ \ref{Hodge decomposition} $) since $ \Delta_{F_B} $ is a self-adjoint operator. By $ {\rm d}_B^2=0 $ we know that Range($ {\rm d}_B $) $ \perp $ Range($ \delta_B $), it implies the second equality. At last, the decomposition ($ \ref{smooth Hodge decomposition} $) is valid due to Theorem $ \ref{main theorem} $.
\end{proof}

The decomposition (\ref{Hodge decomposition}) yields a linear map (Green's operator associated to $ \Delta_{F_B} $) $N_B:H_{B,0}\longrightarrow{\rm Dom}\left(F_B\right)$ as follows. If $ \omega\in\mathcal{H}_B^r $, then $ N_B\omega=0 $; if $ \omega\in{\rm Range}\left(\Delta_{F_B}\right) $, then $ N_B\omega $ is the unique solution of $ \Delta_{F_B}u=\omega $ such that $ N_B\omega\perp\mathcal{H}_B^r $. We denote by $ P_B $ the orthogonal projection operator of $H_{B,0}^r$ onto $\mathcal{H}_B^r$. The properties of $ N_B $ is recorded in the following proposition.

\begin{prop}\label{proposition of N_B}
	$ N_B $ is a bounded operator such that
	
	{\rm (i)} If $ \omega\in H_{B,0}^r $, then $ \omega={\rm d}_B\delta_BN_B\omega+\delta_B{\rm d}_BN_B\omega+P_B\omega $.
	
	{\rm (ii)} $ N_BP_B=P_BN_B=0 $, and $ N_B\Delta_B=\Delta_BN_B={\rm Id}-P_B $ on $ {\rm Dom}\left(F_B\right) $.
	
	{\rm (iii)} If $ \omega\in H_{B,0}^r $ satisfying $ {\rm d}_B\omega=P_B\omega=0 $, then there is a unique solution $ u:=\delta_BN_B\omega $ of the equation $ {\rm d}_Bu=\omega $ such that $ u\perp {\rm Ker}\left({\rm d}_B\right) $.
\end{prop}

\renewcommand{\proofname}{\bf $ Proof $}
\begin{proof}
	The closeness of $ {\rm Range}\left(\Delta_{F_B}\right) $ means that $ \|u\|_B\lesssim\|\Delta_Bu\|_B $ uniformly for all $ u\in{\rm Dom}\left(\Delta_{F_B}\right)\cap{\mathcal{H}_B^r}^\perp $, it turns out that $ \|N_B\omega\|_B\lesssim\|\omega\|_B $ for all $ \omega\in H_{B,0}^r $ which indicates that $ N_B $ is a bounded operator. (i) is a restatement of the orthogonal decomposition. And (ii) is clear from the definition of $ N_B $. Then by (i) we have $ \omega={\rm d}_B\delta_BN_B\omega $ since $ {\rm d}_B\omega=0 $ and $ \omega\perp\mathcal{H}_B^r $. Thus $ u:=\delta_BN_B\omega $ is the unique solution since $ \left(\delta_BN_B\omega,v\right)_B=\left(N_B\omega,{\rm d}_Bv\right)_B=0 $ for any $ v\in{\rm Ker}\left({\rm d}_B\right) $.
\end{proof}

\begin{remark}
	If $\partial M =\emptyset$, then the proposition is proved by the heat equation method in \cite{NRT90}.
\end{remark}

We now turn towards studying the regularity of the equation $ {\rm d}_Bu=\omega $ under the conditions in Theorem \ref{Hodge decomposition theorem} as follows.

\begin{cor}\label{existence theorem 1}
	With the same hypotheses of Theorem \ref{Hodge decomposition theorem}. For $ 1\leq r\leq q $, let $ \omega\in H_{B,0}^r $ satisfying $ {\rm d}_B\omega=0 $ and $ \omega\perp\mathcal{H}_B^r $, then there exists a unique solution $ u $ of the equation $ {\rm d}_Bu=\omega $ which is orthogonal to $ {\rm Ker}\left({\rm d}_B\right) $. Moreover, for any $ s\in\mathbb{Z}_+ $, if $ \omega\in\Omega_B^r\left(\overline{M}\right) $, then $ u\in\Omega_B^{r-1}\left(\overline{M}\right) $ and
	\begin{align}\label{estimate of solution of d_B}
		\|u\|_{B,s+1}^2\lesssim\|\omega\|_{B,s}^2.
	\end{align}
\end{cor}

\begin{proof}
	The existence and uniqueness of the equation $ {\rm d}_Bu=\omega $ is given by (iii) in Propsotion \ref{proposition of N_B}. For the regularity, it suffices to prove that $ N_B $ preserves the smoothness. Using (ii) in the Propsotion \ref{proposition of N_B}, we have
	\begin{align*}
		\left(F_B-{\rm Id}\right)N_B\omega=\Delta_BN_B\omega=\left({\rm Id}-P_B\right)\omega,
	\end{align*}
	then $ N_B\omega\in\Omega_B^r\left(\overline{M}\right) $ when $ \omega\in\Omega_B^r\left(\overline{M}\right) $ since $ F_B-{\rm Id} $ is hypoelliptic by the proof of Proposition \ref{smoothness of eigenform}. For the estimate $ (\ref{estimate of solution of d_B}) $, we only need to show that $ \|N_B\omega\|_{B,s+2}^2\lesssim\|\omega\|_{B,s}^2 $. Indeed, according to (ii) in Propsotion \ref{proposition of N_B} and Proposition \ref{smoothness of eigenform} again,
	\begin{align*}
		\|N_B\omega\|_{B,s+2}^2&\lesssim\|\Delta_BN_B\omega\|_{B,s}^2+\|N_B\omega\|_{B}^2\\
		&\lesssim\|\omega\|_{B,s}^2+\|P_B\omega\|_{B,s}^2+\|N_B\omega\|_{B}^2\\
		&\thicksim\|\omega\|_{B,s}^2+\|P_B\omega\|_{B}^2+\|N_B\omega\|_{B}^2\\
		&\lesssim\|\omega\|_{B,s}^2.
	\end{align*}
	The third line is true since $ \mathcal{H}_B^r $ is a finite-dimensional space by Theorem \ref{Hodge decomposition theorem}. The last line follows from the continuity of $ N_B $ and $P_B$. Hence, 
	$$ \|u\|_{B,s+1}^2\lesssim\|N_B\omega\|_{B,s+2}^2\lesssim\|\omega\|_{B,s}^2. $$
	Thus the proof is complete.
\end{proof}

We finish this section by observing the finite dimensionality of basic cohomology. In fact, thanks to Corollary \ref{existence theorem 1} we know that
$$ \mathcal{H}_B^r\cong H_B^r\left(\overline{M},\mathcal{F}\right):=\frac{\{u\in\Omega_B^r\left(\overline{M}\right)\ |\ {\rm d}_Bu=0\}}{{\rm d}_B\Omega_B^{r-1}\left(\overline{M}\right)} $$
is a finite dimensional space.

\subsection{Existence theorems}
This section is devoted to giving two existence theorems by $L^2$-method for the equation ${\rm d}_Bu=\omega$. In the framework of $L^2$ theory, one needs the positivity to ensure that the priori estimate holds true. To do so, we first give the following definition in terms of a local orthonormal frame $ \lbrace L_1,\cdots,L_q\rbrace $ of $ Q $ with the dual basis $ \lbrace \check{\theta}^1,\cdots,\check{\theta}^q\rbrace $ for convenience.
%

\begin{definition}
	A function $ \varphi\in C^2\left(M\right) $ is said to be transversely $ r $-convex if the transversal quadratic form $ L_{r,\varphi}\left(v,v\right):=\varphi_{\alpha\beta}v_{\alpha K}v_{\beta K}\geq0 $ on $ M $ for all $ v=v_J\check{\theta}^J\in\Omega_B^r\left(\overline{M}\right) $ and all multi-indices $ K $ with $ |K|=r-1 $. We call $ \varphi $ is strictly transversely $ r $-convex if $ L_{r,\varphi}>0 $ on $ M $. 
\end{definition}

\begin{remark}
	The definition is independent of the choices of the local frames. In fact, it only depends on any sum of $ r $ eigenvalues of the transversal quadratic form.
\end{remark}

Now we are ready to formulate the vanishing theorem for the basic cohomology $ H_B^\cdot\left(\overline{M},\mathcal{F}\right) $.

\begin{thm}\label{existence theorem 2}
	Under the same conditions of Theorem \ref{Hodge decomposition theorem}. For $ 1\leq r\leq q $, we further assume that there is a strictly transversal $ r $-convex function $ \varphi\in\Omega_B^0\left(\overline{M}\right) $. Then for any $ \omega\in\Omega_B^r\left(\overline{M}\right) $ with $ {\rm d}_B\omega=0 $, there exists $ u\in\Omega_B^{r-1}\left(\overline{M}\right) $ such that $ {\rm d}_Bu=\omega $ and the estimate $ (\ref{estimate of solution of d_B}) $ is also satisfied for every $ s\in\mathbb{Z}_+ $.
\end{thm}

\renewcommand{\proofname}{\bf $ Proof $}
\begin{proof}
	By Corollary \ref{existence theorem 1}, it suffices to prove that $ \mathcal{H}_B^r=\{0\} $. Then the solution is the canonical solution $ u:=\delta_BN_B\omega $. Without loss of generality, we assume that $ \varphi<0 $ in $ M $ and $ |\varphi| $ is bounded since $ \overline{M} $ is compact. We define a new bundle-like metric on $ M $ as follows
	\begin{align}\label{new metric}
		\hat{g}_M:=e^{-\frac{2}{p}\varphi}g_L\oplus g_Q.
	\end{align}
	This is a conformal modification of the metric along the leaves, while the transversal metric is left unchanged. The notation $ \hat{\cdot} $ refers to metric objects in the changed metric. Since $ g_Q $ is unchanged, the inner product $ \lrcorner $ and the connection $ \nabla $ thus $ R^\nabla $ and $ Ric^\nabla $ will not change when acting on the basic forms. Hence we still use the same symbols to denote them. However, by $ (\ref{expression of some Gamma}) $ in Proposition \ref{coefficients of conection} we have
	\begin{align}\label{H'}
		\hat{\tau}=\hat{g}^{ab}\pi\left(\nabla_{X_a}X_b\right)&=\hat{g}^{ab}\hat{\Gamma}_{ab}^\alpha X_\alpha\nonumber\\
		&=-\frac{1}{2}\hat{g}^{ab}\hat{g}^{\alpha\beta}\left(X_\beta \hat{g}_{ab}+\hat{g}_{ac}\frac{\partial\phi_\beta^c}{\partial x_b}+\hat{g}_{bc}\frac{\partial\phi_\beta^c}{\partial x_a}\right)X_\alpha\nonumber\\
		&=-\frac{1}{2}g^{ab}g^{\alpha\beta}\left(-\frac{2}{p}g_{ab}X_\beta\varphi+X_\beta g_{ab}+g_{ac}\frac{\partial\phi_\beta^c}{\partial x_b}+g_{bc}\frac{\partial\phi_\beta^c}{\partial x_a}\right)X_\alpha\nonumber\\
		&=g^{\alpha\beta}X_\beta\left(\varphi\right)X_\alpha-\frac{1}{2}g^{ab}g^{\alpha\beta}\left(X_\beta g_{ab}+g_{ac}\frac{\partial\phi_\beta^c}{\partial x_b}+g_{bc}\frac{\partial\phi_\beta^c}{\partial x_a}\right)X_\alpha\nonumber\\
		&={\rm grad}\varphi+\tau.
	\end{align}
	
	Since $ \varphi $ is a basic function, $ \hat{\delta}_B:=-g^{\alpha\beta}X_\alpha\lrcorner\nabla_{X_\beta}+\hat{\tau}\lrcorner $ maps $ \Omega_B^r\left(\overline{M}\right) $ to $ \Omega_B^{r-1}\left(\overline{M}\right) $. In view of $ (\ref{H'}) $
	\begin{align*}
		\nabla_{X_\gamma}\hat{\tau}&=g^{\alpha\beta}\left<\nabla_{X_\gamma}{\rm grad}\varphi,X_\beta\right>X_\alpha+\nabla_{X_\gamma}\tau\\
		&=g^{\alpha\beta}\left<\nabla_{X_\gamma}\left(g^{\sigma\lambda}X_\sigma\left(\varphi\right) X_\lambda\right),X_\beta\right>X_\alpha+\nabla_{X_\gamma}\tau\\
		&=g^{\alpha\beta}\left<X_{\gamma}\left(g^{\sigma\lambda}X_\sigma\left(\varphi\right)\right) X_\lambda,X_\beta\right>X_\alpha+g^{\alpha\beta}\left<g^{\sigma\lambda}X_\sigma\left(\varphi\right)\nabla_{X_\gamma}X_\lambda,X_\beta\right>X_\alpha+\nabla_{X_\gamma}\tau\\
		&=g^{\alpha\beta}X_\gamma X_\beta\left(\varphi\right)X_\alpha-g^{\alpha\beta}\Gamma_{\gamma\beta}^\lambda	X_\lambda\left(\varphi\right)X_\alpha+\nabla_{X_\gamma}\tau\\
		&=g^{\alpha\beta}\nabla_{X_\beta,X_\gamma}^2\varphi X_\alpha+\nabla_{X_\gamma}\tau.
	\end{align*}
	Thus,
	\begin{align}\label{kappa^}
		\mathcal{A}_{\hat{\tau}}=\mathcal{A}_{\tau}+\theta^\gamma\wedge g^{\alpha\beta}\nabla_{X_\beta,X_\gamma}^2\varphi X_\alpha\lrcorner.
	\end{align}
	
	Let $ {\rm d}V $ be the volume form assosicted with $ g_M $, then $ {\rm d}\hat{V}=e^{-\varphi}{\rm d}V $, and it is easy to see that 
	$$ \hat{\delta}_B=e^\varphi\delta_Be^{-\varphi}=\delta_B+{\rm grad}\varphi\lrcorner. $$
	Hence, by $ (\ref{kappa^}) $ and Bochner formula $ (\ref{Bochner-Kodaira formula}) $ in terms of the local orthonormal frame $ \lbrace L_1,\cdots,L_q\rbrace $ of $ Q $ introduced at the beginning of this section, we know that for $ v=v_J\check{\theta}^J\in\mathcal{H}_B^r $
	\begin{align}\label{bochner for harmonic form and weight varphi}
		\sum_{|I|=r-1}\int_{M}|L_\alpha\left(\varphi\right) v_{\alpha I}|^2e^{-\varphi}=&\|\nabla v\|'^2+\sum_{|I|=r-1}\int_{M}R^\nabla_{L_\alpha,L_\beta}v_{\alpha I}v_{\beta I}e^{-\varphi}+\int_{M}\left<Ric^\nabla\left(L_\alpha\right)\lrcorner v,V_\alpha\lrcorner v\right>e^{-\varphi}\nonumber\\
		&+\int_{M}\left<\mathcal{A}_\tau v,v\right>e^{-\varphi}+\sum_{|I|=r-1}\int_{M}\varphi_{\alpha\beta}v_{\alpha I}v_{\beta I}e^{-\varphi}+\sum_{|I|=r-1}\int_{\partial M}\rho_{\alpha\beta}v_{\alpha I}v_{\beta I}e^{-\varphi},
	\end{align}
	where $ \|\cdot\|' $ is the norm associated to metric $ \hat{g}_M $. 
	
	Let $ \chi_1\left(t\right)=-{\rm log}\left(-t\right) $ for $ t<0 $, we have
	\begin{align*}
		\left(\chi_1\left(\varphi\right)\right)_{\alpha\beta}v_{\alpha I}v_{\beta I}-|L_\alpha\left(\chi_1\left(\varphi\right)\right) v_{\alpha I}|^2&=\chi_1'\left(\varphi\right)\varphi_{\alpha\beta}v_{\alpha I}v_{\beta I}+\chi_1''\left(\varphi\right)|L_\alpha\left(\varphi\right)v_{\alpha I}|^2-\chi_1'\left(\varphi\right)^2|L_\alpha\left(\varphi\right) v_{\alpha I}|^2\nonumber\\
		&=\chi_1'\left(\varphi\right)\varphi_{\alpha\beta}v_{\alpha I}v_{\beta I}.
	\end{align*}
	Using $ \chi_1\left(\varphi\right) $ to replace $ \varphi $ in the metric $ (\ref{new metric}) $, by means of the equality $ (\ref{bochner for harmonic form and weight varphi}) $ and a similar reasoning of Corollary \ref{H1e} that
	\begin{align}\label{bochner for harmonic form and weight chi_1varphi}
		0\geq&\sum_{|I|=r-1}\int_{M}R^\nabla_{L_\alpha,L_\beta}v_{\alpha I}v_{\beta I}e^{-\chi_1\left(\varphi\right)}+\int_{M}\left<Ric^\nabla\left(L_\alpha\right)\lrcorner v,L_\alpha\lrcorner v\right>e^{-\chi_1\left(\varphi\right)}+\int_{M}\left<\mathcal{A}_\tau v,v\right>e^{-\chi_1\left(\varphi\right)}\nonumber\\
		&+\sum_{|I|=r-1}\int_{M}\chi_1'\left(\varphi\right)\varphi_{\alpha\beta}v_{\alpha I}v_{\beta I}e^{-\chi_1\left(\varphi\right)}+O\left(\|v\|'^2\right).
	\end{align}
	
	Since $ \overline{M} $ is compact, $ R^\nabla_{X_\beta,X_\gamma} $, $ Ric^\nabla $ and $ \mathcal{A}_\tau $ are bounded. Let $ \chi_2\left(t\right)=A_1t $ for $ t<0 $ where $ A_1>>1 $. Then $ \chi_2\left(\varphi\right) $ is still a negative strictly transversal $ r $-convex basic function. Substituting $ \chi_2\left(\varphi\right) $ for $ \varphi $ in $ \chi_1\left(\varphi\right) $, and from $ (\ref{bochner for harmonic form and weight chi_1varphi}) $ we read off 
	\begin{align*}
		0\gtrsim\int_{M}|v|^2e^{-\chi_1\circ\chi_2\left(\varphi\right)}.
	\end{align*}
    Hence $ v=0 $, i.e., $ \mathcal{H}_B^r=\{0\} $ which completes the proof.
\end{proof}

Before giving the statement of the next existence theorem, let's first recall a variant of the Riesz representation theorem(see \cite{Hlv65JL}) which is very useful for solving overdetermined systems.
\begin{lemma}\label{functional analysis}
	Let $ T:H_1\longrightarrow H_2 $ and $ S:H_2\longrightarrow H_3 $ be closed, densely defined linear operators, if $ {\rm Im}T\subseteq {\rm Ker}S $ and there is a constant $ C>0 $ such that 	
	\begin{equation*}
		\|g\|_{H_2}^2\leq C\left(\|T^*g\|_{H_1}^2+\|Sg\|_{H_3}^2\right),\quad  g\in D_S\cap D_{T^*},
	\end{equation*}
	then for each $ f\in {\rm Ker}S $, there is a solution $ u\in H_1 $ such that $ Tu=f $ and $ \|u\|_{H_1}^2\leq C\|f\|_{H_2}^2 $.
\end{lemma}

For $ 0\leq r\leq q $, we use $ L^2_{loc}\left(M,\Omega_B^r\right) $ to denote the set of basic $ r $-currents with locally $ L^2 $-integrable coefficients in $ M $. The weighted $ L^2 $-space $ L^2\left(\Omega_B^r,\phi\right) $ for a smooth real valued function $ \phi $ is defined by
\begin{align}\label{weighted L2 space}
	L^2\left(\Omega_B^r,\phi\right)=\bigg\{u\in L^2_{loc}\left(M,\Omega_B^r\right)\ \big|\ \int_M |u|^2e^{-\phi}<+\infty\bigg\}.
\end{align}
Let $\left(\cdot,\cdot\right)_{B,\phi}$ denote the inner product of $ L^2\left(\Omega_B^r,\phi\right) $. In slight abuse of notation we write $ {\rm d}_B $ for its maximal extension from $ L^2\left(\Omega_B^{\cdot},\phi\right) $ to $ L^2\left(\Omega_B^{\cdot+1},\phi\right) $ with Hilbert adjoint $ {{\rm d}_B^\phi}^* $. By $ \delta_B^\phi:=\delta_B+{\rm grad}\phi\lrcorner $ we denote the formal adjoint of $ {\rm d}_B $ with respect to the inner product $\left(\cdot,\cdot\right)_{B,\phi}$. It turns out that 
$$ \mathcal{D}_{B,\phi}^r:={\rm Dom}\left({{\rm d}_B^\phi}^*\right)\cap\Omega_B^r\left(\overline{M}\right)=\mathcal{D}_B^r. $$

We now turn to the problem of $ L^2 $-existence for the operator $ {\rm d}_B $.

\begin{thm}\label{existence theorem 3}
	For $ 1\leq r\leq q $, let $ M $ be an oriented manifold with a transversally oriented Riemannian foliation, and $ g_M $ is a bundle-like metric with basic mean curvature form. If there exists a strictly transversal $ r $-convex exhaustion basic function $ \varphi $, then for any $ \omega\in L^2_{loc}\left(M,\Omega_B^r\right) $ with $ {\rm d}_B\omega=0 $, there is a basic form $ u\in L^2_{loc}\left(M,\Omega_B^{r-1}\right) $ such that $ {\rm d}_Bu=\omega $.
\end{thm}

\renewcommand{\proofname}{\bf $ Proof $}
\begin{proof}
	The proof is parallel to that of Theorem \ref{existence theorem 2}. We define a new bundle-like metric on $ M $ as $ (\ref{new metric}) $. The notation $ \hat{\cdot} $ refers to metric objects in the changed metric while the relationships of them are clear in the proof of Theorem \ref{existence theorem 2}.
	
	We assume that $ \overline{M} $ is a compact manifold with smooth compact boundary $ \partial M $ at first. By means of $ {\rm d}\hat{V}=e^{-\varphi}{\rm d}V $, we can treat the inner product $ \left(\cdot,\cdot\right)' $ associated to the metric $ \hat{g}_M $ as the inner product of the weighted $ L^2 $-space $ L^2\left(\Omega_B^r,\varphi\right) $. It follows that the weighted Bochner formula is same as the Bochner formula in the new metric $ (\ref{new metric}) $. Thus by $ (\ref{kappa^}) $ and Bochner formula $ (\ref{Bochner-Kodaira formula}) $ in terms of the local orthonormal frame $ \lbrace L_1,\cdots,L_q\rbrace $ of $ Q $ introduced at the beginning of this section, we know that for $ v=v_J\check{\theta}^J\in\Omega_B^r\left(\overline{M}\right)\cap{\rm Dom}\left({{\rm d}_B^{\varphi}}^*\right) $
	\begin{align}\label{bochner for varphi'}
		\|{\rm d}_Bv\|_\varphi^2+\|\delta_B^\varphi v\|_\varphi^2\geq&\int_{M}R^\nabla_{L_\alpha,L_\beta}v_{\alpha I}v_{\beta I}e^{-\varphi}+\int_{M}\left<Ric^\nabla\left(L_\alpha\right)\lrcorner v,L_\alpha\lrcorner v\right>e^{-\varphi}\nonumber\\
		&+\int_{M}\left<\mathcal{A}_\tau v,v\right>e^{-\varphi}+\int_{M}\varphi_{\alpha\beta}v_{\alpha I}v_{\beta I}e^{-\varphi}+O\left(\|v\|_\varphi^2\right),
	\end{align}
	where $ I $ runs over all multi-indices with $ |I|=r-1 $. 
	
	Since the operators $ R^\nabla_{X_\beta,X_\gamma} $, $ Ric^\nabla $ and $ \mathcal{A}_\tau $ are bounded, it allows us to choose a convex increasing function $ \chi\in C^\infty\left(\mathbb{R}\right) $ such that
	\begin{align}\label{chi}
		R^\nabla_{L_\alpha,L_\beta}v_{\alpha I}v_{\beta I}+\left<Ric^\nabla\left(L_\alpha\right)\lrcorner v,L_\alpha\lrcorner v\right>+\left<\mathcal{A}_\tau v,v\right>+O\left(\|v\|_\varphi^2\right)+\chi'\left(\varphi\right)\chi\left(\varphi\right)_{\alpha\beta}v_{\alpha I}v_{\beta I}\geq|v|^2.
	\end{align}
	From the estimate (\ref{bochner for varphi'}) for $ \chi\left(\varphi\right) $ we read off
	\begin{align}\label{hormander's priori estimate}
		\|{\rm d}_Bv\|_{\chi\left(\varphi\right)}^2+\|\delta_B^{\chi\left(\varphi\right)}v\|_{\chi\left(\varphi\right)}^2\geq\|v\|_{\chi\left(\varphi\right)}^2.
	\end{align}
	According to L. H$ \ddot{{\rm o}} $rmander's density lemma (see \cite{Hlv03}), the priori estimate $ (\ref{hormander's priori estimate}) $ holds for any $ v\in{\rm Dom}\left({{\rm d}_B^{\chi\left(\varphi\right)}}^*\right) $, then applying Lemma \ref{functional analysis} to Hilbert spaces
	\begin{equation*}
		H_1=L^2\left(\Omega_B^{r-1},\chi\left(\varphi\right)\right), \quad H_2=L^2\left(\Omega_B^r,\chi\left(\varphi\right)\right), \quad H_3 =L^2\left(\Omega_B^{r+1},\chi\left(\varphi\right)\right),
	\end{equation*}
	and $ S $, $ T $ are both given by the maximal extension of $ {\rm d}_B $. It yields that there is a basic form $ u\in L^2\left(\Omega_B^{r-1},\chi\left(\varphi\right)\right) $ such that $ {\rm d}_Bu=\omega $ and $ \|u\|_{\chi\left(\varphi\right)}^2\leq\|\omega\|_{\chi\left(\varphi\right)}^2 $.

	We proceed with considering the general case. Since $ \omega\in L^2_{loc}\left(M,\Omega_B^r\right) $, and $ R^\nabla_{X_\beta,X_\gamma} $, $ Ric^\nabla $ and $ \mathcal{A}_\tau $ are continious operators of zero order on $M$, we can choose a convex increasing function $ \tilde{\chi}\in C^\infty\left(\mathbb{R}\right) $ such that (\ref{chi}) holds for $\tilde{\chi}$ and $ \omega\in L^2\left(\Omega_B^r,\tilde{\chi}\left(\varphi\right)\right) $. On the other hand, since there exists a strictly transversal $ r $-convex exhaustion basic function $ \varphi $, $ M $ can be exhausted by open subsets $ M_\nu $ with smooth compact boundary $ \partial M_\nu $ ($ \nu\in\mathbb{Z}_+ $). Then we can solve the equation $ {\rm d}_Bu=\omega $ in $ M_\nu $ for each $ \nu $ such that
	\begin{align*}
		\int_{M_\nu}|u_\nu|^2e^{-\tilde{\chi}\left(\varphi\right)}\leq\int_{M_\nu}|\omega|^2e^{-\tilde{\chi}\left(\varphi\right)}.
	\end{align*}
	Moreover, the above estimate enables us to find a weak limit $ u $ of $ u_\nu $, which is the desired solution.
\end{proof}

\subsection{Duality and extension theorems}
In this section, we introduce the induced boundary complex on the boundary $\partial M$, and establish the twisted duality theorem for basic cohomology. As a corollary, we obtain the extension theorem of the induced boundary forms.

Since we always assume that the boundary defining function $\rho$ is a basic function, the integrable subbundle $L$ induces an integrable subbundle $L|_{\partial M}$ and the restriction $g_M|_{\partial M}$ is also bundle like w.r.t. $L|_{\partial M}$. Hence, there is a naturally induced Riemannian foliation on $\partial M$. We will add subscript $\cdot_{b}$ to denote all objects on $\partial M$. 

For $0\leq r\leq q$, we need the transversal star operator $*_Q:\Omega_Q^r\left(M\right)\longrightarrow\Omega_Q^{q-r}\left(M\right)$ introduced by (see \cite{KT83F} and \cite{KT87})
\begin{align}\label{star operator}
	*_Qu=(-1)^{p(q-r)}*\left(u\wedge\chi_{\mathcal{F}}\right),
\end{align}
where $*$ is the Hodge star operator associated to $g_M$, and $\chi_{\mathcal{F}}$ is the characteristic form of the foliation $\mathcal{F}$, i.e., the leafwise volume form. It is elementary to see that $*_Q$ preserves $\Omega_B^\cdot\left(M\right)$ and $*_Q^2u=(-1)^{r(q-r)}u$ for all $u\in\Omega_{Q}^r\left(M\right)$. Let $\kappa$ denote the mean curvature form and we define ${\rm d}_{Q,\kappa}={\rm d}_Q-\kappa\wedge$, then (see \cite{KT83F} and \cite{KT87})
\begin{align}\label{twist operators}
	\delta_Q=(-1)^{q(r+1)+1}*_Q{\rm d}_{Q,\kappa}*_Q,\quad\delta_{Q,\kappa}=(-1)^{q(r+1)+1}*_Q{\rm d}_Q*_Q,
\end{align}
where $\delta_{Q,\kappa}$ is the formal adjoint of ${\rm d}_{Q,\kappa}$ with respect to $(\cdot,\cdot)_M$. Since the first order term of ${\rm d}_{Q,\kappa}$ is same as ${\rm d}_Q$, 
$$\mathcal{D}_{Q,\kappa}^r:={\rm Dom}\left({\rm d}_{Q,\kappa}^*\right)\cap\Omega_Q^\cdot\left(\overline{M}\right)=\mathcal{D}_{Q}^r.$$
In particular, all these relationships still hold for ${\rm d}_{B,\kappa}$ and ${\rm d}_B$. By $\Delta_\kappa$ we denote the Laplacian for ${\rm d}_{B,\kappa}$, i.e., $\Delta_\kappa={\rm d}_{B,\kappa}\delta_{B,\kappa}+\delta_{B,\kappa}{\rm d}_{B,\kappa}$. We now can retrace the steps in previous sections, starting with the corresponding Friedrichs operator $F_{B,\kappa}$. Since we may establish the Bochner formula for ${\rm d}_{Q,\kappa}$ by using $g^{\alpha\beta}X_\alpha\lrcorner\left(\nabla_{X_\beta}\kappa\right)\wedge$ to replace $\mathcal{A}_\tau:=\theta^\alpha\wedge\left(\nabla_{X_\alpha}\tau\right)\lrcorner$ in $(\ref{Bochner-Kodaira formula})$, applying the same arguments literally the Hodge decomposition for ${\rm d}_{B,\kappa}$ and the analogue of Corollary \ref{existence theorem 1} can be proved. We reformulate these results as follows.
\begin{prop}\label{hodge and existence theorem 1 for d_kappa}
	Let $\left(M,g_M,\mathcal{F},\rho\right)$ is same as Theorem $\ref{Hodge decomposition theorem}$. For $ 0\leq r\leq q $, there is an orthogonal decomposition
	\begin{equation*}
		\Omega_B^r\left(\overline{M}\right)={\rm Im}{\rm d}_{B,\kappa}\oplus{\rm Im}\delta_{B,\kappa}\oplus\mathcal{H}_{\kappa}^r,
	\end{equation*}
	where the kernel $ \mathcal{H}_{\kappa}^r $ of $ \Delta_{\kappa}|_{{\rm Dom}\left(F_{B,\kappa}\right)} $ is a finite-dimensional space. Moreover, for $ \omega\in\Omega_B^r\left(\overline{M}\right) $ $(1\leq r\leq q)$ with $ {\rm d}_{B,\kappa}\omega=0 $ and $ \omega\perp\mathcal{H}_{\kappa}^r $, then there exists a unique solution $ u\in\Omega_B^{r-1}\left(\overline{M}\right) $ of the equation $ {\rm d}_{B,\kappa}u=\omega $ which is orthogonal to the $ {\rm Ker}\left({\rm d}_{B,\kappa}\right) $. Thus, we have an isomorphism
	$$H_{B,\kappa}^r\left(\overline{M},\mathcal{F}\right):=\frac{\{u\in\Omega_B^r\left(\overline{M}\right)\ |\ {\rm d}_{B,\kappa}u=0\}}{{\rm d}_{B,\kappa}\Omega_B^{r-1}\left(\overline{M}\right)}\cong\mathcal{H}_{\kappa}^r.$$ 
\end{prop}

Recalling that $\mathcal{D}_B^\cdot:=\{u\in\Omega_{B}^\cdot\left(\overline{M}\right)\ |\ {\rm grad}\rho\lrcorner u=0\ {\rm on}\ \partial M\}$, it's obvious that $\Omega_{B,b}^\cdot\left(\partial M\right)=\{u|_{\partial M}\ |\ u\in\mathcal{D}_B^\cdot\}$. We define $\mathcal{C}_B^\cdot:={\rm Dom}\left(\delta_B^*\right)\cap\Omega_B^\cdot\left(\overline{M}\right)$, and in precise analogy to the reasoning of Proposition \ref{boundary condition} we get
\begin{align}\label{condition of C_B}
	\mathcal{C}_B^\cdot=\{u\in\Omega_B^\cdot\left(\overline{M}\right)\ |\ {\rm d}_B\rho\wedge u=0\ {\rm on}\ \partial M\}=\{u\in\Omega_B^\cdot\left(\overline{M}\right)\ |\ u={\rm d}_B\rho\wedge\cdot+\rho\cdot\},
\end{align}
where the second equality is valid by the definition. We could then argue correspondingly to formulate the de Rham-Hodge decompostion for the operator $\delta_B$. To establish the duality theorem, one needs to derive here some relations between $\mathcal{D}_B^\cdot$ and $\mathcal{C}_B^\cdot$.

\begin{prop}
	For $0\leq r\leq q$, we have $\mathcal{C}_B^r=*_Q\mathcal{D}_B^{q-r}$, ${\rm d}_B\mathcal{C}_B^r\subseteq\mathcal{C}_B^{r+1}$ and $\delta_{B,\kappa}\mathcal{D}_B^{r+1}\subseteq\mathcal{D}_B^r$.
\end{prop}
\begin{proof}
	For the first assertion, $u\in*_Q\mathcal{D}_B^r\Leftrightarrow*_Qu\in\mathcal{D}_B^r\Leftrightarrow{\rm grad}\rho\lrcorner*_Qu=0\ {\rm on}\ \partial M\Leftrightarrow*_Q{\rm grad}\rho\lrcorner*_Qu=0\ {\rm on}\ \partial M\Leftrightarrow{\rm d}_B\rho\wedge u=0\ {\rm on}\ \partial M$ which means $u\in\mathcal{C}_B^r$.
	
	If $u\in\mathcal{C}_B^r$, due to $(\ref{condition of C_B})$ we have $u={\rm d}_B\rho\wedge u'+\rho u''$ where $u'\in\Omega_{B}^{r-1}\left(\overline{M}\right)$ and $u''\in\Omega_B^r\left(\overline{M}\right)$, then ${\rm d}_Bu={\rm d}_B\rho\wedge\left(u''-{\rm d}_Bu'\right)+\rho{\rm d}_Bu''\in\mathcal{C}_B^{r+1}$ again by $(\ref{condition of C_B})$. For the last statement, from $(\ref{twist operators})$ we know $\delta_{B,\kappa}\mathcal{D}_B^{r+1}\subseteq*_Q{\rm d}_B*_Q\mathcal{D}_B^{r+1}\subseteq*_Q{\rm d}_B\mathcal{C}_B^{q-r-1}\subseteq*_Q\mathcal{C}_B^{q-r}\subseteq\mathcal{D}_B^r$.
\end{proof}

Therefore, for $0\leq r\leq q$ we have the following exact commutative diagram
\begin{align}\label{exact commutative diagram}
	\begin{CD}
		0 @>>> \mathcal{C}_B^r @>>> \Omega_B^r\left(\overline{M}\right) @>R>> \Omega_{B,b}^r\left(\partial M\right) @>>>0\\
		@.                   @VV {\rm d}_BV  @VV {\rm d}_BV       @VV {\rm d}_{B,b}V\\
		0 @>>> \mathcal{C}_B^{r+1} @>>> \Omega_B^{r+1}\left(\overline{M}\right) @>R>> \Omega_{B,b}^{r+1}\left(\partial M\right) @>>>0,
	\end{CD}
\end{align}
where $R$ is induced by the restriction $R':\mathcal{D}_B^r\longrightarrow\Omega_{B,b}^r\left(\partial M\right)$. Defining the cohomology 
\begin{align*}
	H_B^r\left(\overline{M},\mathcal{C}\right):=\{u\in\mathcal{C}_B^r\ |\ {\rm d}_Bu=0\}/\{{\rm d}_Bu\ |\ u\in\mathcal{C}_B^{r-1}\},
\end{align*}
we further have the following twisted theorem:

\begin{thm}\label{twisted duality}
	For $0\leq r\leq q$, then
	$$H_B^r\left(\overline{M},\mathcal{C}\right)\cong\left(H_{B,\kappa}^{q-r}\left(\overline{M},\mathcal{F}\right)\right)^*.$$
\end{thm}

\begin{proof}
    We define $\Phi:H_B^r\left(\overline{M},\mathcal{C}\right)\longrightarrow\left(H_{B,\kappa}^{q-r}\left(\overline{M},\mathcal{F}\right)\right)^*$ by $\Phi_u\left(v\right)=\int_{M}u\wedge v\wedge\chi_{\mathcal{F}}$ for $u\in H_B^r\left(\overline{M},\mathcal{C}\right)$ and $v\in H_{B,\kappa}^{q-r}\left(\overline{M},\mathcal{F}\right)$. We first need to prove $\Phi$ is well-defined. Indeed, if $u={\rm d}_Bu'$ for $u'\in\mathcal{C}_B^{r-1}$ and $v\in\Omega_B^{q-r}\left(\overline{M}\right)$ is a ${\rm d}_{B,\kappa}$-closed form, then
	\begin{align*}
		{\rm d}\left(u'\wedge v\wedge\chi_{\mathcal{F}}\right)=u\wedge v\wedge\chi_{\mathcal{F}}+(-1)^{r-1}u'\wedge{\rm d}_Bv\wedge\chi_{\mathcal{F}}+(-1)^{q-1}u'\wedge v\wedge{\rm d}\chi_{\mathcal{F}}.
	\end{align*}
    The Rummler's formula (see \cite{Rh79}) gives ${\rm d}\chi_{\mathcal{F}}=\kappa\wedge\chi_{\mathcal{F}}+\varphi_0$, where $\varphi_0$ is a $p+1$ form such that $X_1\lrcorner\cdots X_p\lrcorner\varphi_0=0$ for any local frame $\{X_a\}_{a=1}^p$ of $L$. It turns out that $w\wedge\varphi_0=0$ for any $w\in\Omega_Q^{q-1}\left(\overline{M}\right)$. Thus,
    \begin{align*}
    	{\rm d}\left(u'\wedge v\wedge\chi_{\mathcal{F}}\right)=&u\wedge v\wedge\chi_{\mathcal{F}}+(-1)^{r-1}u'\wedge{\rm d}_Bv\wedge\chi_{\mathcal{F}}+(-1)^{q-1}u'\wedge v\wedge\kappa\wedge\chi_{\mathcal{F}}\\
        =&u\wedge v\wedge\chi_{\mathcal{F}}+(-1)^{r-1}u'\wedge{\rm d}_{B,\kappa}v\wedge\chi_{\mathcal{F}}\\
        =&u\wedge v\wedge\chi_{\mathcal{F}}.
    \end{align*}
    Since $u'\in C_B^{r-1}$, we have $u'|_{\partial M}={\rm d}_B\rho\wedge u''$ for $u''\in\Omega_B^{r-2}\left(\overline{M}\right)$. By Stokes' theorem
    \begin{align*}
    	\int_{M}u\wedge v\wedge\chi_{\mathcal{F}}=\int_{\partial M}u'\wedge v\wedge\chi_{\mathcal{F}}&=\int_{\partial M}{\rm d}_B\rho\wedge u''\wedge v\wedge\chi_{\mathcal{F}}\\
    	&=\int_{\partial M}{\rm d}\left(\rho\wedge u''\wedge v\wedge\chi_{\mathcal{F}}\right)-\int_{\partial M}\rho{\rm d}\left(\wedge u''\wedge v\wedge\chi_{\mathcal{F}}\right)=0.
    \end{align*}
    Hence, $\Phi_{{\rm d}_Bu'}\left(v\right)=0$ for $v\in H_{B,\kappa}^{q-r}\left(\overline{M},\mathcal{F}\right)$. Similarly, $\Phi_{u}\left({\rm d}_{B,\kappa}v'\right)=0$ for $u\in H_B^r\left(\overline{M},\mathcal{C}\right)$.
    
    For the injectivity of $\Phi$, assume that $\Phi_u=0$ for $u\in H_B^r\left(\overline{M},\mathcal{C}\right)$, then for any $v\in\Omega_B^{q-r}\left(\overline{M}\right)$ with ${\rm d}_{B,\kappa}v=0$,
    \begin{align*}
    	0=\int_Mu\wedge v\wedge\chi_{\mathcal{F}}=(-1)^{r(q-r)}\int_M*_Q*_Qu\wedge v\wedge\chi_{\mathcal{F}}=\int_Mv\wedge*_Q*_Qu\wedge\chi_{\mathcal{F}}=\left(v,*_Qu\right)_B,
    \end{align*}
    i.e., $*_Qu\perp {\rm Ker}\left({\rm d}_{B,\kappa}\right)$. It follows from Proposition \ref{hodge and existence theorem 1 for d_kappa} that $*_Qu=\delta_{B,\kappa}u'$ for $u'\in\mathcal{D}_B^{q-r+1}$ which implies $*_Qu'\in\mathcal{C}_B^{r-1}$. By $(\ref{twist operators})$,
    \begin{align*}
    	u=(-1)^{r(q-r)}*_Q\delta_{B,\kappa}u'=(-1)^{q(r+1)+1}{\rm d}_B*_Qu',
    \end{align*}
    it gives that the cohomology class $[u]=0$ in $H_B^r\left(\overline{M},\mathcal{C}\right)$.
    
    In order to complete the proof, it suffices to prove $\left(\Phi_{H_B^r\left(\overline{M},\mathcal{C}\right)}\right)^{\perp}=0$. In fact, suppose there is ${\rm d}_{B,\kappa}$-closed form $v\in\Omega_{B}^{q-r}\left(\overline{M}\right)$ such that $\Phi_u\left(v\right)=\left(v,*_Qu\right)_B=0$ for all $u\in H_{B}^r\left(\overline{M},\mathcal{C}\right)$. Since $*_Qu\in\mathcal{D}_B^{q-r}$, again by $(\ref{twist operators})$ we have
    \begin{align*}
    	{\rm d}_{B,\kappa}^**_Qu=\delta_{B,\kappa}*_Qu=(-1)^{r^2-1}*_Q{\rm d}_Bu=0,
    \end{align*}
    it follows that $*_Qu\in {\rm Ker}\left({\rm d}_{B,\kappa}^*\right)$ and $v\perp {\rm Ker}\left({\rm d}_{B,\kappa}^*\right)$. In view of Proposition \ref{hodge and existence theorem 1 for d_kappa} there exists a basic form $v'\in\Omega_B^{q-r-1}$ such that $v={\rm d}_{B,\kappa}v'$, i.e., $[v]=0$ in $H_{B,\kappa}^{q-r}\left(\overline{M},\mathcal{F}\right)$.
\end{proof}

\begin{remark}
	$(1)$ If $\partial M=\emptyset$, Theorem \ref{twisted duality} recovers the twisted Poincar{\'e} duality theorem proved by F. W. Kamber and P.Tondeur in \cite{KT83D}.
	
	$(2)$ If $\kappa=0$, Theorem \ref{twisted duality} is the analogue of the Poincar{\'e} duality theorem for manifolds with boundary.
\end{remark}

We finish this section by considering the tangential extension as follows. For $u\in\Omega_{B,b}^\cdot\left(\partial M\right)$, we say $u$ has a tangential extension if there is a form $\tilde{u}\in\Omega_{B}^\cdot\left(\overline{M}\right)$ such that $R\tilde{u}=u$, where $R$ is the projection defined in $(\ref{exact commutative diagram})$. We have the following extension theorem for ${\rm d}_{B,b}$-closed forms on $\Omega_{B,b}^r\left(\partial M\right)$.

\begin{thm}
	For $0\leq r\leq q$, given a ${\rm d}_{B,b}$-closed form $u\in\Omega_{B,b}^r\left(\partial M\right)$, it has a tangential ${\rm d}_B$-closed extension if and only if $\int_{\partial M}u\wedge v\wedge\chi_{\mathcal{F}}=0$ for any $v\in\mathcal{H}_\kappa^{q-r-1}$.
\end{thm}

\begin{proof}
	By the exact commutative diagram $(\ref{exact commutative diagram})$ we have the following long exact cohomology sequence
	\begin{align*}
		\cdots\longrightarrow H_{B}^r\left(\overline{M},\mathcal{F}\right)\longrightarrow H_{B,b}^r\left(\partial M,\mathcal{F}\right)\stackrel{\delta^*}\longrightarrow H_{B}^{r+1}\left(\overline{M},\mathcal{C}\right)\xlongrightarrow[\Phi]{\cong}\left(H_{B,\kappa}^{q-r-1}\left(\overline{M},\mathcal{F}\right)\right)^*\longrightarrow\cdots,
	\end{align*}
    where $\delta^*$ is the indeuced homomorphism, and the isomorphism holds due to Theorem \ref{twisted duality}. Obviously, ${\rm d}_{B,b}u=0$ means that there exists $u'\in\mathcal{D}_B^r$ with $u'|_{\partial M}=u$ such that ${\rm d}_Bu'\in\mathcal{C}_{B}^{r+1}$. Therefore, the ${\rm d}_{B,b}$-closed form $u\in\Omega_{B,b}^r\left(\partial M\right)$ has a weak ${\rm d}_B$-closed extension precisely when $\delta^*[u]=0$ in $H_{B}^{r+1}\left(\overline{M},\mathcal{C}\right)$, i.e., ${\rm d}_Bu'={\rm d}_{B}u''$ for $u''\in\mathcal{C}_{B}^{r+1}$, then $\tilde{u}:=u'-u''$ is the desired extension form in $\Omega_{B}^r\left(\overline{M}\right)$. Hence, it suffices to prove that there exists a form $u''\in\mathcal{C}_B^{r+1}$ such that ${\rm d}_{B}u'={\rm d}_{B}u''$ if and only if $\int_{\partial M}u\wedge v\wedge\chi_{\mathcal{F}}=0$ for any $v\in\mathcal{H}_\kappa^{q-r-1}$.
    
    On the other hand, again by Propositin \ref{hodge and existence theorem 1 for d_kappa}, the equation ${\rm d}_Bu'={\rm d}_Bu''$ is solvable is equivalent to $\Phi_{{\rm d}_Bu'}\left(v\right)=0$ for all $v\in H_{\kappa}^{q-r-1}\left(\overline{M},\mathcal{F}\right)\cong\mathcal{H}_{\kappa}^{q-r-1}$, i.e.,
    \begin{align*}
    	0=\Phi_{{\rm d}_Bu'}\left(v\right)=\int_{M}{\rm d}_Bu'\wedge v\wedge\chi_{\mathcal{F}}&=\int_{M}{\rm d}\left(u'\wedge v\wedge\chi_{\mathcal{F}}\right)-(-1)^r\int_{M}u'\wedge{\rm d}_{B,\kappa}v\wedge\chi_{\mathcal{F}}\\
    	&=\int_{\partial M}u'\wedge v\wedge\chi_{\mathcal{F}}.
    \end{align*}
    The proof is thus complete.
\end{proof}

\section{Hermitian Foliations}
\subsection{Preliminaries of Hermitian foliations}
In this section, we first recall some definitions of Hermitian foliations. To prove the decomposition theorem for Hermitian foliations, we then give some functional analysis preparations.

\begin{definition}\label{hermitian foliation}
	Let $ M' $ be an n $ (=p+2q) $ dimensional manifold equipped with a foliation $ \mathcal{F} $ of codimension $ 2q $, then the foliation is said (transversely) Hermitian if $ \mathcal{F} $ satisfies:
	
	{\rm (i)} There is a bundle-like metric on $ M' $.
	
	{\rm (ii)} There exists a holonomy invariant integrable almost complex structure $ J:Q\longrightarrow Q $, where $Q$ is the quotient bundle of the exact sequence $(\ref{exacts})$ associated to $\mathcal{F}$ with $ {\rm dim}Q=2q $ such that $ g_Q\left(JV,JW\right)=g_Q\left(V,W\right) $ for all $ V,W\in \Gamma Q $.
\end{definition}

Let $\left(M',g_{M'},\mathcal{F}\right)$ be a Hermitian folaition. We consider the complexified normal bundle $ Q_{\mathbb{C}}:=Q\otimes_{\mathbb{R}}\mathbb{C} $. Then the almost complex structure $ J $ splits $ Q_{\mathbb{C}}=Q_{\mathbb{C}}^{1,0}\oplus Q_{\mathbb{C}}^{0,1} $, where
\begin{align*}
	Q_{\mathbb{C}}^{1,0}:=\{V\in \Gamma Q\ |\ JV=iV\}\ {\rm and}\ Q_{\mathbb{C}}^{1,0}:=\{V\in \Gamma Q\ |\ JV=-iV\}.
\end{align*}
Thus, for $ 1\leq r_1\leq 2q $ we have a decomposition of $ \Lambda^{r_1} Q_{\mathbb{C}}^* $ as
\begin{align*}
	\Lambda^{r_1} Q_{\mathbb{C}}^*=\bigoplus_{r'+r=r_1}\Lambda^{r'}{Q_{\mathbb{C}}^{1,0}}^*\otimes\Lambda^{r}{Q_{\mathbb{C}}^{0,1}}^*.
\end{align*}
Let $ \Omega_B^{r',r}\left(M'\right) $ and $ \Omega_Q^{r',r}\left(M'\right) $ be the sets of basic sections and smooth sections of $ \Lambda^{r'}{Q_{\mathbb{C}}^{1,0}}^*\otimes\Lambda^{r}{Q_{\mathbb{C}}^{0,1}}^* $ respectively, which gives rise to decompositions
\begin{align*}
	\Omega_B^{r_1}\left(M'\right)=\bigoplus_{r'+r=r_1}\Omega_B^{r',r}\left(M'\right)\ {\rm and}\ \Omega_Q^{r_1}\left(M'\right)=\bigoplus_{r'+r=r_1}\Omega_Q^{r',r}\left(M'\right).
\end{align*}

Since $ {\rm d}_B $ can be treated as the exterior differential on the local quotient manifold in the local distinguished charts, the exterior derivative decomposes into two operators $ {\rm d}_B=\partial_B+\bar{\partial}_B $ as in the classical case of a complex manifold, where
\begin{align*}
	\partial_B:\Omega_B^{r',r}\left(M'\right)\longrightarrow\Omega_B^{r'+1,r}\left(M'\right)\ {\rm and}\ \bar{\partial}_B:\Omega_B^{r',r}\left(M'\right)\longrightarrow\Omega_B^{r',r+1}\left(M'\right).
\end{align*}
It turns out that $ \partial_B^2=\bar{\partial}_B^2=0 $ and $ \partial_B\bar{\partial}_B=-\bar{\partial}_B\partial_B $. Hence, we obtain the basic Dolbeault complex 
\begin{equation*}
	0\longrightarrow\Omega_B^{\cdot,1}\left(M'\right)\stackrel{\bar{\partial}_B}{\longrightarrow}\cdots\stackrel{\bar{\partial}_B}{\longrightarrow}\Omega_B^{\cdot,r}\left(M'\right)\stackrel{\bar{\partial}_B}{\longrightarrow}\Omega_B^{\cdot,r+1}\left(M'\right)\stackrel{\bar{\partial}_B}{\longrightarrow}\cdots\stackrel{\bar{\partial}_B}{\longrightarrow}\Omega_B^{\cdot,q}\left(M'\right)\longrightarrow0.
\end{equation*}
The cohomology
\begin{equation*}
	H^{\cdot,r}_B\left(M',\mathcal{F}\right)=H\left(\Omega_B^{\cdot,r}\left(M'\right),\bar{\partial}_B\right)
\end{equation*} 
is called the basic Dolbeault cohomology of $ \mathcal{F} $.

For later use, we introduce the operators
\begin{align*}
	\partial_Q|_{\Omega_Q^{r',r}\left(M'\right)}:=\Pi_{r'+1,r}{\rm d}_Q|_{\Omega_Q^{r',r}\left(M'\right)}\ {\rm and}\ \bar{\partial}_Q|_{\Omega_Q^{r',r}\left(M'\right)}:=\Pi_{r',r+1}{\rm d}_Q|_{\Omega_Q^{r',r}\left(M'\right)},
\end{align*}
where $ {\rm d}_Q $ and $ \delta_Q $ defined as (\ref{Q-operator}), and $ \Pi_{r',r}:\Omega_Q^{r'+r}\left(M'\right)\longrightarrow\Omega_Q^{r',r}\left(M'\right) $ is the projection. By (ii) in Definition \ref{hermitian foliation} we get ${\rm d}_Q=\partial_Q+\bar{\partial}_Q$. Clearly $ \partial_Q|_{\Omega_B^{\cdot,\cdot}\left(M'\right)}=\partial_B $ and $ \bar{\partial}_Q|_{\Omega_B^{\cdot,\cdot}\left(M'\right)}=\bar{\partial}_B $ since $ {\rm d}_Q|_{\Omega_B^{\cdot}\left(M'\right)}={\rm d}_B $.

In what follows, we are mainly interested in all objects above in $\left(\overline{M},\mathcal{F}|_{\overline{M}}\right)$. Without loss of generality, we assume that the boundary defining function $\rho$ satisfies $ \rho<0 $ inside $ M $, $ \rho>0 $ outside $ \overline{M} $ and $ |{\rm d}\rho|=1 $ on $ \partial M $. The notation $\cdot\left(\overline{M}\right)$ refers to the subspace of $\cdot\left(M\right)$ consists of elements in $\cdot\left(M\right)$ which can be extended smoothly to $M'$.

Now, we can retrace the steps in Chapter 2. Let $ H_{Q,s}^{\cdot,\cdot} $ and $ H_{B,s}^{\cdot,\cdot} $ as the completions of $ \Omega_B^{\cdot,\cdot}\left(\overline{M}\right) $ and $ \Omega_Q^{\cdot,\cdot}\left(\overline{M}\right) $ with respect to the norms $ (\ref{norms on Q}) $ and $ (\ref{norms on B}) $(or $ (\ref{norm}) $) respectively. The analogue of Proposition $ \ref{Rellich and Sobolev lemma} $ also holds for $ H_{B,s}^{\cdot,\cdot} $. For $ 0\leq r\leq q $, let $\bar{\partial}_B$ be the maximal extension from $H_{B,0}^{\cdot,r}\longrightarrow H_{B,0}^{\cdot,r+1}$ and we define $ \mathcal{D}_B^{\cdot,r}={\rm Dom}\left(\bar{\partial}_B^*\right)\cap\Omega_B^{\cdot,r}\left(\overline{M}\right) $, where $ \bar{\partial}_B^* $ is the Hilbert adjoint of $ {\bar{\partial}}_B $ w.r.t. $\left(\cdot,\cdot\right)_B$. One can define $ \mathcal{D}_Q^{\cdot,r} $ similarly. General divergence theorem yields that for $ u\in\Omega_{Q}^{\cdot,r}\left(\overline{M}\right) $ and $ v\in\Omega_{Q}^{\cdot,r-1}\left(\overline{M}\right) $
\begin{align}\label{divergence theorem for Q}
	\left(\vartheta_Qu,v\right)=\left(u,\bar{\partial}_Qv\right)+\int_{\partial M}\left<\sigma\left(\vartheta_Q,{\rm d}\rho \right)u,v\right>,
\end{align}
where $ \vartheta_Q $ is the formal adjoint of $ \bar{\partial}_Q $ with respect to the inner product $ \left(\cdot,\cdot\right) $, and $ \sigma\left(\vartheta_Q,{\rm d}\rho\right) $ is the symbol of $ \vartheta_Q $. By standard techniques
\begin{align}\label{boundary condition for Q}
	\mathcal{D}_Q^{\cdot,r}=\{u\in\Omega_Q^{\cdot,r}\left(\overline{M}\right)\ |\ \sigma\left(\vartheta_Q,{\rm d}\rho \right)u=0\ {\rm on}\ \partial M\}=\{u\in\Omega_Q^{\cdot,r}\left(\overline{M}\right)\ |\ {{\rm grad}}^{0,1}\rho\lrcorner u=0\ {\rm on}\ \partial M\},
\end{align}
and $ \bar{\partial}_Q^*=\vartheta_Q$ on $\mathcal{D}_Q^{\cdot,r}$. Hence, multiply a smooth function preserves $\mathcal{D}_Q^{\cdot,r}$. Moreover, we have 
\begin{align}\label{DBDQ}
	\mathcal{D}_B^{\cdot,r}=\mathcal{D}_Q^{\cdot,r}\cap\Omega_B^{\cdot,r}\left(\overline{M}\right)=\{u\in\Omega_B^{\cdot,r}\left(\overline{M}\right)\ |\ {{\rm grad}}^{0,1}\rho\lrcorner u=0\ {\rm on}\ \partial M\}.
\end{align}
Indeed, the direction $ \mathcal{D}_Q^{\cdot,r}\cap\Omega_B^{\cdot,r}\left(\overline{M}\right)\subset\mathcal{D}_B^{\cdot,r} $ is trivial. Conversely, $ \rho*\Omega_B^{\cdot,r}\left(\overline{M}\right)=\{\rho u\ |\ u\in\Omega_B^{\cdot,r}\left(\overline{M}\right)\} $ is dense in $ \Omega_B^{\cdot,r}\left(\overline{M}\right) $, and it's obvious that $ \rho*\Omega_B^{\cdot,r}\left(\overline{M}\right)\subset\mathcal{D}_Q^{\cdot,r} $ by $ (\ref{boundary condition for Q}) $ which yields that $ \mathcal{D}_B^{\cdot,r}\subset\mathcal{D}_Q^{\cdot,r}\cap\Omega_B^{\cdot,r}\left(\overline{M}\right) $.

Now we wish to describe the basic Laplacian operator $\Box_B:=\bar{\partial}_B\vartheta_B+\vartheta_B\bar{\partial}_B$ which
suits its role in the basic Dolbeault complex. To this end, for $ 0\leq r\leq q $, one can define $ \bar{G}_B $ on $ \mathcal{D}_B^{\cdot,r} $ by
\begin{align*}
	\bar{G}_B\left(u,v\right)=\left(\bar{\partial}_Bu,\bar{\partial}_Bv\right)_B+\left(\vartheta_Bu,\vartheta_Bv\right)_B+\left(u,v\right)_B.
\end{align*}
Let $ \tilde{\mathcal{D}}_B^{\cdot,r} $ be the completion of $ \mathcal{D}_B^{\cdot,r} $ with respect to $ \bar{G}_B $.

\begin{prop}
	For $ 0\leq r\leq q $, the map $ i:\tilde{\mathcal{D}}_B^{\cdot,r}\longrightarrow H_{B,0}^{\cdot,r} $ induced by $ \mathcal{D}_B^{\cdot,r}\longrightarrow H_{B,0}^{\cdot,r} $ is injective, and $ \tilde{\mathcal{D}}_B^{\cdot,r} $ is dense in $ H_{B,0}^{\cdot,r} $.
\end{prop}
\begin{proof}
	The proof is similar to that one of Proposition \ref{verification of Friedrichs extension theorem} so we omit it.
\end{proof}

Then applying Theorem \ref{Friedrichs extension theorem}, there is a Friedrichs operator $ \bar{F}_B $ associated to the form $ \bar{G}_B $. Moreover, we can conclude the following property of $ \bar{F}_B $ by parallel argument of Proposition \ref{expression of $ F_B $}:

\begin{prop}\label{expression of F}
	For $ 0\leq r\leq q $, if $ u\in\mathcal{D}_B^{\cdot,r} $, then $ u\in{\rm Dom}\left(\bar{F}_B\right) $ if and only if $ \bar{\partial}_Bu\in\mathcal{D}_B^{\cdot,r+1} $, in this case we have $ \bar{F}_Bu=\left(\Box_B+{\rm Id}\right)u $.
\end{prop}

Correspondingly, we can define the $ Q $-norm $\bar{G}_Q$ and $ Q $-Friedrichs operator $\bar{F}_Q$ such that the smooth elements of ${\rm Dom}\left(\bar{F}_Q\right)$ are characterized as Proposition \ref{expression of F}. Furthermore, it follows from $(\ref{DBDQ})$ that $ \bar{F}_Q|_{{\rm Dom}\left(\bar{F}_B\right)}=\bar{F}_B $.

\subsection{The basic Dolbeault decomposition}
This section is devoted to establishing the basic Dolbeault decomposition for Hermitian foliations. Similar to the arguments of Theorem \ref{Hodge decomposition theorem}, we proceed with the proof by establishing the existence and regularity of the equation $\bar{F}_Bu=\omega$. However, the boundary condition creates some difficulties for Hermitian foliations. Specifically, the G$\mathring{{\rm a}}$rding's inequality breaks down near the boundary. Therefore, we primarily concentrate on handling boundary cases by technique of elliptic regularization developed by Kohn-Nirenberg in \cite{KN65A} and \cite{KN65N}.

Let $ \left(\left(x,z\right):=\left(x_1,\cdots,x_p,z_1,\cdots,z_q\right),U\right) $ be a local distinguished coordinate of $ M $. We choose $ X_a=\frac{\partial}{\partial x_a} $ for $ 1\leq a\leq p $, and let $ \{V_\alpha\}_{\alpha=1}^q $ be a local orthonormal basis field for $ \Gamma Q_{\mathbb{C}}^{1,0} $ with dual basis $ \{\omega^\alpha\}_{\alpha=1}^q $ of $ \Omega_Q^{1,0}\left(\overline{M}\right) $. 

As in the section 3.3, we divide the local coordinate charts $ \{\left(\left(x,z\right),U_l\right)\}_{l=1}^L $ into two parts $ \lbrace\left(\left(x,z\right),U'_l\right)\rbrace_{l=1}^{L_1} $ and $ \lbrace\left(\left(x,z\right),U''_l\right)\rbrace_{l=1}^{L_2} $, where $ U'_l\cap\partial M=\emptyset $ and $ U_l''\cap\partial M\neq\emptyset $. Let $ \lbrace\eta_l\rbrace_{l=1}^{L}:=\lbrace\eta_l'\rbrace_{l=1}^{L_1}\cup\lbrace\eta_l''\rbrace_{l=1}^{L_2} $ be a partition of unity subordinate to the covering $ \lbrace U_l\rbrace_{l=1}^{L} $, where $ \eta_l' $ and $ \eta_l'' $ with support in $ U'_l $ and $ U''_l $ respectively. For the boundary charts, we also use special boundary coordinate charts for convenience as follows. 

If $ \left(\left(x,z\right),U\right) $ is a coordinate neighborhood such that $ U\cap\partial M\neq\emptyset $, let $ \{t_1,\cdots,t_{n-1},\rho\} $ be the real coordinate functions set and $ \omega^q:=\sqrt{2}\partial_B\rho $. This kind of coordinate charts are called special boundary charts which were originally introduced by M. E. Ash (see \cite{Ame64}). In the special boundary charts, it's clear that $ V_\alpha\left(\rho\right)=0 $ for all $ \alpha<q $.

\begin{definition}\label{befd}
	For $ 0\leq r,r'\leq q $, we set $ E_U\left(v\right)^2:=\sum_{\alpha=1}^{q}\sum_{|I|=r',|J|=r}\|\bar{V}_\alpha\left(v_{I,J}\right)\|^2+\int_{\partial M}|v|^2+\|v\|^2 $ for $ v=v_{I,J}\omega^I\wedge\bar{\omega}^J\in\Omega_Q^{r',r}\left(U\right) $. We say that the basic estimate holds in $ U'' $, provided that for all $ v\in\mathcal{D}_Q^{r',r}\left(U''\right) $,
	\begin{align}\label{local basic estimate for dolbeault}
		E_U\left(v\right)^2\lesssim \bar{G}_{Q}\left(v,v\right)+O\left(\sum_{a=1}^{p}\|X_av\|\|v\|\right).
	\end{align}
	We say that the basic estimate holds in $ \mathcal{D}_Q^{r',r} $, if the estimate $ (\ref{local basic estimate for dolbeault}) $ holds on any boundary chart.
\end{definition}

We are in a position to formulate the existence and regularity of the equation $ \bar{F}_Bu=\omega $.
\begin{thm}\label{main theorem for dolbeault}
	Let $ 0\leq r,r'\leq q $, assume that the basic estimate holds in $ \mathcal{D}_Q^{r',r} $. For given $ \omega\in H_{B,0}^{r',r} $, there is a unique $ u\in\tilde{\mathcal{D}}_B^{r',r} $, such that $ \bar{F}_Bu=\omega $. Furthermore, if $ \omega\in\Omega_B^{r',r}\left(\overline{M}\right) $, then $ u\in\Omega_B^{r',r}\left(\overline{M}\right) $, and for each integer $ s\geq0 $, we have
	\begin{align}\label{estimate for dolbeault}
		\|u\|_{B,s+1}^2\lesssim\|\omega\|_{B,s}^2.
	\end{align}
\end{thm}

We will prove the estimate ($ \ref{estimate for dolbeault} $) for smooth forms at first, it's enough to prove the desired estimate locally by a partition of unity. The interior estimate is dependent on the ellipticity of basic complex Laplace-Beltram operator $\Box_B:=\bar{\partial}_B\vartheta_B+\vartheta_B\bar{\partial}_B$. More precisely, we have

\begin{prop}\label{mainp 1}
	For each $s\in\mathbb{Z}_+$ and each real $\eta'\in C_c^\infty\left(U'\right)$, then 
	$$\|\eta'u\|_{s+2}^2\lesssim\|\bar{F}_Bu\|_{B,s}^2$$
	holds uniformly for any $u\in{\rm Dom}\left(\bar{F}_B\right)\cap\Omega_B^{r',r}\left(\overline{M}\right)$.
\end{prop}

\begin{proof}
	Since $ \Box_B:=\bar{\partial}_B\vartheta_B+\vartheta_B\bar{\partial}_B $ is the restriction of the elliptic operator $ \frac{1}{2}\Delta-\hat{S} $ on $\Omega_B^{r',r}\left(\overline{M}\right)$ by Corollary 6.15 in Chapter VI in \cite{Djp12C} and Proposition \ref{difference of two Laplacian}, where $ \Delta:={\rm d}\delta+\delta{\rm d} $ and $ \hat{S} $ is a first order operator. It gives G$\mathring{{\rm a}}$rding's type inequality for $u\in{\rm Dom}\left(\bar{F}_B\right)\cap\Omega_B^{r',r}\left(\overline{M}\right)$, then one can obtain the desired estimate by parallel arguments in the proof of (\ref{estimate}).
\end{proof}

To derive a estimate near the boundary for smooth forms, the tangential Fourier transform on smooth function in special boundary chart will be used which is defined as follows
\begin{align*}
	\mathcal{F}w\left(\xi,\rho\right)\equiv\tilde{w}\left(\xi,\rho\right)=\frac{1}{\left(2\pi\right)^{-\left(n-1\right)/2}}\int_{\mathbb{R}^{n-1}}e^{-i\left<t,\xi\right>}w\left(t,\rho\right)dt,
\end{align*}
where $ t=\left(t_1,\cdots,t_{n-1}\right) $. We define operator $ \Lambda_t^s $ by
\begin{align*}
	\Lambda_t^sw=\frac{1}{\left(2\pi\right)^{-\left(n-1\right)/2}}\int_{\mathbb{R}^{n-1}}\left(1+|\xi|^2\right)^{s/2}e^{i\left<t,\xi\right>}\tilde{w}\left(\xi,\rho\right)d\xi,
\end{align*}
and then define tangential Sobolev norms by
$$ |||w|||_s^2=\|\Lambda_t^sw\|^2, $$
these operators can act on forms componentwise as usual. In order to consider all derivatives in special boundary chart $ U'' $, for $ v\in\Omega_Q^{\cdot,\cdot}\left(U''\right) $ we define 
\begin{align*}
	|||Dv|||_s^2=\sum_{i=1}^{n}|||D_i v|||_s^2+|||v|||_s^2\thicksim|||v|||_{s+1}^2+|||D_{\rho}v|||_s^2,
\end{align*}
where $ D_i=D_t^i:=\frac{\partial}{\partial t_i} $ for $ i<n $ and $ D_n=D_{\rho}:=\frac{\partial}{\partial\rho} $.

For the purpose to handle the boundary estimate, we further need the following series of lemmas. The first lemma is given by J. J. Kohn in \cite{FK72}:
\begin{lemma}\label{mainl 1}
	Let $ V'' $ be a special boundary, and $ M_1,\cdots,M_N $ be homogeneous first-order operators on $ V'' $ such that $ \tilde{M}:=\left(M_1,\cdots,M_N\right) $ is elliptic. Then for each $ p\in\partial M\cap V'' $, there is a neighborhood $ U''\Subset V'' $ of $ p $ such that $ \sum_{i=1}^{n}|||D_iv|||_{-\frac{1}{2}}^2\lesssim\sum_{l=1}^{N}|||M_lv|||_{-\frac{1}{2}}^2+\int_{\partial M}|v|^2 $ for all $ v\in\Omega_{Q,c}^{r',r}\left(U''\cap \overline{M}\right) $.
\end{lemma}

Here we denote by $\Omega_{Q,c}^{r',r}\left(U''\cap \overline{M}\right)$ the subspace of $\Omega_{Q}^{r',r}\left(\overline{M}\right)$ consists of elements with support in $U''$ but do not necessarily vanish on $\partial M$. Applying Lemma \ref{mainl 1} to $\left(V_1,\cdots,V_p,\bar{V}_1,\cdots,\bar{V}_q\right)$ we have
\begin{lemma}\label{mainl 2}
	For any $ p\in\partial M $, there is a special boundary chart $ U'' $ such that $ |||Dv|||_{-\frac{1}{2}}^2\lesssim E\left(v\right)^2+\sum_{a=1}^{p}\|X_a\left(v\right)\|^2 $ for all $ v\in\Omega_{Q,c}^{r',r}\left(U''\cap \overline{M}\right) $.
\end{lemma}

\begin{lemma}\label{mainl 3}
	For $ 0\leq r',r\leq q $, suppose the basic estimate holds in $ \mathcal{D}_Q^{r',r} $. Let $ U'' $ be a special boundary chart, for any $ m\in\mathbb{Z}_+ $ and real-valued $\eta''\in C_c^\infty\left(U''\right)$ then
	\begin{align*}
		|||D\eta''u|||_{\frac{m-2}{2}}^2\lesssim|||\bar{F}_Bu|||_{\frac{m-2}{2}}^2+\|\bar{F}_Bu\|_B^2
	\end{align*}
	hold for all $ u\in {\rm Dom}\left(\bar{F}_B\right)\cap\mathcal{D}_B^{r',r} $.
\end{lemma}
\begin{proof}
	We prove this lemma using induction on $m$. For $m=1$, by means of basic estimate and Lemma \ref{mainl 2},
	\begin{align*}
		|||D\eta''u|||_{-\frac{1}{2}}^2\lesssim E\left(\eta''u\right)^2+\sum_{a=1}^{p}\|X_a\left(\eta''u\right)\|^2\lesssim \bar{G}_Q\left(\eta''u,\eta''u\right)+\|u\|_B^2.
	\end{align*}
	Proceeding by Lemma 2.4.3 in \cite{FK72} we have
	\begin{align*}
		\bar{G}_Q\left(\eta''u,\eta''u\right)&=\bar{G}_Q\left(u,\left(\eta''\right)^2u\right)+O\left(\|u\|_B^2\right)\\
		&=\left(\bar{F}_Bu,\left(\eta''\right)^2u\right)+O\left(\|u\|_B^2\right)\\
		&\lesssim\|\bar{F}_Bu\|_B\|u\|_B+O\left(\|u\|_B^2\right)\\
		&\lesssim\|\bar{F}_Bu\|_B^2,
	\end{align*}
	where the last line follows since $\bar{F}_B^{-1}$ is a bounded operator.
	
	Suppose the estimate holds for $ m-1 $. For $m$ we get
	\begin{align*}
		|||D\eta''u|||_{\frac{m-2}{2}}^2=|||D\Lambda_t^{\left(m-1\right)/2}\eta''u|||_{-\frac{1}{2}}^2\lesssim|||DA_t\eta''u|||_{-\frac{1}{2}}^2+|||D\eta''u|||_{\frac{m-3}{2}}^2,
	\end{align*}
	where $A_t:=\hat{\eta}''\Lambda_t^{\left(m-1\right)/2} $ is a pseudo-differential operator of $\frac{m-1}{2}$-order and $\hat{\eta}''\in C^\infty_c\left(U'',\mathbb{R}\right)$ such that $\hat{\eta}''|_{{\rm supp}\eta}\equiv1$. Then it suffices to control the first term in the R.H.S of the above estimate. Applying Lemma \ref{mainl 2} and Lemma 2.4.2 in \cite{FK72} and by basic estimate, for any $ 0<\varepsilon<<1 $, there a exists constant $ C_\varepsilon>0 $ such that
	\begin{align*}
		|||DA_t\eta''u|||_{-\frac{1}{2}}^2&\lesssim E\left(A_t\eta''u\right)^2+\sum_{a=1}^{p}|||X_a\left(A_t\eta''u\right)|||_{-\frac{1}{2}}^2\\
		&\lesssim\bar{G}_Q\left(A_t\eta''u,A_t\eta''u\right)+O\left(\sum_{a=1}^{p}\|X_a\left(A_t\eta''u\right)\|\|A_t\eta''u\|\right)+|||\tilde{\eta}''u|||_{\frac{m-2}{2}}^2\\
		&={\rm Re}\bar{G}_Q\left(u,\overline{\eta''}\left(A_t\right)^*A_t\eta''u\right)+O\left(|||D\tilde{\eta}''u|||_{\frac{m-3}{2}}^2\right)\\
		&={\rm Re}\left(A_t\eta''\bar{F}_Bu,A_t\eta''u\right)+O\left(|||D\eta''u|||_{\frac{m-3}{2}}^2+|||D\tilde{\eta}''u|||_{\frac{m-3}{2}}^2\right)\\
		&\lesssim|||A_t\eta''\bar{F}_Bu|||_{-\frac{1}{2}}|||A_t\eta''u|||_{\frac{1}{2}}+O\left(|||D\eta''u|||_{\frac{m-3}{2}}^2+|||D\tilde{\eta}''u|||_{\frac{m-3}{2}}^2\right)\\
		&\lesssim C_\varepsilon|||\bar{F}_Bu|||_{\frac{m-2}{2}}+\varepsilon|||D\eta''u|||_{\frac{m-2}{2}}+O\left(|||D\eta''u|||_{\frac{m-3}{2}}^2+|||D\tilde{\eta}''u|||_{\frac{m-3}{2}}^2\right),
	\end{align*}
	where $\tilde{\eta}''$ denotes matrices
	of functions involving $\eta''$ and its derivatives. Then the proof is complete by inductive hypothesis.
\end{proof}

Now, we are prepared to prove Theorem \ref{main theorem for dolbeault}. 
\renewcommand{\proofname}{\bf $ Proof\ of\ Theorem\ \ref{main theorem for dolbeault} $}
\begin{proof}
	We first prove the estimate $(\ref{estimate for dolbeault})$ for smooth basic forms. The interior estimate is given in Proposition \ref{mainp 1}. For the boundary estimate, employing Lemma \ref{mainl 3} for $ m=2s+2 $ we have
	\begin{align}\label{normal direction estimate}
		|||D\eta''u|||_{s}^2\lesssim||\bar{F}_Bu||_{B,s}^2.
	\end{align}

	We show the estimate $ (\ref{estimate for dolbeault}) $ by induction on $ s $. For $ s=0 $,
	\begin{align*}
		\|u\|_{B,1}^2\thicksim&\sum_{i=1}^{n}\sum_{l=1}^{L_1}\|D_i\eta_l' u\|^2+\sum_{l=1}^{L_2}\|D\eta_l''u\|^2+\|u\|_{B}^2\lesssim||\bar{F}_Bu||_{s}^2.
	\end{align*}
	Assume that the estimate holds for $ s-1 $, then for $ s $ we get
	\begin{align}\label{norms of s+1}
		\|u\|_{B,s+1}^2\thicksim&\sum_{|I|=s}\sum_{l=1}^{L_1}\|D^{I}\eta_l' u\|_1^2+\sum_{l=1}^{L_2}|||D\eta_l''u|||_s^2+\|u\|_{B,s}^2+\sum_{|J|+\iota=s+1,\iota\geq2}\sum_{l=1}^{L_2}\|D_t^{J}D_\rho^\iota\eta_l''u\|^2,
	\end{align}
	where $ D_t^J:=\left(\frac{\partial}{\partial x}\right)^{J} $, and the $n_{th}$-component of $ J $ equals to zero. Thus the middle two items of R.H.S. in $ (\ref{norms of s+1}) $ can be estimated by $ (\ref{normal direction estimate}) $ and inductive hypothesis. The first term is controlled according to Proposition \ref{mainp 1}. For the last term, applying $ D_t^JD_\rho^{\iota-2}(|J|+\iota=s+2) $ to the analogue of (\ref{ellipticity}) and by induction on $\iota$, we know that $ \|D_t^{J}D_\rho^\iota\eta_l''u\|^2 $ can be quantified by the norms of derivatives of $ F_B u $ of order $ s $ and the norms of derivatives of $ u $ which are already estimated.
	
	Hence, we have proved the priori estimate when $ u $ is a smooth basic form. To obtain the regularity of the equation $\bar{F}_Bu=\omega$, we add an extra term onto $\bar{G}_B$ such that G$\mathring{{\rm a}}$rding's type inequality holds. For $ 0<\epsilon\leq1 $, defining norm $ \bar{G}_B^\epsilon $ on $ \mathcal{D}_B^{r',r} $ by
	\begin{equation*}
		\bar{G}_B^\epsilon\left(u,v\right):=\bar{G}_B\left(u,v\right)+\epsilon\left(\nabla u,\nabla v\right)\thicksim \bar{G}_B\left(u,v\right)+\epsilon\sum_{i=1}^{n}\sum_{l=1}^{L}\left(D_i\eta_lu,D_i\eta_lv\right),
	\end{equation*}
	then extend it to $ \tilde{\mathcal{D}}_{B,\epsilon}^{r',r} $ which is the completion of $ \mathcal{D}_B^{r',r} $ under $ \bar{G}_B^\epsilon $. Similarly, we can define $ \bar{G}_Q^\epsilon $ on $ \tilde{\mathcal{D}}_{Q,\epsilon}^{r',r} $ by the same way. Then there exist Friedrichs operators $ \bar{F}_B^\epsilon $ and $ \bar{F}_Q^\epsilon $ associated to $ \bar{G}_B^\epsilon $ and $ \bar{G}_Q^\epsilon $ respectively. Furthermore, $ \bar{F}_B^\epsilon $ is elliptic on $ \Omega_B^{\cdot,\cdot}\left(\overline{M}\right) $,  and the G$ \mathring{\rm a} $rding's type inequality $ \epsilon\|u\|_{1}^2\lesssim \bar{G}_Q^\epsilon\left(u,u\right) $ hold for all $ u\in\tilde{\mathcal{D}}_{Q,\epsilon}^{r',r} $. The arguments in the proof of $(\ref{estimate for dolbeault})$ go through without change to show
	\begin{align}\label{estimate for smooth u^epsilon}
		\|u^\epsilon\|_{s+1}^2\lesssim\|\omega\|_{B,s}^2,\quad \forall\epsilon>0,
	\end{align}
	for the smooth solution of the equation $ \bar{F}_B^\epsilon u^\epsilon=\omega $. Indeed, the proof of $ (\ref{estimate for dolbeault}) $ essentially depends on the Lemma \ref{mainl 1}-Lemma \ref{mainl 3} and the ellipticity of $ \bar{F}_B $, while all requirements are fulfilled since $ \bar{G}_Q\left(\cdot,\cdot\right)\leq \bar{G}_Q^\epsilon\left(\cdot,\cdot\right) $ and $ \bar{F}_B^\epsilon $ is also elliptic. 
	
	Moreover, the regularity of the equation $ \bar{F}_B^\epsilon u^\epsilon=\omega $ is clear by corresponding discussions in the proof of Theorem \ref{main theorem}. Now we are going to prove the regularity of original equation by $(\ref{estimate for smooth u^epsilon})$. For $ \epsilon>0 $, let $ u^\epsilon $ be the solution of $ \bar{F}_B^\epsilon u^\epsilon=\omega $. If $ \omega $ is smooth, then $ u^\epsilon $ is smooth and the estimates $(\ref{estimate for smooth u^epsilon}) $ still hold for all $ \epsilon\rightarrow0 $. Thus, we can choose sequences $ \epsilon_\nu\rightarrow0 $ such that $ u^{\epsilon_\nu} $ convergences in $ H_{B,s}^{r',r} $ for each integer $ s $ by Proposition $ \ref{Rellich and Sobolev lemma} $. For our purpose, we need to prove that $ u^\epsilon\rightarrow u $ in $ H_{B,0}^{r',r} $ when $ \epsilon\rightarrow0 $, then $ u\in H_{B,s}^{r',r} $ for all $ s $, thus $ u $ is smooth again by Proposition $ \ref{Rellich and Sobolev lemma} $.
	
	In fact, by means of the estimate $ (\ref{estimate for smooth u^epsilon}) $ we know that $ u^{\epsilon}\in H_{B,0}^{r',r} $ and $ \|u^{\epsilon}\|_{B}^2\leq\|u^{\epsilon}\|_{B,1}^2\lesssim\|\omega\|_{B}^2 $ uniformly as $ \epsilon\rightarrow0 $. In view of Banach-Saks Theorem, one can choose subsequence $ \{u^{\epsilon_\nu}\}_{\nu=1}^\infty $ such that its arithemetic mean converges to $ u' $ in $ H_{B,0}^{r',r} $. It remains to show that $ u'=u $, and it suffices to prove that $ u'\in {\rm Dom}\left(\bar{F}_B\right) $ and $ \bar{F}_Bu'=\omega $ i.e.,, $ \bar{G}_B\left(u',v\right)=\left(\omega,v\right)_B $ for all $ v\in\mathcal{D}_B^{r',r} $. In fact,
	\begin{align*}
		\left(\omega,v\right)_B=\frac{1}{m}\sum_{\nu=1}^{m}\bar{G}_B^{\epsilon_\nu}\left(u^{\epsilon_\nu},v\right)=\bar{G}_B\left(\frac{1}{m}\sum_{\nu=1}^{m}u^{\epsilon_\nu},v\right)+\frac{1}{m}\sum_{\nu=1}^{m}\epsilon_\nu\left(u^{\epsilon_\nu},v\right)_{B,1}\stackrel{{m}\rightarrow\infty}{\longrightarrow}\bar{G}_B\left(u',v\right),
	\end{align*}
	which gives the desired conclusion.
\end{proof}

With Theorem \ref{main theorem for dolbeault} in our possession, we are able to establish the Hodge decomposition for Hermitian foliations by repeating the procedures outlined in Theorem \ref{Hodge decomposition theorem}.

\begin{thm}\label{Hodge decomposition theorem for dolbeault}
	For $ 0\leq r',r\leq q $, let $ \mathcal{F} $ be a transversally oriented Hermitian foliation on a compact oriented manifold $ M $ with smooth basic boundary $ \partial M $, i.e., $ \rho $ is a basic function. Assume $ g_M $ is a bundle-like metric with basic mean curvature form $ \kappa $ and basic estimate holds in $ \mathcal{D}_Q^{r',r} $. There is an orthogonal decomposition
	\begin{equation}\label{Hodge decomposition for dolbeault}
		H_{B,0}^{r',r}={\rm Range}\left(\Box_{\bar{F}_B}\right)\oplus\mathcal{H}_B^{r',r}=\bar{\partial}_B\vartheta_B{\rm Dom}\left(\bar{F}_B\right)\oplus\vartheta_B\bar{\partial}_B{\rm Dom}\left(\bar{F}_B\right)\oplus\mathcal{H}_B^{r',r},
	\end{equation}
	where the kernel $ \mathcal{H}_B^{r',r} $ of $ \Box_{\bar{F}_B}:=\Box_B|_{\bar{F}_B} $ is a finite-dimensional space. Moreover, we have
	\begin{equation*}
		\Omega_B^{r',r}\left(\overline{M}\right)={\rm Im}\bar{\partial}_B\oplus{\rm Im}\vartheta_B\oplus\mathcal{H}_B^{r',r}.
	\end{equation*}
\end{thm}

The decomposition (\ref{Hodge decomposition for dolbeault}) also gives us a Neumann operator $N_B$ (we have slightly abused notation: the Neumann operator and Green operator are of the same
type, but not necessarily the same.) associated to $ \Box_B $ with properties analogous to the proposition \ref{proposition of N_B}. Then, a similar reasoning of Corollary \ref{existence theorem 1} infers the following the regularity of the equation $ \bar{\partial}_Bu=\omega $ under the conditions in Theorem \ref{Hodge decomposition theorem for dolbeault}.

\begin{cor}\label{existence theorem 1 for dolbeault}
	With the same hypotheses of Theorem \ref{Hodge decomposition theorem for dolbeault}. For $ 1\leq r\leq q $, let $ \omega\in H_{B,0}^{\cdot,r} $ satisfying $ \bar{\partial}_B\omega=0 $ and $ \omega\perp\mathcal{H}_B^{\cdot,r} $, then there exists a unique solution $ u $ of the equation $ \bar{\partial}_Bu=\omega $ which is orthogonal to $ {\rm Ker}\left(\bar{\partial}_B\right) $. Moreover, for any $ s\in\mathbb{Z}_+ $, if $ \omega\in\Omega_B^{\cdot,r}\left(\overline{M}\right) $, then $ u\in\Omega_B^{\cdot,r-1}\left(\overline{M}\right) $ and
	\begin{align}\label{estimate of solution of partial_B}
		\|u\|_{B,s}^2\lesssim\|\omega\|_{B,s}^2.
	\end{align}
	Moreover,
	$$ \mathcal{H}_B^{\cdot,r}\cong H_B^{\cdot,r}\left(\overline{M},\mathcal{F}\right):=\frac{\{u\in\Omega_B^{\cdot,r}\left(\overline{M}\right)\ |\ \bar{\partial}_Bu=0\}}{\bar{\partial}_B\Omega_B^{\cdot,r-1}\left(\overline{M}\right)} $$
	is a finite dimensional space.
\end{cor}

\subsection{Geometric characterize of basic estimate}
In this section, we wish to give the geometric condition of basic estimate holds in $ \mathcal{D}_Q^{r',r} $. The methods uesd here are similar to L. H{\"o}rmander in \cite{Hlv65JL}.

We need to introduce a quadratic form. For any real-valued $ w\in\Omega_B^0\left(\overline{M}\right) $, we have $ \partial_B\bar{\partial}_Bw=w_{\alpha\bar{\beta}}\omega^\alpha\wedge\bar{\omega}^\beta $, then
\begin{align}\label{levi form for dolbeault}
	L_w\left(\xi,\bar{\xi}\right):=w_{\alpha\bar{\beta}}\xi_\alpha\bar{\xi}_\beta
\end{align}
is a well-defined Levi form since $\partial_B\bar{\partial}_B=-\bar{\partial}_B\partial_B$, where $ \xi=\xi_\alpha V_\alpha $ is a transversal vector field. For subsequent utilization we give the local expression of the Levi form $ (\ref{levi form for dolbeault}) $ as follows. For smooth function $w'$,
\begin{align}\label{local expression of levi form}
	\partial_Q\bar{\partial}_Qw'=\partial_Q\left(\bar{V}_\beta\left(w'\right)\bar{\omega}^\beta\right)=\left(V_\alpha\bar{V}_\beta w'+\bar{a}_{\alpha\beta}^\gamma\bar{V}_\gamma w'\right)\omega^\alpha\wedge\bar{\omega}^\beta,
\end{align}
where $ \bar{a}_{\alpha\beta}^\gamma $ are given by $ \partial_Q\bar{\omega}^\gamma=\bar{a}_{\alpha\beta}^\gamma\omega^\alpha\wedge\bar{\omega}^\beta $. 
Hence $ (\ref{local expression of levi form}) $ gives the local expression of $(\ref{levi form for dolbeault})$ since $\partial_Q\bar{\partial}_Qw=\partial_B\bar{\partial}_Bw$ for basic function $w$.

As mentioned earlier, the Bochner formula is crucial for the sequel argument in this section. We are going to establish the Bochner formula in weighted case for later use which gives the unweighted case by letting weight function equal to zero. Here the weight $ L^2 $-spaces $ L^2\left(\Omega_B^{\cdot,\cdot},\phi\right) $ and $ L^2\left(\Omega_Q^{\cdot,\cdot},\phi\right) $ can be defined similarly as (\ref{weighted L2 space}). We will add superscript or subscript $ \phi $ denote all objects in weighted case. Obviously, the formal adjoint $\vartheta_B^\phi$ the of $ \bar{\partial}_B $ is $ \vartheta_B+{\rm grad}\phi\lrcorner $ which implies 
$$ \mathcal{D}_{B,\phi}^{\cdot,r}:={\rm Dom}\left({\bar{\partial}_{B,\phi}}^*\right)\cap\Omega_B^{\cdot,r}\left(\overline{M}\right)=\mathcal{D}_B^{\cdot,r}. $$

\begin{prop}
	For $ 0\leq r\leq q $, let $ \left(\left(x,z\right),U\right) $ be a local coordinate chart, then for $ v=v_{I,J}\omega^I\wedge\bar{\omega}^J\in\mathcal{D}_Q^{r',r}\left(U\right) $ we have
	\begin{align}\label{Bochner formula for dolbeault}
		\|\bar{\partial}_Qv\|_\phi^2+\|\vartheta_Q^\phi v\|_\phi^2=&\sum_{\alpha=1}^q\sum_{I,J}\|\bar{V}_\alpha v_{I,J}\|_\phi^2+\sum_{|K|=r-1}\int_{M}\left<\phi_{\alpha\bar{\beta}}v_{I,\alpha K},v_{I,\beta K}\right>e^{-\phi}\nonumber\\
		&+\sum_{|K|=r-1}\int_{\partial M}\left<\rho_{\alpha\bar{\beta}}v_{I,\alpha K},v_{I,\beta K}\right>e^{-\phi}+O\left(\left(E'_\phi\left(v\right)+\sum_{a=1}^{p}\|X_av\|_\phi\right)\|v\|_{\phi}\right),
	\end{align}
	where $E'_\phi\left(v\right)^2=\sum_{\alpha=1}^q\sum_{I,J}\|\bar{V}_\alpha v_{I,J}\|_\phi^2+\|v\|_\phi^2$.
\end{prop}

\renewcommand{\proofname}{\bf $ Proof $}
\begin{proof}
	We only prove the formula $(\ref{Bochner formula for dolbeault})$ in boundary chart since the interior expression is simpler, and let $U$ be a special boundary chart without loss of generality. By means of formula $ (\ref{divergence theorem for Q}) $, the local expressions of $ \bar{\partial}_Q $ and $ \vartheta_Q^\phi $ are given by 
	\begin{align}
		\bar{\partial}_Qv&=\bar{V}_\alpha\left(v_{I,J}\right)\bar{\omega}^\alpha\wedge\omega^I\wedge\bar{\omega}^J+\cdots,\label{local expression of partial_Q}\\
		\vartheta_Q^\phi v&=\left(-1\right)^{r'}\left(\bar{V}_\alpha\right)_\phi^*\left(v_{I,\alpha K}\right)\omega^I\wedge\bar{\omega}^K+\cdots,\label{local expression of vartheta_Q}
	\end{align}
	where $ \left(\bar{V}_\alpha\right)_\phi^*:=-V_\alpha+V_\alpha\left(\phi\right)+\cdots $ is the formal adjoint of $ \bar{V}_\alpha $ with respect to the inner product $ \left(\cdot,\cdot\right)_\phi $, $ K $ runs over all multi-indices with length $ r-1 $ and dots denote zero order terms.
	
	Then by $ (\ref{local expression of partial_Q}) $ we know that
	\begin{align}\label{bd1}
		\|\bar{\partial}_Qv\|_\phi^2=&\left(\bar{V}_\alpha\left(v_{I,J}\right)\bar{\omega}_\alpha\wedge\omega^I\wedge\bar{\omega}^J+\cdots,\bar{V}_\alpha\left(v_{I,J}\right)\bar{\omega}_\alpha\wedge\omega^I\wedge\bar{\omega}^J+\cdots\right)_\phi\nonumber\\
		=&\sum_{\alpha=1}^q\sum_{I,J}\|\bar{V}_\alpha v_{I,J}\|_\phi^2-\left(\bar{V}_\alpha v_{I,\beta K},\bar{V}_\beta v_{I,\alpha K}\right)_\phi+O\left(E'_\phi\left(v\right)\|v\|_\phi\right).
	\end{align}
	Let I denote the second term in R.H.S. of the equality $ (\ref{bd1}) $, due to $(\ref{local expression of vartheta_Q})$,
	\begin{align}\label{bd2}
		{\rm I}=&-\left(\left(\bar{V}_\beta\right)_\phi^*\bar{V}_\alpha v_{I,\beta K},v_{I,\alpha K}\right)_\phi+O\left(E'_\phi\left(v\right)\|v\|_{\phi}\right)\nonumber\\
		=&-\|\left(\bar{V}_\alpha\right)_\phi^*v_{I,\alpha K}\|_\phi^2-\left(\left[\left(\bar{V}_\beta\right)_\phi^*,\bar{V}_\alpha\right] v_{I,\beta K},v_{I,\alpha K}\right)_\phi+O\left(E'_\phi\left(v\right)\|v\|_\phi\right)\nonumber\\
		=&-\|\vartheta_Q^\phi v\|_\phi^2-\left(\left[\left(\bar{V}_\beta\right)_\phi^*,\bar{V}_\alpha\right]v_{I,\beta K},v_{I,\alpha K}\right)_\phi+O\left(E'_\phi\left(v\right)\|v\|_\phi\right).
	\end{align}
	Note that the boundary terms vanish in $ (\ref{bd2}) $ since $ v_{I,qK}\equiv0 $ on $ \partial M $ and $ \bar{V}_\alpha $ and $ V_\alpha $ are tangential when $ \alpha<q $ in the special boundary chart. Now we are going to handle the second term in R.H.S. of $ (\ref{bd2}) $ which we denoted by II. Since $ {\rm d}_Q^2=\partial_Q\bar{\partial}_Q+\bar{\partial}_Q\partial_Q+\partial_Q^2+\bar{\partial}_Q^2=-\frac{1}{2}\theta^\alpha\wedge\theta^\beta\wedge\nabla_{\pi^\perp\left(\nabla_{X_\alpha}^MX_\beta\right)} $ (see Lemma 9 in \cite{Kya08}), where $X_1,\cdots,X_{2q}$ are local frames of $Q$ and $\pi^\perp:={\rm Id}-\pi$. In view of $ (\ref{local expression of levi form}) $ there exist smooth functions $h_a$ for $a=1,\cdots,p$ such that
	\begin{align*}
		\left(\bar{V}_\beta V_\alpha v_{I,\beta K}+\bar{a}_{\beta\alpha}^\gamma\bar{V}_\gamma v_{I,\beta K}\right)-\left(\bar{V}_\alpha V_\beta v_{I,\beta K}+a_{\alpha\beta}^\gamma V_\gamma v_{I,\beta K}\right)=h_aX_av_{I,\beta K}.
	\end{align*}
	It turns out that
	\begin{align*}
		-\left[\left(\bar{V}_\beta\right)_\phi^*,\bar{V}_\alpha\right]v_{I,\beta K}&=\bar{V}_\alpha V_\beta\phi+a_{\alpha\beta}^\gamma V_\gamma v_{I,\beta K}-\bar{a}_{\beta\alpha}^\gamma\bar{V}_\gamma v_{I,\beta K}\\
		&=\phi_{\beta\bar{\alpha}}-a_{\alpha\beta}^\gamma\left(\bar{V}_\gamma\right)_\phi^*v_{I,\beta K}-\bar{a}_{\beta\alpha}^\gamma\bar{V}_\gamma v_{I,\beta K}+h_aX_av_{I,\beta K}.
	\end{align*}
	Hence
	\begin{align}\label{bd3}
		{\rm II}=\left(\phi_{\beta\bar{\alpha}}v_{I,\beta K},v_{I,\alpha K}\right)_\phi-\left(a_{\alpha\beta}^\gamma\left(\bar{V}_\gamma\right)_\phi^*v_{I,\beta K},v_{I,\alpha K}\right)_\phi+O\left(\left(E'_\phi\left(v\right)+\sum_{a=1}^{p}\|X_au\|_\phi\right)\|v\|_\phi\right),
	\end{align}
	At last, by III we denote the second item in R.H.S. of equality $ (\ref{bd3}) $. Then
	\begin{align}\label{bd4}
		{\rm III}=&\int_{\partial M}\left<a_{\alpha\beta}^\gamma V_\gamma\left(\rho\right)v_{I,\beta K},v_{I,\alpha K}\right>e^{-\phi}+O\left(E'_\phi\left(v\right)\|v\|_\phi\right)\nonumber\\
		=&\int_{\partial M}\left<\rho_{\beta\bar{\alpha}}v_{I,\beta K},v_{I,\alpha K}\right>e^{-\phi}+O\left(E'_\phi\left(v\right)\|v\|_\phi\right).
	\end{align}
	The first line follows from divergence theorem, and the second line is valid since $ \bar{V}_\alpha\left(\rho\right)=0 $ for $ \alpha<q $ and $ v_{I,qK}=0 $ on $ \partial M $ in $ U'' $. Combining $ (\ref{bd1})\sim(\ref{bd4}) $ gives the conclusion.
\end{proof}

Let $ \lbrace\lambda_\alpha\rbrace_{\alpha=1}^{q-1} $ be the set of eigenvalues of the Levi form $L_\rho$ on the tangent space $ \xi_q=0 $ at $ p\in\partial M\cap U'' $ where $ U'' $ is a special boundary chart. Set $ \lambda_\alpha^+:=\rm max\lbrace0,\lambda_\alpha\rbrace $ and $ \lambda_\alpha^-:=\rm max\lbrace-\lambda_\alpha,0\rbrace $. We recall a definition which is originally introduced by L. H$ \ddot{\rm o} $rmander in \cite{Hlv65JL} as follows.
\begin{definition}
	A real valued function $ w\in C^2\left(M\right) $ is said to be satisfied $ Z_r $-condition at a point $ x_0 $, provided that $ {\rm grad}w\left(x_0\right)\neq0 $ and there are at least $ q-r $ positive eigenvalues or at least $ r+1 $ negative eigenvalues of the Levi form $ L_w $ on the plane $ \sum_{\alpha=1}^{q}V_\alpha\left(w\right)\xi_\alpha=0 $. In particular, we say the boundary $ \partial M $ satisfies $ Z_r $-condition if its defining function $ \rho $ satisfies $ Z_r $-condition at every point of $ \partial M $.
\end{definition}

\begin{thm}\label{geometric condition for basic estimate for dolbeault}
	For $ 0\leq r',r\leq q $, basic estimate holds in $ \mathcal{D}_Q^{r',r} $ if and only if $ \partial M $ satisfies $ Z_r $-condition.
\end{thm}

\begin{remark}
	The definition is independent of the choice of the metric on $ M $.
\end{remark}

We shall prove Theorem \ref{geometric condition for basic estimate for dolbeault} using the following lemma. Since the proof is similar to that of the expression $ (\ref{Bochner formula for dolbeault}) $, we formulate this lemma without proof.

\begin{lemma}\label{integration by parts for first order term for dolbeault}
	Given $ p\in\partial M $ and let $ \left(\left(x,z\right),U''\right) $ be a special boundary chart near p, then for any $ v\in C_c^\infty\left(U''\right) $ we have
	\begin{equation*}
		\|\bar{V}_\alpha v\|_\phi^2\geq-\int_{M}\phi_{\alpha\bar{\alpha}}|v|^2e^{-\phi}-\int_{\partial M}\rho_{\alpha\bar{\alpha}}|v|^2e^{-\phi}
		+O\left(\left(E'_\phi\left(v\right)+\sum_{a=1}^{p}\|X_av\|_\phi\right)\|v\|_{\phi}\right),
	\end{equation*}
	recalling that $E'_\phi\left(v\right)=\sum_{\alpha=1}^q\|\bar{V}_\alpha v\|_\phi^2+\|v\|_\phi^2$.
\end{lemma}

\renewcommand{\proofname}{\bf $ Proof\ of\ Theorem\ \ref{geometric condition for basic estimate for dolbeault} $}
\begin{proof}
	Sufficiency. For any $ p\in\partial M $, let $ V'' $ be a special boundary chart near $ p $ on which Lemma $ \ref{integration by parts for first order term for dolbeault} $ holds such that $ \left(x,z\right)\left(p\right)=0 $. By a unitary transformation, we can always arrange that
	the Levi form to be diagonal at $ p $ for $ \xi=\left(\xi_\alpha\right)_{1\leq\alpha\leq q}\in\mathbb{C}^{q} $,
	\begin{equation*}
		\sum_{\alpha,\beta=1}^{q-1}\rho_{\alpha\bar{\beta}}\left(0\right)\xi_\alpha \bar{\xi}_\beta=\sum_{\alpha=1}^{q-1}\lambda_\alpha|\xi_\alpha|^2,
	\end{equation*}
	where $ \lambda_\gamma<0 $ for $ \gamma=1,\cdots,\mu $ and $ \lambda_\gamma\geq0 $ for $ \gamma>\mu $, i.e., $ \rho_{\alpha\bar{\beta}}=\lambda_\alpha\delta_{\alpha\beta}+b_{\alpha\beta} $ and $ b_{\alpha\beta}\left(0\right)=0 $. Then for $ v=v_{I,J}\omega^I\wedge\bar{\omega}^J\in\mathcal{D}_Q^{r',r}\left(U''\right) $ and any $ 0<\varepsilon<<1 $, we choose $ U''\Subset V'' $ such that 
	\begin{align}\label{error of diagonalize}
		\sum_{\alpha,\beta=1}^{q}\sum_{|I|=r',|J|=r}\int_{U''\cap\partial M}|b_{\alpha\beta}v_{I,J}|^2<\frac{\varepsilon}{8}E\left(v\right)^2.
	\end{align}
	
	From the Bochner formula $ \left(\ref{Bochner formula for dolbeault}\right) $ for $v$ and $\phi=0$ we read off
	\begin{align}\label{first inequality}
		G_Q\left(v,v\right)\geq&\varepsilon\sum_{\alpha=1}^q\sum_{I,J}\|\bar{V}_\alpha v_{I,J}\|^2+\left(1-\varepsilon\right)\sum_{\alpha=1}^{\mu}\sum_{I,J}\|\bar{V}_\alpha v_{I,J}\|^2+\sum_{|K|=r-1}\int_{\partial M}\left<\rho_{\alpha\bar{\beta}}v_{I,\alpha K},v_{I,\beta K}\right>\nonumber\\
		&+O\left(\left(E\left(v\right)+\sum_{a=1}^{p}\|X_av\|\right)\|v\|\right),
	\end{align}
	Using Lemma $ \ref{integration by parts for first order term for dolbeault} $ for $\phi=0$ and adding, we obtain that 
	\begin{align}\label{sum of the integration by parts for first order term}
		\sum_{\alpha=1}^\mu\sum_{I,J}\int_{\partial M}-\rho_{\alpha\alpha}|v_{I,J}|^2\leq\sum_{\alpha=1}^{\mu}\sum_{I,J}\|\bar{V}_\alpha v_{I,J}\|^2+O\left(\left(E\left(v\right)+\sum_{a=1}^{p}\|X_av\|\right)\|v\|\right).
	\end{align}
	Then we substitute (\ref{sum of the integration by parts for first order term}) into the inequality (\ref{first inequality}), and by $ (\ref{error of diagonalize}) $ we have 
	\begin{align*}
		G_Q\left(v,v\right)\geq\varepsilon\sum_{\alpha=1}^q\sum_{I,J}\|\bar{V}_\alpha v_{I,J}\|^2&+\sum_{I,J}\int_{\partial M}A_J\left(\varepsilon;\lambda\right)|v_{I,J}|^2+O\left(\left(E\left(v\right)+\sum_{a=1}^{p}\|X_av\|\right)\|v\|\right)-\frac{\varepsilon}{2}E\left(v\right)^2,
	\end{align*}
	where
	\begin{align*}
		A_J\left(\varepsilon;\lambda\right)=\left(\left(1-\varepsilon\right)\sum_{\alpha=1}^{q-1}\lambda_\alpha^-+\sum_{\alpha\in J}\lambda_\alpha\right)=\left(\left(1-\varepsilon\right)\sum_{j\notin J}\lambda_\alpha^-+\sum_{\alpha\in J}\lambda_\alpha^+\right),	
	\end{align*}
	which is positive by the hypothesis either $ \lambda_\alpha>0 $ for some $ \alpha\in J $ or $ \lambda_\alpha<0 $ for some $ \alpha\notin J $ when $ \varepsilon<<1 $. It implies that there exists constant $ C_\varepsilon>0 $ such that
	\begin{align*}
		G_Q\left(v,v\right)\geq&\varepsilon\left(\sum_{\alpha=1}^q\sum_{I,J}\|\bar{V}_\alpha v_{I,J}\|^2+\sum_{I,J}\int_{\partial M}|v_{I,J}|^2\right)-\frac{3\varepsilon}{4}E\left(v\right)^2-C_\varepsilon\|v\|^2+O\left(\sum_{a=1}^{p}\|X_av\|\|v\|\right)\nonumber\\
		\geq&\frac{\varepsilon}{4}E\left(v\right)^2-C_\varepsilon\|v\|^2+O\left(\sum_{a=1}^{p}\|X_av\|\|v\|\right).
	\end{align*}
	Using a partition of unity we know that basic estimate holds for $ u\in\mathcal{D}_Q^{r',r} $ if $ \partial M $ is $ Z_r $-convex.
	
	Necessity. For any $p\in\partial M$, we choose a special boundary coordinate $\left(\left(x,z\right),U''\right)$ near $p$ such that $\left(x,z\right)\left(p\right)=0$ and 
	\begin{align*}
		U''\cap\overline{M}=\{\left(x,z\right)\in U''\ |\ {\rm Im}z_q=y_{2q}\geq\rho_1\left(y'\right)\},
	\end{align*}
	where $y'=\left(y_1,\cdots,y_{2q-1}\right)$ and we assume that $\rho_1\in C^\infty\left(\overline{M}\right)$ vanishes to the $2^{nd}$-order at $p$ by Implicit Function Theorem. It turns out that $\rho=\left(\rho_1-y_{2q}\right)\left(1+O\left(|z|\right)\right)$. By a unitary transformation we can achieve that the levi form
	\begin{align*}
		\sum_{\alpha,\beta=1}^{q-1}\frac{\partial^2\rho_1\left(0\right)}{\partial z_\alpha\partial\bar{z}_\beta}z_\alpha\bar{z}_{\beta}=\sum_{\alpha,\beta=1}^{q-1}\frac{\partial^2\rho\left(0\right)}{\partial z_\alpha\partial\bar{z}_\beta}z_\alpha\bar{z}_{\beta}=\sum_{\alpha=1}^{q-1}\lambda_\alpha|z_\alpha|^2,
	\end{align*}
	which gives
	\begin{align}\label{Taylor formula}
		\rho_1\left(z\right)=\sum_{\alpha=1}^{q-1}\lambda_\alpha|z_\alpha|^2+P(z')+O\left(|z_q||z'|+|y_{2q-1}|^2+|z|^3\right),
	\end{align}
	where $z'=\left(z_1,\cdots,z_{q-1}\right)$ and $ P(z') $ is a homogeneous polynomial of degree 2.
	
	Let $ J $ be a multi-index of length $ r $ with $q\notin J$ and let $I$ be an multi-index of length $ r' $. Choose $ v^t=v^t_{I',J'}\omega^{I'}\wedge\bar{\omega}^{J'}\in\mathcal{D}_Q^{r',r}\left(U''\cap\overline{M}\right) $ such that $ v^t_{I',J'}=\psi_1\left(x\right)\psi_2\left(tz\right)e^{it^2z_q} $ if $ \left(I',J'\right)=\left(I,J\right) $ and $ v^t_{I',J'}=0 $ if $ \left(I',J'\right)\neq\left(I,J\right) $ where $ \psi_1\in C_c^\infty\left(\mathbb{R}^p\right) $ and $\psi_2\in C_c^\infty\left(\mathbb{C}^q\right)$ such that $\int_{\mathbb{R}^p}|\psi_1|^2=1$.
	
	In slight abuse of notation we let $y'=ty'$ and $y_{2q}=t^2y_{2q}$, then
	\begin{align}\label{g1}
		t^{2q-1}\sum_{\alpha=1}^{q}\bigg\|\frac{\partial}{\partial\bar{z}_\alpha}\left(\psi_1\left(x\right)\psi\left(tz\right)e^{it^2z_q}\right)\bigg\|^2=\sum_{\alpha=1}^{q}\int_{\mathbb{R}^{2q-1}}\int_{\rho_2\left(y'\right)+O\left(\frac{1}{t}\right)}^\infty\bigg|\frac{\partial\psi_2}{\partial\bar{z}_\alpha}\left(y',\frac{y_{2q}}{t}\right)\bigg|^2e^{-2y_{2q}}dy,
	\end{align}
	where $\rho_2$ is the second order part of $(\ref{Taylor formula})$. And $(\ref{g1})$ convergences to
	\begin{align}\label{g2}
		\frac{1}{2}\int_{\mathbb{R}^{2q-1}}\sum_{\alpha=1}^{q}\bigg|\frac{\partial\psi_2}{\partial\bar{z}_\alpha}\left(y',0\right)\bigg|^2e^{-2\rho_2\left(y'\right)}dy'.
	\end{align}
	Similarly, we have
	\begin{align}
		&t^{2q-1}\|v^t\|^2=O\left(\frac{1}{t^2}\right)\stackrel{t\rightarrow\infty}{\longrightarrow}0,\quad t^{2q-1}\sum_{a=1}^{p}\bigg\|\frac{\partial v^t}{\partial x_a}\bigg\|^2=O\left(\frac{1}{t^2}\right)\stackrel{t\rightarrow\infty}{\longrightarrow}0,\label{g3}\\
		&t^{2q-1}\int_{\partial M}\left<\rho_{\alpha\bar{\beta}}v^t_{I,\alpha k},v_{I,\beta K}^t\right>\stackrel{t\rightarrow\infty}{\longrightarrow}\sum_{\alpha\in J}\lambda_\alpha\int_{\mathbb{R}^{2q-1}}|\psi_2\left(y',0\right)|^2e^{-2\rho_2\left(y'\right)}dy',\label{g4}\\
		&t^{2q-1}\int_{\partial M}|v^t|^2\stackrel{t\rightarrow\infty}{\longrightarrow}\int_{\mathbb{R}^{2q-1}}|\psi_2\left(y',0\right)|^2e^{-2\rho_2\left(y'\right)}dy'.\label{g5}
	\end{align}
	
	Since $U''$ is a special boundary chart, there exists $0<\varepsilon<<1$ such that 
	$$\|\bar{V}_\alpha v_{I,J}\|^2\leq\left(1+\varepsilon\right)\|\frac{\partial v_{I,J}}{\partial \bar{z}_\alpha}\|^2$$
	by shrinking $U''$ if necessary. In view of basic estimate for $v\in\mathcal{D}_Q^{r',r}$ and Bochner formula (\ref{Bochner formula for dolbeault}), there is a constant $C>0$ so that
	\begin{align}\label{estimate for boundary term}
		C\int_{\partial M}|v|^2\leq\left(1+\varepsilon\right)\sum_{\alpha=1}^{q}\|\frac{\partial v_{I,J}}{\partial \bar{z}_\alpha}\|^2+\int_{\partial M}\left<\rho_{\alpha\bar{\beta}}v_{I,\alpha k},v_{I,\beta K}\right>+O\left(\|v\|^2+\sum_{a=1}^{p}\|X_av\|\|v\|\right).
	\end{align}
	Then substituting $v^t$ into the estimate $(\ref{estimate for boundary term})$ and by $(\ref{g2})\sim(\ref{g5})$ we obtain
	\begin{align}\label{g6}
		\frac{2}{1+\varepsilon}\left(C-\sum_{\alpha\in J}\lambda_\alpha\right)\int_{\mathbb{R}^{2q-1}}|\psi_2\left(y',0\right)|^2e^{-2\rho_2\left(y'\right)}dy'\leq\int_{\mathbb{R}^{2q-1}}\sum_{\alpha=1}^{q}\bigg|\frac{\partial\psi_2}{\partial\bar{z}_\alpha}\left(y',0\right)\bigg|^2e^{-2\rho_2\left(y'\right)}dy'.
	\end{align}
	
	Now we choose $\psi_2\left(z\right)=\psi_3\left(z'\right)\psi_3\left(z_q\right)e^{P\left(z'\right)}$ where $\psi_2\in C_c^\infty\left(\mathbb{C}^{q-1}\right)$ and $\psi_3\in C_c^\infty\left(\mathbb{C}\right)$ such that $\frac{\partial\psi_3}{\partial\bar{z}_q}=0$ when $y_{2q}=0$, then by $(\ref{g6})$
	\begin{align*}
		\frac{2}{1+\varepsilon}\left(C-\sum_{\alpha\in J}\lambda_\alpha\right)\int_{\mathbb{C}^{q-1}}|\psi_3\left(z'\right)|^2e^{-2L\left(z'\right)}dy'\leq\int_{\mathbb{C}^{q-1}}\sum_{\alpha=1}^{q-1}\bigg|\frac{\partial\psi_3}{\partial\bar{z}_\alpha}\left(z'\right)\bigg|^2e^{-2L\left(z'\right)}dy',
	\end{align*}
	where $L\left(z'\right)=\sum_{\alpha=1}^{q-1}\lambda_\alpha|z_\alpha|^2$. According to Lemma 3.2.2 in \cite{Hlv65JL},
	\begin{align*}
		\frac{1}{1+\varepsilon}\sum_{\alpha\in J}\lambda_\alpha+\sum_{\alpha=1}^{q-1}\lambda_\alpha^->0
	\end{align*}
	which gives the conclusion by letting $\varepsilon\rightarrow0$.
\end{proof}

We finish this section with a remark on the necessity of the assupmtion that basic estimate holds in $\mathcal{D}_Q^{r',r}$. In the context of Riemannian foliations, the gradient term of the Bochner formula (\ref{Bochner-Kodaira formula}) for basic forms could control all first order derivatives of forms, then it can be easily deduced that $G_B$ is $H_{B,1}$-coercive on the space $\tilde{\mathcal{D}}_B^\cdot$. This serves as the keystone for demonstrating the regularity and estimate of the solution of the equation $F_Bu=\omega$. However, in the case of Hermitian foliations, except for the situation with interior estimates, the gradient term of the Bochner formula (\ref{Bochner formula for dolbeault}) does not have the capability to control all derivatives. Consequently, the boundary integral can not be handled by the same argument in the proof of Corollary \ref{H1e}, requiring us to rely on the basic estimate for $\mathcal{D}_Q^{r',r}$.

\subsection{Existence theorems}
In this short section, we will give two existence theorems via weighted $L^2$-method. The following definitions are used to provide positivity.

\begin{definition}
	Let $ \rho $ be a boundary defining function for a manifold $ M $, the boundary $ \partial M $ is said to be transversely $ r $-convex if the quadratic form $ L_{r,\rho}\left(u,\bar{u}\right):=\rho_{\alpha\bar{\beta}}u_{I,\alpha K}\bar{u}_{I,\beta K} $ is nonnegative on $ \partial M $ for each $ u=u_{I,J}\omega^I\wedge\bar{\omega}^J\in\mathcal{D}_B^{r',r} $ and for all multi-indices $ K $ with $|K|=r-1$. In particular, we call the boundary $ \partial M $ is strictly transversely $r$-convex if $L_{r,\rho}>0$ on $\partial M$.
\end{definition}

\begin{definition}
	A function $ \varphi\in C^2\left(M\right) $ is said to be transversely $ r $-plurisubharmonic provided that the quadratic form $ L_{r,\varphi}\left(v,\bar{v}\right):=\varphi_{\alpha\bar{\beta}}v_{I,\alpha K}\bar{v}_{I,\beta K}\geq0 $ on $ M $ for all $ v=v_{I,J}\omega^I\wedge\bar{\omega}^J\in\Omega_B^{r',r}\left(M\right) $ and all multi-indices $K$ of length $r-1$. We call $ \varphi $ is strictly transversely $ r $-plurisubharmonic if $ L_{r,\varphi}>0 $ on $ M $. 
\end{definition}

\begin{remark}
	$ (1) $ The definitions are independent of the choices of the local frames. In fact, they only depend on any sum of $ r $ eigenvalues of the corresponding quadratic form.
	
	$ (2) $ The strictly transverse $ r $-convexity of the boundary $ \partial M $ is more restrictive than  $Z_r$-condition when $r\geq1$.
\end{remark}

Now we are in a position to give the vanishing theorem for basic Dolbeault cohomology $ H_B^{\cdot,\cdot}\left(\overline{M},\mathcal{F}\right) $.

\begin{thm}\label{existence theorem 2 for dolbeault}
	With the same hypotheses in Theorem \ref{Hodge decomposition theorem for dolbeault}. For $ 1\leq r\leq q $, we further assume that $ \partial M $ is transversely $ r $-convex and there is a strictly transversal $ r $-plurisubharmonic function $ \varphi\in\Omega_B^0\left(\overline{M}\right) $. Then for any $ \omega\in\Omega_B^{\cdot,r}\left(\overline{M}\right) $ with $ \bar{\partial}_B\omega=0 $, there exists $ u\in\Omega_B^{\cdot,r-1}\left(\overline{M}\right) $ such that $ \bar{\partial}_Bu=\omega $ and the estimate $ (\ref{estimate of solution of partial_B}) $ is also satisfied for every $ s\in\mathbb{Z}_+ $.
\end{thm}
\renewcommand{\proofname}{\bf $ Proof $}
\begin{proof}
	The proof follows exactly the same line as that of Theorem \ref{existence theorem 2}. The key point is to establish the analogue of estimate $(\ref{bochner for harmonic form and weight varphi})$. Indeed, using a partition of unity $\{\eta_l\}_{l=1}^N$, we read off the following estimate from Bochner formula $(\ref{Bochner formula for dolbeault})$ for $u=u_{I,J}\omega^I\wedge\bar{\omega}^J\in\mathcal{H}_B^{\cdot,r}$ and $r$-convexity of $\partial M$:
	\begin{align}\label{bochner for dolbeault harmonic form and weight varphi}
			\|V_\alpha\left(\varphi\right)\eta_lu_{I,\alpha K}\|_\varphi^2+O\left(\|u\|_\varphi^2\right)\geq\|\bar{\partial}_Q\eta_lu\|_\varphi^2+\|\vartheta_Q\eta_lu\|_\varphi^2\geq\int_{M}\left<\varphi_{\alpha\bar{\beta}}\eta_lu_{I,\alpha k},\eta_lu_{I,\beta k}\right>e^{-\varphi}.
		\end{align}
	One may replace $(\ref{bochner for harmonic form and weight varphi})$ by $(\ref{bochner for dolbeault harmonic form and weight varphi})$, then the rest part can be given by the same argument as in the proof of Theorem \ref{existence theorem 2} if we note that $\|u\|_{\varphi}^2\leq\sum_{l=1}^{N}\|\eta_lu\|_{B,\varphi}^2$.
\end{proof}

The reasoning of Theorem \ref{existence theorem 3} applied to $\bar{\partial}_B$-equations also shows the following $ L^2 $-existence theorem for the operator $ \bar{\partial}_B $.

\begin{thm}\label{existence theorem 3 for dolbeault}
	For $ 1\leq r\leq q $, let $ M $ be an oriented manifold with transversally oriented Hermitian foliation, and $ g_M $ is bundle-like metric with basic mean curvature form $ \kappa $. If there exists a strictly transversal $ r $-plurisubharmonic exhaustion basic function $ \varphi $, then for any $ \omega\in L^2_{loc}\left(M,\Omega_B^{\cdot,r}\right) $ with $ \bar{\partial}_B\omega=0 $, there is a basic form $ u\in L^2_{loc}\left(M,\Omega_B^{\cdot,r-1}\right) $ such that $ \bar{\partial}_Bu=\omega $.
\end{thm}


\subsection{Global regularity of $\bar{\partial}_B$-equations}
In \cite{Kjj73}, J. J. Kohn proved the global regularity for the solution of $\bar{\partial}$-equation in pseudoconvex domain in $\mathbb{C}^n$ by L. H{\"o}rmander's weighted $L^2$-method. In this section, we will deduce a regularity theorem for the equation $ \bar{\partial}_Bu=\omega $ by similar arguments. 

Note that the proof of Corollary \ref{existence theorem 1 for dolbeault} essentially depends on the hypothesis that the boundary $ \partial M $ satisfies $ Z_r $-condition. However, we can obtain a similar result under alternative geometric hypothesis in the weighted setup.

For $ 0\leq r',r\leq q $, we define the weighted form $ \bar{G}_{B,\phi} $ on $ \mathcal{D}_B^{r',r} $ by
\begin{align*}
	\bar{G}_{B,\phi}\left(u,v\right)=\left(\bar{\partial}_Bu,\bar{\partial}_Bv\right)_{B,\phi}+\left(\vartheta_B^\phi u,\vartheta_B^\phi v\right)_{B,\phi}+\left(u,v\right)_{B,\phi}.
\end{align*}
Obviously, all these weighted and unweighted norms are equivalent when $ \overline{M} $ is compact. We can therefore identify the completion of $ \mathcal{D}_B^{r',r} $ under $\bar{G}_{B,\phi}$ with $ \tilde{\mathcal{D}}_B^{r',r} $. Correspondingly, one may define the weighted $ Q $-norm $\bar{G}_{Q,\phi}$. Henceforth, we add the subscript or supperscript $ \phi $ referring the objects we have defined before in the weighted context. 

It's easy to see that the operators $ \bar{F}_B^\phi $ and $ \Box_B^\phi $ satisfying the proposition \ref{expression of $ F_B $} for the weighted setup by the same proof as unweighted case. To prove Corollary \ref{existence theorem 1 for dolbeault}, we need further preparations. The following proposition plays the same role as G$ \mathring{\rm a} $rding's type inequality and basic estimate in the proof of Theorem \ref{main theorem for dolbeault}.

\begin{prop}\label{weaker condition of priori estimate}
	For $ 0\leq r',r\leq q $, let $ M $ be an oriented Riemannian manifold with compact basic boundary $ \partial M $, suppose that there is an exhaustion basic function $ \varphi' $ satisfying $ Z_r $-condition outside a compacat $ K\Subset M $. One can find a basic function $ \varphi $ such that for each $ T>0 $ and any $ s\in\mathbb{N} $, there exists a constant $ C_T>0 $ s.t. for all $ u\in\mathcal{D}_B^{r',r} $ we have
	\begin{align}\label{priori estimate in weigthed setup}
		T\|D_t^I\eta''u\|_{T\varphi}^2\lesssim\|\bar{\partial}_QD_t^I\eta''u\|_{T\varphi}^2+\|\vartheta_Q^{T\varphi}D_t^I\eta''u\|_{T\varphi}^2+\|u\|_{B,s}^2+C_T\|u\|_{B,s-1}^2,
	\end{align}
	where $ \eta'' $ is any smooth function with compact support in a special boundary chart $ U'' $ and recall that $ D_t^I $ is a tagential differentiation of order $ s $. Furthermore, the estimate still holds if we replace $ \eta'' $ and $ D_t^I $ by $ \eta' $ and $ D^I $, where $ \eta' $ is any smooth function in the interior chart. In particular,
	\begin{align}\label{priori estimate in weigthed setup for s=0}
		T\|u\|_{B,T\varphi}^2\lesssim\|\bar{\partial}_Bu\|_{B,T\varphi}^2+\|\vartheta_B^{T\varphi}u\|_{B,T\varphi}^2+\|u\|_{B,T\varphi}^2+C_T\|u\|_{-1}^2.
	\end{align}
\end{prop}
\begin{proof}
	Since $ \varphi' $ is an exhaustion function, we can choose $ \nu $ sufficiently large such that $ M_\nu\cap{\rm supp}\left(\eta''\right)\neq\emptyset $ and $ K\subseteq M_\nu $, where $ M_\nu:=\{\varphi'<\nu\} $. For any $ u\in\mathcal{D}_B^{r',r} $, we will prove the estimate $ (\ref{priori estimate in weigthed setup}) $ locally in $ \overline{M}_\nu\cap U'' $ for $ u^\nu:=\frac{\nu-\varphi'}{\nu}u|_{M_\nu} $. Let $ \nu\rightarrow\infty $, then the desired estimate for $ u $ can be given by using a partition of unity.
	
	For any given point $ p\in\overline{M}_\nu\cap U'' $, choosing local coordinate $ \left(\left(x,z\right),U''\right) $ s.t. $ \left(x,z\right)\left(p\right)=0 $, and assume that $ \bar{V}_\alpha\left(\varphi'\right)=0 $ for $\alpha<q$. By a unitary transformation
	\begin{align*}
		\sum_{\alpha,\beta=1}^{q-1}\varphi_{\alpha\bar{\beta}}'\left(0\right)\xi_\alpha \bar{\xi}_\beta=\sum_{\alpha=1}^{q-1}\lambda_\alpha|\xi_\alpha|^2,
	\end{align*}
	where $ \xi=\left(\xi_\alpha\right)_{1\leq\alpha\leq q}\in\mathbb{C}^{q} $. 
	Let $ \lambda_\gamma<0 $ for $ \gamma\leq\sigma $ and $ \lambda_\gamma\geq0 $ for $ \gamma>\sigma $. For any multi-index $ I $ of order $ s $ and $ 0<\varepsilon<<1 $, since $ D_t^I\eta''u^\nu\in\mathcal{D}_{Q,\nu}^{r',r}:={\rm Dom}\left(\left(\bar{\partial}_Q|_{M_\nu}\right)^*\right)\cap\Omega_Q^{r',r}\left(\overline{M}\right) $, from the weighted Bochner formula $ \left(\ref{Bochner formula for dolbeault}\right) $ in $ M_\nu $ we obtain
	\begin{align}\label{Bochner inequality in M_nu}
		\|\bar{\partial}_QD_t^I\eta''u^\nu\|_{T\varphi'}^2+\|\vartheta_Q^{T\varphi'}D_t^I\eta''u^\nu\|_{T\varphi'}^2\geq&\varepsilon\sum_{\alpha=1}^{q}\sum_{I,J}\int_{M_\nu}|\bar{V}_\alpha D_t^I\eta''u^\nu_{I,J}|^2e^{-T\varphi'}\nonumber\\
		&+\left(1-\varepsilon\right)\sum_{\alpha=1}^{\sigma}\sum_{I,J}\int_{M_\nu}|\bar{V}_\alpha D_t^I\eta''u^\nu_{I,J}|^2e^{-T\varphi'}+O\left(\|u^\nu\|_{B,s}^2\right)\nonumber\\
		&+\sum_{|K|=r-1}\int_{M_\nu}\left<T\varphi'_{\alpha\bar{\beta}}D_t^I\eta''u^\nu_{I,\alpha K},D_t^I\eta''u^\nu_{I,\beta K}\right>e^{-T\varphi'}.
	\end{align}
	
	By Lemma $ \ref{integration by parts for first order term for dolbeault} $ and adding index $ \alpha,I,J$,
	\begin{align}\label{sum of the integration by parts for first order term in weighted setup}
		\sum_{\alpha=1}^\sigma\sum_{I,J}\int_{ M_\nu}-T\varphi'_{\alpha\bar{\alpha}}|D_t^I\eta''u_{I,J}^\nu|^2e^{-T\varphi'}\leq\sum_{\alpha=1}^{\sigma}\sum_{I,J}\int_{M_\nu}|\bar{V}_\alpha D_t^I\eta''u_{I,J}^\nu|^2e^{-T\varphi'}+O\left(E'_{T\varphi'}\left(D_t^I\eta''u^\nu\right)\|u^\nu\|_{B,s}\right),
	\end{align}
	where $ E'_{T\varphi'}\left(D_t^I\eta''u^\nu\right)^2=\sum_{\beta=1}^{q}\sum_{I,J}\|\bar{V}_\beta D_t^I\eta''u_{I,J}^\nu\|_{T\varphi'}^2+\|D_t^I\eta''u^\nu\|_{T\varphi'}^2 $.
	
	Substituting (\ref{sum of the integration by parts for first order term in weighted setup}) into the estimate (\ref{Bochner inequality in M_nu}) gives 
	\begin{align}\label{lower estimate for Bochner}
		\|\bar{\partial}_QD_t^I\eta''u^\nu\|_{T\varphi'}^2+\|\vartheta_Q^{T\varphi'}D_t^I\eta''u^\nu\|_{T\varphi'}\geq T\int_{M_\nu}B_\varepsilon\left(\varphi';u^\nu,u^\nu\right)e^{-T\varphi'}+O\left(\|u^\nu\|_{B,s}^2\right),
	\end{align}
	where
	\begin{align*}
		B_\varepsilon\left(\varphi';u^\nu,u^\nu\right)=-\left(1-\varepsilon\right)\sum_{\alpha=1}^{\sigma}\sum_{I,J}\varphi'_{\alpha\bar{\alpha}}|D_t^I\eta''u_{I,J}^\nu|^2+\sum_{\alpha,\beta=1}^{q}\sum_{|K|=r-1}\varphi'_{\alpha\bar{\beta}}\left<D_t^I\eta''u^\nu_{I,\alpha K},D_t^I\eta''u^\nu_{I,\beta K}\right>,
	\end{align*}
	which is positive at $ p $ if we replace $ \varphi' $ by $ e^{A\varphi'} $ for sufficient large constant $ A>0 $ whenever $ 0<\varepsilon<<1 $ by the hypothesis that $ \varphi' $ satisfies $ Z_r $-condition and (parallel arguments in) Lemma 3.3.3 in \cite{Hlv65JL}. Furthermore,
	\begin{align}\label{A_r condition for varphi}
		\varphi:=e^{A\varphi'}
	\end{align}
	still satisfies $ Z_r $-condition and exhaustibility, it turns out that estimate $ (\ref{lower estimate for Bochner}) $ holds for weight $ \varphi $. Then there exists a sufficiently small neighborhood $ V_p''\subseteq\overline{M}_\nu\cap U'' $ such that $ B_\varepsilon\left(\varphi;u^\nu,u^\nu\right) $ is positive in $ V_p'' $. Due to $ {\rm supp}\left(\eta''\right)\cap\partial M_\nu $ is compact, we can choose open subset $ K' $ satisfying $ K\Subset K'\Subset M_\nu $ and there is a maximal constant $A$ such that $ B_\varepsilon\left(\varphi;u^\nu,u^\nu\right) $ is positive in the compact set $ \left(K'\right)^c\cap{\rm supp}\left(\eta''\right)\neq\emptyset $ for such $ \varphi $ defined as $ (\ref{A_r condition for varphi}) $, where $\left(K'\right)^c$ is the complement of $K'$ w.r.t. $\overline{M}_\nu$. It yields that for $ 0<\varepsilon<<1 $
	\begin{align}\label{lower estimate for B on K'}
		T\int_{M_\nu\setminus K'}B_\varepsilon\left(\varphi;u^\nu,u^\nu\right)e^{-T\varphi}\gtrsim T\sum_{I,J}\int_{M_\nu\setminus K'}|D_t^I\eta''u_{I,J}^\nu|^2e^{-T\varphi}.
	\end{align}
	
	Since $ K'\Subset M $, $ \varphi_{\alpha\bar{\beta}} $ is bounded on $ K' $ for any $ \alpha $ and $ \beta $. Hence, for $ \zeta\in C_c^\infty\left(M_\nu\right) $ with $ \zeta|_{K'}\equiv1 $, there is a constant $ C\geq0 $ such that
	\begin{align}\label{lower estimate for B on M-K'}
		T&\int_{K'}B_\varepsilon\left(\varphi;u^\nu,u^\nu\right)e^{-T\varphi}\nonumber\\
		&\geq-CT\|\zeta D_t^I\eta''u^\nu\|_{T\varphi}^2\nonumber\\
		&\geq-2C^2T^2\|\zeta D_t^I\eta''u^\nu\|_{-1}^2-\frac{1}{2}\|\zeta D_t^I\eta''u^\nu\|_{1}^2\nonumber\\
		&\geq-C_T\|u^\nu\|_{B,s-1}^2-\frac{1}{2}G_Q^{T\varphi}\left(\zeta D_t^I\eta''u^\nu,\zeta D_t^I\eta''u^\nu\right)+O\left(\|u^\nu\|_{B,s}^2\right)\nonumber\\
		&\geq-C_T\|u^\nu\|_{B,s-1}^2-\frac{1}{2}\left(\|\bar{\partial}_QD_t^I\eta''u^\nu\|_{T\varphi}^2+\|\vartheta_Q^{T\varphi}D_t^I\eta''u^\nu\|_{T\varphi}^2\right)+O\left(\|u^\nu\|_{B,s}^2\right),
	\end{align}
	where the third line holds by Cauchy-Schwarz inequality, and the fourth line is valid in view of G$\mathring{{\rm a}}$rding's inequality claimed in the proof of Proposition \ref{mainp 1}.
	
	Then substituting $ (\ref{lower estimate for B on K'}) $ and $(\ref{lower estimate for B on M-K'}) $ into R.H.S. of $ (\ref{lower estimate for Bochner}) $ gives the conclusion for $ u^\nu $, thus for any $ u\in\mathcal{D}_B^{r',r} $. For any interior chart $ U' $, the proof of the desired estimate is similar to $ (\ref{lower estimate for B on M-K'}) $, and the estimate holds for $A=1$ in that case. Hence, we can choose $A$ sufficiently large such that $(\ref{priori estimate in weigthed setup})$ hold on all $U''$ since $\partial M$ is compact. The estimate $ (\ref{priori estimate in weigthed setup for s=0}) $ follows from $ (\ref{priori estimate in weigthed setup}) $ for $ s=0 $ by a partion of unity.
\end{proof}

In what follows, we always assume that the hypothesis of Proposition \ref{weaker condition of priori estimate} holds without special mention, then with $ (\ref{priori estimate in weigthed setup}) $ and $ (\ref{priori estimate in weigthed setup for s=0}) $ in hand, we can retrace the steps in section 3.3. We are going to prove the Neumann operator $ N_B^{T\varphi} $ for $ \Box_{\bar{F}_B}^{T\varphi} $ is well-defined where $ \varphi $ is defined by $ (\ref{A_r condition for varphi}) $. The crucial point is proving that $ {\rm Range}\left(\Box_{\bar{F}_B}^{T\varphi}\right) $ is closed.

\begin{prop}
	For $ 0\leq r',r\leq q $, we know that $ {\rm Range}\left(\Box_{\bar{F}_B}^{T\varphi}\right) $ is closed in $ H_{B,0}^{r',r} $ and the dimension of the kernel $ \mathcal{H}_{B,T\varphi}^{r',r} $ of $ \Box_{\bar{F}_B}^{T\varphi} $ is finite for sufficiently large $ T $.
\end{prop}
\begin{proof}
	By the estimate $ (\ref{priori estimate in weigthed setup for s=0}) $, for $ u\in\mathcal{H}_{B,T\varphi}^{r',r} $ if $ T $ is sufficiently large we have
	\begin{align*}
		\|u\|_{B,T\varphi}^2\leq C_T\|u\|_{B,-1}^2,
	\end{align*}
	since $ i:H_{B,0}^{r',r}\longrightarrow H_{B,-1}^{r',r} $ is compact, it follows that $ \mathcal{H}_{B,T\varphi}^{r',r} $ is finite dimensional.
	
	In order to prove $ {\rm Range}\left(\Box_{\bar{F}_B}^{T\varphi}\right) $ is closed, it suffices to prove that
	\begin{align}\label{range is closed}
		\|u\|_{B,T\varphi}^2\lesssim\|\bar{\partial}_Bu\|_{B,T\varphi}^2+\|\vartheta_B^{T\varphi}u\|_{B,T\varphi}^2,
	\end{align}
	whenever $ u\in\tilde{\mathcal{D}}_B^{r',r} $ with $ u\perp\mathcal{H}_{B,T\varphi}^{r',r} $. Thus for $ u\in{\rm Dom}\left(\Box_{\bar{F}_B}^{T\varphi}\right)\cap\left(\mathcal{H}_{B,T\varphi}^{r',r}\right)^\perp $, the estimate $ (\ref{range is closed}) $ gives $ \|u\|_{B,T\varphi}\lesssim\|\Box_{\bar{F}_B}^{T\varphi}u\|_{B,T\varphi} $.
	
	Now we prove the estimate $ (\ref{range is closed}) $ by a contradiction. Choosing a smooth sequence $ \{u_\nu\}_{\nu=1}^\infty $ such that $ u_\nu\perp\mathcal{H}_{B,T\varphi}^{r',r} $, $ \|u_\nu\|_{B,T\varphi}=1 $ and
	\begin{align}\label{contradiction estimate}
		\|u_\nu\|_{B,T\varphi}^2\geq\nu\left(\|\bar{\partial}_Bu_\nu\|_{B,T\varphi}^2+\|\vartheta_B^{T\varphi}u_\nu\|_{B,T\varphi}^2\right).
	\end{align}
	Combining $ (\ref{priori estimate in weigthed setup for s=0}) $ and $ (\ref{contradiction estimate}) $, we obtain that $ \|u_\nu\|_{B,T\varphi}^2\leq C_T\|u_\nu\|_{B,-1}^2 $ for large $ \nu $. Thus $ u_\nu\stackrel{\|\cdot\|_B}{\longrightarrow}u $ as $ \nu\rightarrow\infty $ with unit norm and $ u\perp\mathcal{H}_{B,T\varphi}^{r',r} $. However, $ u\in\mathcal{H}_{B,T\varphi}^{r',r} $ by $ (\ref{contradiction estimate}) $ which is a contradiction.
\end{proof}

Then we can define $ N_B^{T\varphi} $ and $ P_B^{T\varphi} $ as unweighted case satisfying the properties in the analogue of Propsotion \ref{proposition of N_B} by repeating the argument. In particular, we can solve the equation $ \bar{\partial}_Bu=\omega $ whenever $ \bar{\partial}_B\omega=P_B\omega=0 $ with canonical solution $ u_T:=\vartheta_BN_B^{T\varphi}\omega $ such that $ u_T\perp {\rm Ker}\left(\bar{\partial}_B\right) $ in the inner product $ \left(\cdot,\cdot\right)_{B,T\varphi} $. We can also establish the regularity of the canonical solution as the weak version of Corollary \ref{existence theorem 1} as follows. 

\begin{thm}\label{smooth existence theorem}
	With the same hypothesis as in Proposition \ref{weaker condition of priori estimate}. Let $ 1\leq r\leq q $, if $ \omega\in\Omega_{B}^{\cdot,r}\left(\overline{M}\right) $ satisfies $ \bar{\partial}_B\omega=P_B\omega=0 $. Then for each $ s\in\mathbb{N} $, there exists $ T_s>0 $ such that $ u_T:=\vartheta_B^{T\varphi}N_B^{T\varphi}\omega\in\Omega_{B,s}^{\cdot,r-1}\left(\overline{M}\right) $ is the solution of $ \bar{\partial}_Bu=\omega $ whenever $ T>T_s $.
\end{thm}

For the purpose of proving Theorem \ref{smooth existence theorem}, we need to study the continuity of some weigthed operators in $ H_{B,s}^{r',r} $ as the following propositions.

\begin{prop}\label{continuity of N_B}
	With the same hypothesis as in Proposition \ref{weaker condition of priori estimate}. Let $ 0\leq r',r\leq q $, for each $ s\in\mathbb{N} $, there exists $ T_s>0 $ such that $ N_B^{T\varphi} $ is bounded in $ H_{B,s}^{r',r} $ if $ T>T_s $.
\end{prop}

We need the following lemma to prove Proposition \ref{continuity of N_B}.

\begin{lemma}\label{tangential integration by parts in weighted case}
	For $ 0\leq r',r\leq q $, let $ \omega\in H_{B,0}^{r',r} $, and $ u_{T}\in\Omega_B^{r',r}\left(\overline{M}\right) $ such that $ \Box_{\bar{F}_B}^{T\varphi}u_T=\omega $, then for all tangential derivatives $ D_t^{I} $ of order $ s $ in special boundary chart $ U'' $ and $ \eta''\in C_c^\infty\left(U''\right) $ we have
	\begin{align*}
		\bar{G}'_{Q,T\varphi}\left(D_t^I\eta''u_T,D_t^I\eta''u_T\right)\lesssim C_T\|\omega\|_{B,s}^2+\|u_T\|_{B,s}^2+C_T\|u_T\|_{B,s-1}^2,
	\end{align*}
	where
	$$ \bar{G}'_{Q,T\varphi}\left(D_t^I\eta''u_T,D_t^I\eta''u_T\right):=\|\bar{\partial}_QD_t^I\eta''u_T\|_{T\varphi}^2+\|\vartheta_Q^{T\varphi}D_t^I\eta''u_T\|_{T\varphi}^2. $$
	In addition, if we replace $ D_t^I $ and $ \eta'' $ by $ D^I $ and $ \eta' $ where $ \eta'\in C_c^\infty\left(U'\right) $ and $ U' $ is any interior chart, then the above estimate still holds for any $ u_T\in\Omega_B^{r',r}\left(\overline{M}\right)\cap{\rm Dom}\left(\Box_{\bar{F}_B}^{T\varphi}\right) $.
\end{lemma}

\begin{proof}
	We prove the desired estimate by induction on $ s $. For $ s=0 $,
	\begin{align*}
		\bar{G}'_{Q,T\varphi}\left(\eta''u_T,\eta''u_T\right)\lesssim \bar{G}_{B,T\varphi}\left(u_T,u_T\right)=\left(\left(\Box_{\bar{F}_B}^{T\varphi}+{\rm Id}\right)u_T,u_T\right)_{B,T\varphi}\lesssim\|\omega\|_{B}^2+\|u_T\|_{B}^2.
	\end{align*}
	
	Suppose the estimate holds for $ s-1 $, then $ u_T\in\Omega_B^{r',r}\left(\overline{M}\right)\cap{\rm Dom}\left(\Box_{\bar{F}_B}^{T\varphi}\right) $
	\begin{align}\label{e 1}
		&\left(\bar{\partial}_QD_t^I\eta''u_T,\bar{\partial}_QD_t^I\eta'' u_T\right)_{T\varphi}\nonumber\\
		=&\left(D_t^I\eta''\bar{\partial}_Qu_T,\bar{\partial}_QD_t^I\eta'' u_T\right)_{T\varphi}+O\left(\|u_T\|_{B,s}\sqrt{\bar{G}'_{Q,T\varphi}\left(D_t^I\eta''u_T,D_t^I\eta''u_T\right)}\right)\nonumber\\
		=&\left(\bar{\partial}_Qu_T,\eta''\left(D_t^I\right)_{T\varphi}^*\bar{\partial}_QD_t^I\eta'' u_T\right)_{T\varphi}+O\left(\|u_T\|_{B,s}\sqrt{\bar{G}'_{Q,T\varphi}\left(D_t^I\eta''u_T,D_t^I\eta''u_T\right)}\right)\nonumber\\
		=&\left(\bar{\partial}_Qu_T,\bar{\partial}_Q\eta''\left(D_t^I\right)_{T\varphi}^*D_t^I\eta'' u_T\right)_{T\varphi}+\left(\bar{\partial}_Qu_T,\left[\eta''\left(D_t^I\right)_{T\varphi}^*,\bar{\partial}_Q\right]D_t^I\eta'' u_T\right)_{T\varphi}\nonumber\\
		&+O\left(\|u_T\|_{B,s}\sqrt{\bar{G}'_{Q,T\varphi}\left(D_t^I\eta''u_T,D_t^I\eta''u_T\right)}\right),
	\end{align}
	where $ \left(D_t^I\right)_{T\varphi}^* $ is the formal adjoint of $ D_t^I $ with respect to $ \left(\cdot,\cdot\right)_{T\varphi} $. For the second item on R.H.S. of $ (\ref{e 1}) $, because $ s^{th} $-order
	term in $ \left[\eta''\left(D_t^I\right)_{T\varphi}^*,\bar{\partial}_Q\right] $ is independent of $ T $,
	\begin{align}\label{e 2}
		&\left(\bar{\partial}_Qu_T,\left[\eta''\left(D_t^I\right)_{T\varphi}^*,\bar{\partial}_Q\right]D_t^I\eta'' u_T\right)_{T\varphi}\nonumber\\
		=&\left(\bar{\partial}_Qu_T,D_t^I\eta''\left[\eta''\left(D_t^I\right)_{T\varphi}^*,\bar{\partial}_Q\right]u_T\right)_{T\varphi}+\left(\bar{\partial}_Qu_T,\left[\left[\eta''\left(D_t^I\right)_{T\varphi}^*,\bar{\partial}_Q\right],D_t^I\eta''\right]u_T\right)_{T\varphi}\nonumber\\
		=&\left(\eta''\left(D_t^I\right)_{T\varphi}^*\bar{\partial}_Qu_T,\left[\eta''\left(D_t^I\right)_{T\varphi}^*,\bar{\partial}_Q\right]u_T\right)_{T\varphi}+O\left(\|u_T\|_{B,s}^2\right)+O_T\left(\|u_T\|_{B,s}\|u_T\|_{B,s-1}\right),
	\end{align}
	where $ O_T $ means that the constant depends on $ T $. For the first term on R.H.S. of $ (\ref{e 2}) $,
	\begin{align}\label{e 3}
		&\left(\eta''\left(D_t^I\right)_{T\varphi}^*\bar{\partial}_Qu_T,\left[\eta''\left(D_t^I\right)_{T\varphi}^*,\bar{\partial}_Q\right]u_T\right)_{T\varphi}\nonumber\\
		=&\left(\bar{\partial}_Q\eta''\left(\left(-1\right)^sD_t^I+C_TS_t\right)u_T,\left[\eta''\left(D_t^I\right)_{T\varphi}^*,\bar{\partial}_Q\right]u_T\right)_{T\varphi}+\|\left[\eta''\left(D_t^I\right)_{T\varphi}^*,\bar{\partial}_Q\right]u_T\|_{T\varphi}^2\nonumber\\
		=&O\left(\left(\sqrt{\bar{G}'_{Q,T\varphi}\left(D_t^I\eta''u_T,D_t^I\eta''u_T\right)}+C_T\sqrt{\bar{G}'_{Q,T\varphi}\left(S_t\eta''u_T,S_t\eta''u_T\right)}\right)\left(\|u_T\|_{B,s}+C_T\|u_T\|_{B,s-1}\right)\right)\nonumber\\
		&+O\left(\|u_T\|_{B,s}^2+C_T\|u_T\|_{B,s-1}^2\right).
	\end{align}
	The second equality holds due to $ \left(D_t^I\right)_{T\varphi}^*=\left(-1\right)^sD_t^I+C_TS_t $, where $ S_t $ denotes the tangential operator of order $ s-1 $. By inductive hypothesis we have
	\begin{align}\label{e 4}
		\bar{G}'_{Q,T\varphi}\left(S_t\eta''u_T,S_t\eta''u_T\right)\lesssim C_T\|\omega\|_{B,s-1}^2+\|u_T\|_{B,s-1}^2+C_T\|u_T\|_{B,s-2}^2.
	\end{align}
	Hence, for any $ 0<\varepsilon<<1 $, combining $ (\ref{e 1})\sim(\ref{e 4}) $ gives
	\begin{align}\label{e 5}
		\|\bar{\partial}_QD_t^I\eta''u_T\|_{T\varphi}^2\lesssim\left(\bar{\partial}_Qu_T,\bar{\partial}_Q\eta''\left(D_t^I\right)_{T\varphi}^*D_t^I\eta'' u_T\right)_{T\varphi}&+\varepsilon \bar{G}'_{Q,T\varphi}\left(D_t^I\eta''u_T,D_t^I\eta''u_T\right)\nonumber\\
		&+C_T\|\omega\|_{B,s}^2+\|u_T\|_{B,s}^2+C_T\|u_T\|_{B,s-1}^2.
	\end{align}
	
	Correspondingly, the norm $ \|\vartheta_QD_t^I\eta''u_T\|_{T\varphi}^2 $ can be controlled as $ (\ref{e 5}) $. Thus,
	\begin{align*}
		\bar{G}'_{Q,T\varphi}\left(D_t^I\eta''u_T,D_t^I\eta''u_T\right)\lesssim&\left(\Box_{\bar{F}_B}^{T\varphi}u_T,\eta''\left(D_t^I\right)_{T\varphi}^*D_t^I\eta'' u_T\right)_{T\varphi}+\varepsilon \bar{G}'_{Q,T\varphi}\left(D_t^I\eta''u_T,D_t^I\eta''u_T\right)\nonumber\\
		&+C_T\|\omega\|_{B,s}^2+\|u_T\|_{B,s}^2+C_T\|u_T\|_{B,s-1}^2\nonumber\\
		\lesssim& C_T\|\omega\|_{B,s}^2+\|u_T\|_{B,s}^2+C_T\|u_T\|_{B,s-1}^2+\varepsilon \bar{G}'_{Q,T\varphi}\left(D_t^I\eta''u_T,D_t^I\eta''u_T\right).
	\end{align*}
	
	The interior estimate can be given by the parallel calculation as above.
\end{proof}

Now, we are prepared to give the proof of Proposition \ref{continuity of N_B}.

\renewcommand{\proofname}{\bf $ Proof\ of\ Proposition\ \ref{continuity of N_B} $}
\begin{proof}
	We first prove the continuity of $ N_B^{T\varphi} $ for smooth forms by induction on $ s $. The case $ s=0 $ is trivial, we assume that the conclusion is valid for $ s-1 $. By Proposition \ref{weaker condition of priori estimate} and Lemma \ref{tangential integration by parts in weighted case}, we know that the tangential and interior derivatives can be controlled. Indeed, for $ u_T\in\Omega_B^{r',r}\left(\overline{M}\right)\cap{\rm Dom}\left(\Box_{\bar{F}_B}^{T\varphi}\right) $
	\begin{align}\label{n 1}
		T\|D^I\eta'u_T\|_{B,T\varphi}^2+T\|D_t^I\eta''u_T\|_{B,T\varphi}^2\lesssim C_T\|\omega\|_{B,s}^2+\|u_T\|_{B,s}^2+C_T\|u_T\|_{B,s-1}^2,
	\end{align}
	where the notations above are same as in Lemma \ref{tangential integration by parts in weighted case}.
	
	To estimate the normal direction, note that $ \bar{\partial}_B\oplus\vartheta_B^{T\varphi} $ is elliptic, so
	\begin{align}\label{ellipticity 2}
		\left(\bar{\partial}_B\oplus\vartheta_B^{T\varphi}\right)u_T=A_0D_\rho u_T+\sum_{\alpha=1}^{q-1}A_\alpha D_t^\alpha u_T+B_0u_T,
	\end{align}
	where $ A_0 $ is invertible. Applying $ D_t^{K}\eta'' $ $ \left(|K|=s-1\right) $ to the equality $ (\ref{ellipticity 2}) $,
	\begin{align}\label{n 2}
		\|D_t^KD_\rho\eta''u_T\|_{T\varphi}^2&\lesssim\bar{G}'_{Q,T\varphi}\left(D_t^K\eta''u_T,D_t^K\eta''u_T\right)+\sum_{|I|=s}\|D_t^I\eta''u_T\|_{T\varphi}^2+C_T\|u_T\|_{B,s-1}^2\nonumber\\
		&\lesssim C_T\|\omega\|_{B,s-1}^2+\sum_{|I|=s}\|D_t^I\eta''u_T\|_{T\varphi}^2+C_T\|u_T\|_{B,s-1}^2.
	\end{align}
	The last line follows from Lemma \ref{tangential integration by parts in weighted case}. 
	
	For the higher order normal derivatives we proceed by induction on its order $ \iota $. The case $ \iota=1 $ is obtained as $ (\ref{n 2}) $. Since the $ 2^{nd} $-order of $ \Box_{\bar{F}_B}^{T\varphi} $ is same as $ \Box_B $, we have
	\begin{align}\label{ellipticity 3}
		D_\rho^2\eta''u_T=\eta_l''A'\Box_{\bar{F}_B}^{T\varphi}u_T-\sum_{\alpha=p+1}^{n-1}B_\alpha D_t^\alpha D_\rho\eta''u_T-\sum_{\alpha,\beta=p+1}^{n-1}C_{\alpha\beta}D_t^\alpha D_t^\beta\eta''u_T+\cdots,
	\end{align}
	where $ A' $ is invertible and dots denote lower order terms.
	Applying $ D_t^JD_\rho^{\iota-2} $ $ \left(|J|+\iota=s\right) $ to $ (\ref{ellipticity 3}) $ and by the inductive hypothesis for $ k\leq\iota-1 $ we obtain
	\begin{align}\label{n 3}
		\|D_t^JD_\rho^{\iota}\eta''u_T\|_{T\varphi}^2\lesssim&\|\omega\|_{B,s-2}^2+\sum_{k=\iota-2}^{\iota-1}\sum_{|J|+k=s}\|D_t^JD_\rho^{k}\eta''u_T\|_{T\varphi}^2+\sum_{|I|=s}\|D_t^I\eta''u_T\|_{T\varphi}^2+C_T\|u_T\|_{B,s-1}^2\nonumber\\
		\lesssim&C_T\|\omega\|_{B,s-1}^2+\sum_{|I|=s}\|D_t^I\eta''u_T\|_{T\varphi}^2+C_T\|u_T\|_{B,s-1}^2.
	\end{align}
	Thus, combining $ (\ref{n 1}) $ and $ (\ref{n 3}) $ gives
	\begin{align*}
		T\|u_T\|_{B,s}^2\leq C_T\left(\|\omega\|_{B,s}^2+\|u_T\|_{B,s-1}^2\right)+O\left(\|u_T\|_{B,s}^2\right).
	\end{align*}
	
	Let $ T $ is sufficiently large, then by inductive hypothesis for $ s-1 $,
	\begin{align*}
		\|u_T\|_{B,s}^2\leq C_T\left(\|\omega\|_{B,s}^2+\|u_T\|_{B,s-1}^2\right)\leq C_T\|\omega\|_{B,s}^2,
	\end{align*}
	which proves the conclusion on $ \Omega_B^{r',r}\left(\overline{M}\right)\cap{\rm Dom}\left(\Box_{\bar{F}_B}^{T\varphi}\right) $. Then the general case follows from the the technique of elliptic regularization as in the proof of Theorem \ref{main theorem for dolbeault} line by line.
\end{proof}

We can also show that the related operators of $ N_B^{T\varphi} $ are also continious in $ H_{B,s}^{r',r} $ for $ 0\leq r',r\leq q $ as follows.
\begin{prop}\label{continuity of related operators}
	With the same hypothesis as in Proposition \ref{weaker condition of priori estimate}. Let $ 0\leq r',r\leq q $, for each $ s\in\mathbb{N} $, there exists $ T_s>0 $ such that $ \bar{\partial}_BN_B^{T\varphi} $, $ \vartheta_B^{T\varphi}N_B^{T\varphi} $, $ \bar{\partial}_B\vartheta_B^{T\varphi}N_B^{T\varphi} $, $ \vartheta_B^{T\varphi}\bar{\partial}_BN_B^{T\varphi} $ and $ P_B^{T\varphi} $ are bounded in $ H_{B,s}^{r',r} $ if $ T>T_s $.
\end{prop}

\renewcommand{\proofname}{\bf $ Proof $}
\begin{proof}
	By the argument in the proof of Theorem \ref{continuity of N_B}, it suffices to prove that the norms of all tangential derivatives can be controlled. Thus, we work on a special boundary chart $ U'' $ with $ \eta''\in C_c^\infty\left(U''\right) $ and tangential derivatives $ D_t^I $ of order $ s $. For any nonzero $ \omega\in{\rm Range}\left(\Box_{\bar{F}_B}^{T\varphi}\right)\cap\Omega_B^{r',r}\left(\overline{M}\right) $,
	\begin{align}\label{f 1}
		&\left(D_t^I\eta''\vartheta_B^{T\varphi}N_B^{T\varphi}\omega,D_t^I\eta''\vartheta_B^{T\varphi}N_B^{T\varphi}\omega\right)_{T\varphi}\nonumber\\
		=&\left(\vartheta_B^{T\varphi}D_t^I\eta''N_B^{T\varphi}\omega,D_t^I\eta''\vartheta_B^{T\varphi}N_B^{T\varphi}\omega\right)_{T\varphi}+O_T\left(\|N_B^{T\varphi}\omega\|_{B,s}\|\vartheta_B^{T\varphi}N_B^{T\varphi}\omega\|_{B,s}^2\right)\nonumber\\
		=&\left(D_t^I\eta''N_B^{T\varphi}\omega,D_t^I\eta''\bar{\partial}_B\vartheta_B^{T\varphi}N_B^{T\varphi}\omega\right)_{T\varphi}+O_T\left(\|N_B^{T\varphi}\omega\|_{B,s}\|\vartheta_B^{T\varphi}N_B^{T\varphi}\omega\|_{B,s}^2\right).
	\end{align}
	Correspondingly, $ \|D_t^I\eta''\bar{\partial}_BN_B^{T\varphi}\omega\|_{T\varphi} $ can be estimated as $ (\ref{f 1}) $. Thus for $ 0<\varepsilon<<1 $
	\begin{align}\label{f 2}
		&\|D_t^I\eta''\vartheta_B^{T\varphi}N_B^{T\varphi}\omega\|_{T\varphi}^2+\|D_t^I\eta''\bar{\partial}_BN_B^{T\varphi}\omega\|_{T\varphi}^2\nonumber\\
		=&\left(D_t^I\eta''N_B^{T\varphi}\omega,D_t^I\eta''\Box_{\bar{F}_B}^{T\varphi}N_B^{T\varphi}\omega\right)_{T\varphi}+O_T\left(\|N_B^{T\varphi}\omega\|_{B,s}\left(\|\vartheta_B^{T\varphi}N_B^{T\varphi}\omega\|_{B,s}^2+\|\bar{\partial}_BN_B^{T\varphi}\omega\|_{B,s}^2\right)\right)\nonumber\\
		\lesssim&C_T\|N_B^{T\varphi}\omega\|_{B,s}^2+\|\omega\|_{B,s}^2+\varepsilon\left(\|\vartheta_B^{T\varphi}N_B^{T\varphi}\omega\|_{B,s}^2+\|\bar{\partial}_BN_B^{T\varphi}\omega\|_{B,s}^2\right).
	\end{align}
	By $ (\ref{f 2}) $ and the continuity of $ N_B^{T\varphi} $ for $ T>T_s $, we know that
	\begin{align}\label{f 3}
		\|\vartheta_B^{T\varphi}N_B^{T\varphi}\omega\|_{B,s}^2+\|\bar{\partial}_BN_B^{T\varphi}\omega\|_{B,s}^2\leq C_T\|\omega\|_{B,s}^2.
	\end{align}
	
	We only provide the details of the proof that the operator $ \vartheta_B^{T\varphi}\bar{\partial}_BN_B^{T\varphi} $ is bounded in $ H_{B,s}^{r',r} $ since the assertion for $ \bar{\partial}_B\vartheta_B^{T\varphi}N_B^{T\varphi} $ can be obtained in a similar manner. 
	\begin{align}\label{f 4}
		&\left(D_t^I\eta''\vartheta_B^{T\varphi}\bar{\partial}_BN_B^{T\varphi}\omega,D_t^I\eta''\vartheta_B^{T\varphi}\bar{\partial}_BN_B^{T\varphi}\omega\right)_{T\varphi}\nonumber\\
		=&\left(\vartheta_B^{T\varphi}D_t^I\eta''\bar{\partial}_BN_B^{T\varphi}\omega,D_t^I\eta''\vartheta_B^{T\varphi}\bar{\partial}_BN_B^{T\varphi}\omega\right)_{T\varphi}+O_T\left(\|\bar{\partial}_BN_B^{T\varphi}\omega\|_{B,s}\|\vartheta_B^{T\varphi}\bar{\partial}_BN_B^{T\varphi}\omega\|_{B,s}\right)\nonumber\\
		=&\left(D_t^I\eta''\bar{\partial}_BN_B^{T\varphi}\omega,D_t^I\eta''\bar{\partial}_B\vartheta_B^{T\varphi}\bar{\partial}_BN_B^{T\varphi}\omega\right)_{T\varphi}+O_T\left(\|\bar{\partial}_BN_B^{T\varphi}\omega\|_{B,s}\|\vartheta_B^{T\varphi}\bar{\partial}_BN_B^{T\varphi}\omega\|_{B,s}\right)\nonumber\\
		=&\left(D_t^I\eta''\bar{\partial}_BN_B^{T\varphi}\omega,D_t^I\eta''\bar{\partial}_B\omega\right)_{T\varphi}+O_T\left(\|\bar{\partial}_BN_B^{T\varphi}\omega\|_{B,s}\|\vartheta_B^{T\varphi}\bar{\partial}_BN_B^{T\varphi}\omega\|_{B,s}\right)\nonumber\\
		=&\left(D_t^I\eta''\bar{\partial}_BN_B^{T\varphi}\omega,\bar{\partial}_BD_t^I\eta''\omega\right)_{T\varphi}+O_T\left(\|\bar{\partial}_BN_B^{T\varphi}\omega\|_{B,s}\left(\|\vartheta_B^{T\varphi}\bar{\partial}_BN_B^{T\varphi}\omega\|_{B,s}+\|\omega\|_{B,s}\right)\right)\nonumber\\
		=&\left(D_t^I\eta''\vartheta_B^{T\varphi}\bar{\partial}_BN_B^{T\varphi}\omega,D_t^I\eta''\omega\right)_{T\varphi}+O_T\left(\|\bar{\partial}_BN_B^{T\varphi}\omega\|_{B,s}\left(\|\vartheta_B^{T\varphi}\bar{\partial}_BN_B^{T\varphi}\omega\|_{B,s}+\|\omega\|_{B,s}\right)\right)\nonumber\\
		=&O\left(\|\vartheta_B^{T\varphi}\bar{\partial}_BN_B^{T\varphi}\omega\|_{B,s}\|\omega\|_{B,s}\right)+O_T\left(\|\bar{\partial}_BN_B^{T\varphi}\omega\|_{B,s}\left(\|\vartheta_B^{T\varphi}\bar{\partial}_BN_B^{T\varphi}\omega\|_{B,s}+\|\omega\|_{B,s}\right)\right),
	\end{align}
	where the third equality follows from $ \bar{\partial}_B^2=0 $. Thanks to $ (\ref{f 3}) $ and $ (\ref{f 4}) $ we have
	\begin{align*}
		\|\vartheta_B^{T\varphi}\bar{\partial}_BN_B^{T\varphi}\omega\|_{B,s}^2\leq C_T\|\omega\|_{B,s}^2.
	\end{align*}
	
	The continuity of the operator $ P_B^{T\varphi} $ in $ H_{B,s}^{r',r} $ is trivial since $ P_B^{T\varphi}={\rm Id}-\vartheta_B^{T\varphi}\bar{\partial}_BN_B^{T\varphi}-\bar{\partial}_B\vartheta_B^{T\varphi}N_B^{T\varphi} $ in $ H_{B,0}^{r',r} $.
\end{proof}

\begin{remark}\label{interpolation}
	All of the above operators are all bounded on $ H_{B,m}^{r',r} $ for $ m\leq s $ by interpolation (see \cite{BL76}).
\end{remark}

\renewcommand{\proofname}{\bf $ Proof\ of\ Theorem\ \ref{smooth existence theorem} $}
\begin{proof}
	The theorem is an immediate consequence by Proposition \ref{continuity of related operators} and Proposition \ref{Rellich and Sobolev lemma} since $ u_T=\vartheta_B^{T\varphi}N_B^{T\varphi}\omega $.
\end{proof}

If the exhaustion basic function $ \varphi' $ in Proposition \ref{weaker condition of priori estimate} also satisfies $ Z_{r-1} $-condition for $ 1\leq r\leq q $, then one can obtain the analogue of Corollary \ref{existence theorem 1 for dolbeault} without estimate. To be precise, we have

\begin{thm}\label{grt.}
	For $ 1\leq r\leq q $, let $ M $ be an oriented Riemannian manifold with compact basic boundary $ \partial M $, assume that there is an exhaustion basic function $ \varphi' $ satisfying $ Z_r $ and $Z_{r-1}$-condition outside a compact $ K\Subset M $. If the equation $ \bar{\partial}_Bu=\omega $ is solvable for $ \omega\in H_{B,0}^{\cdot,r} $, then the solution $ u\in\Omega_B^{\cdot,r-1}\left(\overline{M}\right) $ whenever $ \omega\in\Omega_B^{\cdot,r}\left(\overline{M}\right) $.
\end{thm}

\renewcommand{\proofname}{\bf $ Proof $}
\begin{proof}
	Let $ \omega\in\Omega_B^{\cdot,r}\left(\overline{M}\right) $, then for any $ s\in\mathbb{N} $, there exists a solution $ u_s\in H_{B,s}^{\cdot,r-1} $ of the equation $ \bar{\partial}_Bu=\omega $ due to Theorem \ref{smooth existence theorem}. We will construct a new solution $ u\in\Omega_B^{\cdot,r-1}\left(\overline{M}\right) $ by modifying the sequence $ \{u_s\}_{s=0}^\infty $.
	
	We first claim that $ {\rm Ker}\left(\bar{\partial}_B\right)\cap H_{B,s_1}^{\cdot,r-1} $ is dense in $ {\rm Ker}\left(\bar{\partial}_B\right)\cap H_{B,s_2}^{\cdot,r-1} $ for $ s_1\geq s_2\geq0 $. In fact, since $ H_{B,s_1}^{\cdot,r-1} $ is dense in $ H_{B,s_2}^{\cdot,r-1} $, for any $ u_{s_2}\in H_{B,s_2}^{\cdot,r-1} $, there is a sequence $ \{u_\nu\}_{\nu=1}^\infty\subseteq H_{B,s_1}^{\cdot,r-1} $ such that $ u_\nu\stackrel{\|\cdot\|_{B,s_2}}{\longrightarrow}u_{s_2} $ as $ \nu\rightarrow\infty $. Then we choose $ T $ is large enough such that $ R_{B}^{T\varphi}:={\rm Id}-\vartheta_B^{T\varphi}\bar{\partial}_BN_B^{T\varphi} $ is continious on $ H_{B,s_2}^{\cdot,r-1} $ by Proposition \ref{continuity of related operators}. Since $ \bar{\partial}_BR_{B}^{T\varphi}=\bar{\partial}_B-\bar{\partial}_B\Box_{\bar{F}_B}^{T\varphi}N_B^{T\varphi}=0 $, we have $ {\rm Range}\left(R_{B}^{T\varphi}\right)\subseteq {\rm Ker}\left(\bar{\partial}_B\right) $. Thus $ \tilde{u}_\nu:=R_{B}^{T\varphi}u_\nu\in {\rm Ker}\left(\bar{\partial}_B\right)\cap H_{B,s_2}^{\cdot,r-1} $ and $ \tilde{u}_\nu\stackrel{\|\cdot\|_{B,s_1}}{\longrightarrow}u_{s_2} $ as $ \nu\rightarrow\infty $ from Remark \ref{interpolation}.
	
	Now we can generate a new solution sequence $ \{\hat{u}_\nu\}_{\nu=1}^\infty $ inductively such that $ \hat{u}_\nu\in H_{B,\nu}^{\cdot,r-1} $. Set $ \hat{u}_1=u_1 $, assume that $ \hat{u}_{s} $ is already know, then $ u_{s+1}-\hat{u}_{s}\in {\rm Ker}\left(\bar{\partial}_B\right)\cap H_{B,s}^{\cdot,r-1} $, thus there exists $ v_{s+1}\in{\rm Ker}\left(\bar{\partial}_B\right)\cap H_{B,s+1}^{\cdot,r-1} $ such that
	\begin{align}\label{error}
		\|u_{s+1}-\hat{u}_{s}-v_{s+1}\|_{B,s-1}^2\leq\frac{1}{2^{s}}.
	\end{align}
	We define $ \hat{u}_{s+1}:=u_{s+1}-v_{s+1} $, it's obvious that $ \bar{\partial}_B\hat{u}_{s+1}=\omega $ and $ \hat{u}_{s+1}\in H_{B,s+1}^{\cdot,r-1} $. Setting
	\begin{align*}
		u:=\hat{u}_1+\sum_{\nu=1}^{\infty}\left(\hat{u}_{\nu+1}-\hat{u}_\nu\right).
	\end{align*}
	Hence, $ u $ is well-defined by $ (\ref{error}) $ and $ u\in H_{B,s}^{\cdot,r-1} $ for all $ s\in\mathbb{N} $ which implies that $ u\in\Omega_B^{\cdot,r-1}\left(\overline{M}\right) $.
\end{proof}

\begin{cor}\label{smooth solution 2}
	For $ 2\leq r\leq q $, $ \overline{M} $ is an oriented manifold with transversally oriented Riemannian foliation with basic boundary, and $ g_M $ is a bundle-like metric with basic mean curvature form. If there exists a strictly transversal $ r-1 $-convex exhaustion basic function on $M$, then for any $ \omega\in\Omega_B^{\cdot,r}\left(\overline{M}\right) $ with $ \bar{\partial}_B\omega=0 $, there is a basic form $ u\in\Omega_B^{\cdot,r-1}\left(\overline{M}\right) $ such that $ \bar{\partial}_Bu=\omega $.
\end{cor}

\begin{proof}
	Thoerem \ref{existence theorem 3 for dolbeault} yields that the equation $ \bar{\partial}_Bu=\omega $ is solvable. Then the conclusion follows from Theorem \ref{grt.}.
\end{proof}

\begin{remark}
	We require $r\geq2$ in Corollary \ref{smooth solution 2} since $0$-convex function is meaningless. The smooth solution given in Theorem \ref{grt.} is not the canonical solution with minimal $ L^2 $ norm, while the smooth solution obtained in Theorem \ref{existence theorem 2 for dolbeault} is the canonical solution satisfying the estimate $ (\ref{estimate of solution of partial_B}) $.
\end{remark}

\subsection{Extension and duality theorems}
In this section, we will introduce the induced boundary complex on $\partial M$, then we consider the extension and duality problems of the induced boundary basic forms. 

Since the boundary defining function $\rho$ is basic, the Hermitian foliation on $\overline{M}$ induce a transversely CR foliation (i.e., the transverse manifold of the foliation is a CR manifold) on the closed manifold $\partial M$, and the induced metric is also bundle-like. In the sequel, the notation $\cdot_{b}$ refers to objects on $\partial M$.

For $0\leq r',r\leq q$, the transversal star operator $*_Q$ defined by (\ref{star operator}) induces 
$$*_Q:\Omega_Q^{r',r}\left(M\right)\longrightarrow\Omega_{Q}^{q-r,q-r'}\left(M\right).$$ 
Clearly, $*_Q$ also preserves $\Omega_B^{r',r}\left(M\right)$ and $*_Q^2u=(-1)^{r'+r}u$ for all $u\in\Omega_{Q}^{r',r}\left(M\right)$. We write the mean curvature form $\kappa=\kappa^{1,0}+\kappa^{0,1}$, where $\kappa^{1,0}=\frac{1}{2}\left(\kappa+iJ\kappa\right)$ and $\kappa^{0,1}=\frac{1}{2}\left(\kappa-iJ\kappa\right)$. Then the operator ${\rm d}_{Q,\kappa}=\bar{\partial}_Q-\kappa\wedge=\partial_{Q,\kappa}+\bar{\partial}_{Q,\kappa} $, where (see \cite{JR21})
\begin{align*}
	\partial_{Q,\kappa}=\partial_Q-\kappa^{1,0}\wedge,\quad\bar{\partial}_{Q,\kappa}=\bar{\partial}_Q-\kappa^{0,1}\wedge.
\end{align*}
Furthermore,
\begin{align}\label{twist operators for dolbeault}
	\vartheta_Q=-*_Q\partial_{Q,\kappa}*_Q,\quad\vartheta_{Q,\kappa}=-*_Q\partial_{Q}*_Q,
\end{align}
where $\vartheta_{Q,\kappa}$ is the formal adjoint of $\partial_{Q,\kappa}$ with respect to $(\cdot,\cdot)_M$. Since the first order terms of $\bar{\partial}_{Q,\kappa}$ is same as $\bar{\partial}_{Q}$, then 
$$\mathcal{D}_{Q,\kappa}^{r',r}:={\rm Dom}\left(\bar{\partial}_{Q,\kappa}^*\right)\cap\Omega_Q^{r',r}=\mathcal{D}_{Q}^{r',r},$$
and the difference of the Bochner formula for $\bar{\partial}_{Q,\kappa},\vartheta_{Q,\kappa}$ and the Bochner formula for $\bar{\partial}_{Q},\vartheta_{Q}$ is zero order term. Then the preceding argument goes through as in Theorem \ref{Hodge decomposition theorem for dolbeault} without change to show

\begin{prop}\label{hodge and existence theorem 1 for partial_kappa}
	Let $\left(M,g_M,\mathcal{F},\rho\right)$ is same as in Theorem $\ref{Hodge decomposition theorem for dolbeault}$. Assume that $\partial M$ satisfies $Z_r$-condition. For $ 0\leq r',r\leq q $, There is an orthogonal decomposition
	\begin{equation*}
		\Omega_B^{r',r}\left(\overline{M}\right)={\rm Im}\bar{\partial}_{B,\kappa}\oplus{\rm Im}\vartheta_{B,\kappa}\oplus\mathcal{H}_{\kappa}^{r',r},
	\end{equation*}
	where $ \mathcal{H}_{\kappa}^{r',r}:={\rm Ker}\left(\bar{\partial}_{B,\kappa}\right)\cap{\rm Ker}\left(\bar{\partial}_{B,\kappa}^*\right) $ is a finite-dimensional space. Moreover, for $ \omega\in\Omega_B^{r',r} $ $(1\leq r\leq q)$ satisfying $ \bar{\partial}_{B,\kappa}\omega=0 $ and $ \omega\perp\mathcal{H}_{\kappa}^{r',r} $, then there exists a unique solution $ u\in\Omega_B^{r',r-1} $ of the equation $ \bar{\partial}_{B,\kappa}u=\omega $ with $ u\perp {\rm Ker}\left(\bar{\partial}_{B,\kappa}\right) $. Thus, we have an isomorphism
	$$H_{B,\kappa}^{r',r}\left(\overline{M},\mathcal{F}\right):=\frac{\{u\in\Omega_B^{r',r}\left(\overline{M}\right)\ |\ \bar{\partial}_{B,\kappa}u=0\}}{\bar{\partial}_{B,\kappa}\Omega_B^{r',r-1}\left(\overline{M}\right)}\cong\mathcal{H}_{\kappa}^{r',r}.$$
\end{prop}

Defining $\mathcal{C}_B^{\cdot,\cdot}:={\rm Dom}\left(\vartheta_B^*\right)\cap\Omega_B^{\cdot,\cdot}\left(\overline{M}\right)$. By the reasoning of $(\ref{DBDQ})$ that
\begin{align}\label{condition of C_B for dolbeault}
	\mathcal{C}_B^{\cdot,\cdot}=\{u\in\Omega_B^{\cdot,\cdot}\left(\overline{M}\right)\ |\ \bar{\partial}_B\rho\wedge u=0\ {\rm on}\ \partial M\}=\{u\in\Omega_B^{\cdot,\cdot}\left(\overline{M}\right)\ |\ u=\bar{\partial}_B\rho\wedge\cdot+\rho\cdot\}.
\end{align}
Also, $(\ref{twist operators for dolbeault})$ and $(\ref{condition of C_B for dolbeault})$ implies the relations between $\mathcal{D}_B^{\cdot,\cdot}$ and $\mathcal{C}_B^{\cdot,\cdot}$ as follows.
\begin{prop}
	For $0\leq r',r\leq q$, we have $\mathcal{C}_B^{r',r}=*_Q\bar{\mathcal{D}}_B^{q-r',q-r}$, $\bar{\partial}_B\mathcal{C}_B^{r',r}\subset\mathcal{C}_B^{r',r+1}$, and $\vartheta_{B,\kappa}\mathcal{D}_B^{r',r+1}\subset\mathcal{D}_B^{r',r}$.
\end{prop}

Hence, the cohomology $$H^{r',r}_B\left(\overline{M},\mathcal{C}\right):=\frac{{\rm Ker}\left(\bar{\partial}_B\right)\cap\mathcal{C}_B^{r',r}}{\bar{\partial}_B\mathcal{C}_B^{r',r-1}}$$ 
is well-defined. Then one can define the induced basic form 
$$\Omega_{B,b}^{r',r}\left(\partial M\right):=\{i^*u\ |\ u\in \mathcal{D}_B^{r',r},\ i:\partial M\longrightarrow M\},$$
and we obtain
\begin{align}\label{exact commutative diagram for dolbeault}
	\begin{CD}
		0 @>>> \mathcal{C}_B^{r'r} @>>> \Omega_B^{r'r}\left(\overline{M}\right) @>R>> \Omega_{B,b}^{r'r}\left(\partial M\right) @>>>0\\
		@.                         @VV \bar{\partial}_BV  @VV \bar{\partial}_BV       @VV \bar{\partial}_{B,b}V\\
		0 @>>> \mathcal{C}_B^{r',r+1} @>>> \Omega_B^{r',r+1}\left(\overline{M}\right) @>R>> \Omega_{B,b}^{r',r+1}\left(\partial M\right) @>>>0,
	\end{CD}
\end{align}
where $\bar{\partial}_{B,b}$ is induced by $\bar{\partial}_B$ with $\bar{\partial}_{B,b}^2=0$ and $R$ is the is restriction followed by projection of
$\mathcal{D}_B^{r',r}|_{\partial M}$ onto $\Omega_{B,b}^{r',r}\left(\partial M\right)$. Moreover, corresponding argument in Theorem \ref{twisted duality} are also
valid for Hermitian foliations. Precise statement is given as following, however, require a boundary condition since Proposition \ref{hodge and existence theorem 1 for partial_kappa} holds with geometric assumption of the boundary.

\begin{thm}\label{twisted duality for dolbeault}
	For $0\leq r',r\leq q$, if $\partial M$ satisfies $Z_{q-r}$-condition, then 
	$$H_B^{r',r}\left(\overline{M},\mathcal{C}\right)\cong\left(H_{B,\kappa}^{q-r',q-r}\left(\overline{M},\mathcal{F}\right)\right)^*.$$
\end{thm}

At last, we consider the weak extension problem in the sense of J. J. Kohn and H. Rossi in \cite{KR65} as follows, let $u\in\Omega_{B,b}^{r',r}\left(\partial M\right)$, we say $u$ has a weak extension if there is a form $\tilde{u}\in\Omega_{B}^{r',r}\left(\overline{M}\right)$ such that $R\tilde{u}=u$, where $R$ is the projection defined in $(\ref{exact commutative diagram for dolbeault})$. Note that the weak extension problem is the complex analogue of tangential extension problem in Riemannian foliations. Theorem \ref{twisted duality for dolbeault} plays the same role as Theorem \ref{twisted duality}, we therefore have the following extension theorem for $\bar{\partial}_{B,b}$-closed forms on $\Omega_{B,b}^{r',r}\left(\partial M\right)$ as an application.

\begin{thm}
	For $0\leq r',r\leq q$, asumme that the boundary $\partial M$ satisfies $Z_{q-r-1}$-condition, then $\bar{\partial}_{B,b}$-closed form $u\in\Omega_{B,b}^{r',r}\left(\partial M\right)$ have a weak $\bar{\partial}_{B}$-closed extension if and only if $\int_{\partial M}u\wedge v\wedge\chi_{\mathcal{F}}=0$ for any $v\in\mathcal{H}_\kappa^{q-r',q-r-1}$.
\end{thm}

\renewcommand\refname{References}

\bigskip
{\small Addresses:}

{\small Qingchun Ji}

{\small School of Mathematics, Fudan University} 

{\small Shanghai Center for Mathematical Sciences}

{\small Shanghai 200433, China }

{\small Email: qingchunji@fudan.edu.cn}

\bigskip

{\small Jun Yao}

{\small School of Mathematics, Fudan University}

{\small Shanghai 200433, China }

{\small Email: jyao\_fd@fudan.edu.cn}
\end{document}